\documentclass[11pt,reqno,a4paper,psamsfonts]{amsart}
\usepackage[utf8]{inputenc}
\usepackage[T1]{fontenc}
\usepackage[english]{babel}
\usepackage{amssymb,amsmath,amscd,enumerate,amsthm}
\usepackage[psamsfonts]{amsfonts}
\usepackage{latexsym}
\usepackage{amsbsy}
\usepackage{ulem}
\usepackage[left=3.cm,right=3.cm,top=3.cm,bottom=3.cm]{geometry}
\usepackage[hidelinks]{hyperref}
\usepackage{todonotes}
\usepackage{frcursive}
\usepackage{pdfsync}
\usepackage{xcolor,hyperref}
\usepackage{relsize,exscale}
\hypersetup{backref,colorlinks=true,linkcolor=blue,citecolor=blue}
\usepackage{pdfsync}
\usepackage{enumerate}
\usepackage{verbatim}
\usepackage{color}
\usepackage[all,color ]{xy}
\usepackage{graphicx}
\newcommand{\point}{\mathbf{\cdot}}

\newcommand{\C}{{\mathbb{C}}}

\newcommand{\F}{{\mathcal{F}}}
\newcommand{\R}{{\mathbb{R}}}%
\newcommand{\Z}{{\mathbb{Z}}}

\let\CAL=\mathcal%
\def\mathcal#1{{\CAL#1}}%
\newcommand{\N}{\mathbb{N}}
\newcommand{\Ve}{{\mathsf{Ve}}}
\newcommand{\Ed}{{\mathsf{Ed}}}
\newcommand{\AG}{{\mathsf{A}}}

\newcommand{\K}{{\mathsf{K}}}

\newcommand{\ag}{\mathsf{a}}

\let\ge\eg

\let\GA\AG
\let\GV\VG
\let\GE\EG

\newcommand{\VSG}{\mathbf{VecG}}
\newcommand{\A}{{\mathsf{A}}}

\let\ltE=\ptE
\newcommand{\AR}{{\mathsf{R}}}

\newcommand{\Cex}{$\mc C^{\mr{ex}}$}

\def\ftrait{\;{{}^{\underline{\phantom{\longleftrightarrow}}}}\;}%
\def\bdf{{\leftarrow\!\mapsto}}%
\newcommand{\G}{\mathcal{G}}

\renewcommand{\emph}[1]{{\textbf{#1}}}
\newcommand{\un}[1]{{\underline{#1}}}

\newcommand{\mf}[1]{{\mathfrak{#1}}}
\newcommand{\mr}[1]{{{\mathrm{#1}}}}
\newcommand{\mc}[1]{{\mathcal{#1}}}
\newcommand{\mb}[1]{{\mathbb{#1}}}
\newcommand{\wt}[1]{{\widetilde{#1}}}

\newcommand{\mbf}[1]{\mathbf{#1}}
\let\msf\msl

\newcommand{\br}[1]{{\breve{#1}}}

\newcommand{\iso}{{\overset{\sim}{\longrightarrow}}}

\newcommand{\DiffQ}{{\mr{Diff}_Q(\C\times Q,(0,u_0))}}
\newcommand{\DiffQzero}{{\mr{Diff}_Q^0(\C\times Q,(0,u_0))}}

\let\DiffQz\DiffQzero
\newcommand{\DD}[1]{\frac{\partial\phantom{ #1}}{\partial #1}}
\newtheorem{teo}{Theorem}[section]

\newtheorem{lema}[teo]{Lemma}

\newtheorem{prop}[teo]{Proposition}
\newtheorem{defin}[teo]{Definition}
\newtheorem{cor}[teo]{Corollary}

\newtheorem{obs2}[teo]{Remark}
\newtheorem{recap2}[teo]{Recapitulation}
\newtheorem{ex2}[teo]{Example}

\newenvironment{obs}{\begin{obs2}\rm}{\hfill\qed\end{obs2}}

\newenvironment{dem2}[1]{\begin{proof}[Proof #1]}{\end{proof}}
\def\bibartp#1#2#3#4#5#6#7#8
{\bibitem{#1} {\sc #2}, {\it #3}, {#4}, {\bf t. #5}, pages { #6} to
{ #7}, {({#8})}}
\def\bibart#1#2#3#4#5#6
{\bibitem{#1} {\sc #2}, {\it #3}, {#4}, {\bf #5}, {({#6})}}
\def\bibliv#1#2#3#4#5
{\bibitem{#1} {\sc #2}, {\it #3}, {#4}, {({#5})}}
\def\bibaart#1#2#3#4
{\bibitem{#1} {\sc #2}, {\it #3}, {#4}}
\definecolor{vert}{rgb}{0,0.46,0}
\usepackage{amsfonts,mathrsfs}
\newcommand{\ms}{\mathscr}

\title[Topological universal families of holomorphic foliations]{Topological moduli space for germs of holomorphic foliations II: universal deformations}
\date{\today}
\author{David Mar\'{\i}n, Jean-Fran\c{c}ois Mattei and \'{E}liane Salem}
\thanks{D. Mar\'{\i}n acknowledges financial support from the Spanish Ministry of Science, Innovation and Universities, through grants MTM2015-66165-P and PGC2018-095998-B-I00 and by the Agency for Management of University and Research Grants of Catalonia through the grant 2017SGR1725.}
\address{Departament de Matem\`{a}tiques\\ Universitat Aut\`{o}noma de Barcelona \\ E-08193 Cerdanyola del Vall\`es (Barcelona)\\ Spain\\ \newline \indent Centre de Recerca Matem\`atica, Campus de Bellaterra, E-08193 Cerdanyola del Vall\`es, Spain} \email{davidmp@mat.uab.es}
\address{Institut de Math\'{e}matiques de Toulouse\\ Universit\'{e} Paul Sabatier\\ 118, Route de Narbonne\\ F-31062 Toulouse Cedex 9, France} \email{jean-francois.mattei@math.univ-toulouse.fr}
\address{Sorbonne Universit\'e, Universit\'e de  Paris, CNRS,  Institut de Math\'ematiques de Jussieu - Paris Rive Gauche, F-75005 Paris, France}
\email{eliane.salem@imj-prg.fr}



\usepackage{tikz}

\newcommand{\Def}{\mathrm{Def}}
\newcommand{\aut}{\mathrm{Aut}}
\newcommand{\fix}{\mathrm{Fix}}
\newcommand{\sym}{\mathrm{Sym}}
\newcommand{\T}{\mathcal{T}}
\newcommand{\fol}{\mathbf{Fol}}
\newcommand{\foltf}{\mathbf{Fol_{ft}}}

\newcommand{\grgr}{\mathbf{GrG}}

\begin{document}

%% Enter full title and short title for running headers
%\title{Topological moduli space for germs of holomorphic foliations II: universal deformations}
%\shorttitle{Topological universal families of holomorphic foliations}
%
%
%
%% Author name(s)
%\author{David Mar\'in\affil{1} and Jean-Fran\c{c}ois Mattei\affil{2} and \'Eliane Salem\affil{3}}
%% Abbreviated author name for running headers
%\abbrevauthor{D. Mar\'in and J-F. Mattei and \'E. Salem}
%% Abbreviated author name for first page header
%\headabbrevauthor{Mar\'in, D. and J-F. Mattei and \'E. Salem}
%
%\address{%
%\affilnum{1}
%Departament de Matem\`{a}tiques, Universitat Aut\`{o}noma de Barcelona,  E-08193 Cerdanyola del Vall\`es (Barcelona), Spain\\ \newline \indent Centre de Recerca Matem\`atica, Campus de Bellaterra, Cerdanyola del Vall\`es, Spain\\
%\affilnum{2} Institut de Math\'{e}matiques de Toulouse, Universit\'{e} Paul Sabatier, 118 Route de Narbonne, F-31062 Toulouse Cedex 9, France\\
%\affilnum{3} Sorbonne Universit\'e, Universit\'e de  Paris, CNRS,  Institut de Math\'ematiques de Jussieu - Paris Rive Gauche, F-75005 Paris, France
%}
%
%% Address / e-mail address of corresponding author
%\correspdetails{davidmp@mat.uab.es}
%
%% Received/revised/accepted dates will be entered by the publisher during production of an accepted paper. Please do not edit these placeholders for submission.
%%\received{1 Month 20XX}
%%\revised{11 Month 20XX}
%%\accepted{21 Month 20XX}
%
%% Enter details of editor communicating this article
%%\communicated{Laura DeMarco}

\begin{abstract}
This work deals with the topological classification of singular foliation germs on $(\C^{2},0)$.  Working in a suitable class of foliations we fix the topological invariants given by the separatrix set, the Cama\-cho\--Sad indices and the projective holonomy representations and we prove the existence of a topological universal deformation through which every equisingular deformation uniquely factorizes up to topological conjugacy. 
This is done by representing the functor of  topological classes of equisingular deformations of a fixed foliation. We also describe the functorial dependence of this representation with respect to the foliation.
\end{abstract}

\maketitle
\tableofcontents

\section{Introduction}\label{introduction}

This work inserts in a series of three  papers  whose goal is to obtain a topological classification of singular foliation germs 
on $(\C^2,0)$ through  the construction of a topological moduli space, the description of its algebraic and topological properties and the construction of a family containing all topological types  with minimal redundancy.\\

In  the article \cite{MM3}, completed by \cite{Loic1} and \cite[Appendix]{MMS},  the authors give for  a generic  germ of foliation $\F$ on $(\C^2,0)$
a list of topological invariants:
\begin{enumerate}[a)]
\item the combinatorial  reduction of singularities of $\F$,  
\item  the Camacho-Sad indices of the singularities of the reduced foliation $\F^\sharp$,
\item the holonomies of $\F^\sharp$ along the invariant  components of the exceptional divisor $\mc E_\F$ of the reduction.
\end{enumerate}
We call this collection of invariants the
\emph{semi-local type} of $\F$ and $\mr{SL}(\F)$  will denote the set of foliations having same semi-local type as $\F$. 
In the present paper we are interested in ``germs of families'' in $\mr{SL}(\F)$ at $\F$, that we call equisingular deformations of $\F$. For any generic foliation we prove the existence of a ``topological universal deformation'' through which   any equisingular deformation of $\F$ uniquely factorizes. We also provide an infinitesimal criterion of universality.
These results will allow us to study in a forecoming paper \cite{MMS3} factorizing properties of the global family constructed in \cite{MMS} that contains all topological types in $\mr{SL}(\F)$.\\

Classically a   \emph{deformation of a foliation}   $\F$ over a germ  of manifold $P^\point=(P,t_0)$ is   a germ of foliation  $\F_{P^\point}$ on $(\C^2\times P, (0,t_0))$ defined by a germ of holomorphic vector field $X(x,y,t)$  that coincides on $\C^2\times\{t_0\}$ with a vector field   defining $\F$ and moreover is  tangent to the fibers of the canonical  projection $\mr{pr}_P : \C^2\times P\to P$.
%
%
%\[ 
%X(x,y,t)=a(x,y,t)\partial_x+b(x,y,t)\partial_y\,,\quad A(x,y,t_0)=a(x,y), \quad B(x,y,t_0)=b(x,y)
% \]
If  $\lambda : (Q,u_0)\to P^\point$ is a germ of holomorphic map, the \emph{pull-back} of $\F_{P^\point}$ by $\lambda$ is the deformation $\lambda^\ast\F_{P^\point}$ of $\F$ over $(Q,u_0)$ defined by the vector field $X(x,y,\lambda(t))$. 
Two deformations $\F_{P^\point}$ and $\F'_{P^\point}$  are \emph{topologically conjugated} if there exists  a $\mc C^0$-automorphism $\Phi$ of $(\C^2\times P, (0,t_0))$  that sends the leaves of $\F_{P^\point}$ on that of $\F'_{P^\point}$,   
 and satisfies
 \[ 
 \mr{pr}_{P}\circ \Phi=\mr{pr}_P\,,\quad\Phi(x,y,t_0)=(x,y,t_0)\,.
  \]
As in 
 \cite{MMS} we say that the deformation  $\F_{P^\point}$ is \textit{equi\-singular}  if the foliations given by the vector fields $X_t(x,y):= X(x,y,t)$ on the fibers $\C^2\times\{t\}$ can  be ``simultaneous  reduced'' and  belong to $\mr{sl}(\F)$, see the precise definition \ref{SL-def}. 
 We will prove:\\

\noindent \textbf{Main Theorem.} \textit{Every finite type generalized curve\footnote{
i.e. a germ of foliation  $\F$ such that  the foliation $\F^\sharp$ obtained after reduction  is without  \textit{saddle-node}  (i.e. singularity given  by a vector field germ whose linear part has exactly one non-zero eigenvalue); however  $\F^\sharp$ may have  \emph{nodal}  singularities (i.e. defined by a vector field germ such that the ratio of the eigenvalues of its linear part is strictly positive) and the exceptional divisor of the reduction may have  irreducible components   non invariant by $\F^\sharp$. For more details we refer to \cite{CCD}.} foliation possesses 
a topological universal deformation.}\\

\noindent \emph{Topological universality} of a deformation $\F_{Q^\point}$  of $\F$ 
means that for any germ of manifold $P^\point$ and any equisingular deformation $\F_{P^\point}$ of $\F$ over $P^\point$,  there exists a unique holomorphic  map germ $\lambda:P^\point\to Q^\point$ such that $\F_{P^\point}$  is topologically conjugated to  $\lambda^\ast\F_{Q^\point}$. In fact we will prove the stronger result that the topological conjugacy between  $\F_{P^\point}$ and  $\lambda^\ast\F_{Q^\point}$ is realized by an \emph{excellent} (or \Cex) homeomorphism, i.e.  it lifts through the equireduction maps  of $\F_{P^\point}$ and   $\lambda^\ast\F_{Q^\point}$ and its lifting fulfills a regularity property, see Definition \ref{excellent}.\\

We obtain a universal deformation of $\F$  by representing the functor  $\Def_\F$ that associates to any germ of manifold $P^\point$, the  set $\Def_{\F}^{P^\point}$ of \Cex-conjugacy classes of deformations of $\F$ over $P^\point$.
To describe the dependence of this representation with respect to $\F$ we define,
up to excellent conjugacy, the pull-back of an equisingular deformation of $\F$ by a \Cex-conjugacy $\phi:\G\to\F$.
We thus get a contravariant \emph{deformation functor}
\[
\Def :\mbf{Man^\point}\times\fol\to \mbf{Set}^\point\,,
\quad
(P^\point, \F)\mapsto \Def_{\F}^{P^\point}\,,
\]
which associates to a  foliation $\F$ and a germ of manifold  $P^\point$, the set $\mr{Def}_\F^{P^\point}$. Here $\mbf{Man^\point}$ is the category of germs of complex manifolds, the morphism sets $\mc O(P^\point,Q^\point)$ consisting
%being the spaces 
of holomorphic map germs compatible with the pointing, and  $\mbf{Fol}$ is the category whose objects are the germs of foliations   which are generalized curves of finite type, the morphisms being \Cex-conjugacies. In fact, we will construct a suitable (pointed by $0$)  cohomogical $\C$-vector space $H^1(\A,\T_{\F})^\point$  associated to $\F$ and an isomorphism of functors
\begin{equation}\label{represIntro}
\mr{Def}\iso \big( (P^\point, \F)\mapsto \mc O(P^\point,H^1(\A,\T_{\F})^\point\big)\,.
\end{equation}
%As immediate corollary, we get the representability of $\Def_\F$.\\
\bigskip

The paper is organized in the following way:\\

- In Chapter~\ref{secAlgTools} we  further develop the key notion of \emph{group-graph} already introduced in \cite{MMS}.  We define  the notion of \emph{regular group-graph} and we describe its cohomology (see Theorem \ref{teobasegeom}).\\

- The notion of equisingular deformation is introduced in Chapter 
\ref{SectEquising}. 
Its characteristic property, stated in Theorem 
 \ref{psiD}, is the triviality  along each irreducible component  of the exceptional  divisor of the equireduction.
 This allows (Theorem~\ref{famconsigma}) 
 to define for a \Cex-conjugacy $\phi:\G\to\F$, the pull-back map $\phi^\ast:\Def_\F^{P^\point}\to\Def_\G^{P^\point}$, and   the functor $\Def$.
 \\
 
 - In Chapter \ref{sectAutSymGrGr} we consider the group-graph $\mr{Aut}_{\F}^{P^\point}$,  over the dual graph $\A_\F$ of $\mc E_\F$, of excellent automorphisms  of the constant deformation of $\F$ over $P^\point$. For an equisingular deformation, the trivializing maps given by Theorem~\ref{psiD} provide a cocycle with values in this group-graph. In this way we obtain a natural transformation from the functor $\mr{Def}$ to the functor that associates to $\F$ and $P^\point$ the cohomology space 
 $H^1(\A_\F, \mr{Aut}_{\F}^{P^\point})$. This transformation is an isomorphism of functors (Theorem \ref{cocycle})
 \begin{equation}\label{isonat}
  \Def \;\iso\;\left(
\big(P^\point,\F) \mapsto H^1(\A_\F, \mr{Aut}_{\F}^{P^\point})
 \right)\,.
 \end{equation}
By taking the quotient of $\mr{Aut}_{\F}^{P^\point}$ by the normal subgroup-graph of automor\-phisms fixing each leaf, we obtain a simpler group-graph  $\mr{Sym}_{\F}^{P^\point}$ with same cohomology as $\mr{Aut}_\F^{P^\point}$ (Proposition~\ref{isoH1AutSym}).\\
 
 - The notion of finite type foliation is defined and  cohomologically characterized (Theorem~\ref{characterization}) in Chapter~\ref{sectFinTypeInfTrSym}. 
For such a foliation  the cohomology of the group-graph $\mr{Sym}_\F^{P^\point}$ over $\A_\F$ is completely given by restricting it to an appropriate subgraph $\AR_\F\subset\A_\F$ (Theorem~\ref{aut-Raut/fix}). 
The advantage of this restriction is that over $\AR_\F$ the group-graph $\mr{Sym}_\F^{P^\point}$ is isomorphic (via the ``exponential morphism'') to the abelian group-graph $\T_\F^{P^\point}$ of $\C$-vector spaces of \emph{infinitesimal transverse symmetries} of the constant deformation, see Definition~\ref{defInfSymTrans}. 
This study gives the natural isomorphisms
\begin{equation}\label{introiso}
 H^1(\A_\F, \mr{Aut}_{\F}^{P^\point})\iso H^1(\A_\F, \mr{Sym}_{\F}^{P^\point})\iso H^1(\AR_\F, \mr{Sym}_{\F}^{P^\point})\iso  H^1(\AR_\F, \T_{\F}^{P^\point})\,.
\end{equation} 
The structure of  $\T_\F^{P^\point}$ over $\AR_\F$ is the tensor product $\T_\F\otimes_\C\mf M_{P^\point}$ of the group-graph of infinitesimal symetries of $\F$ with the maximal ideal of $\mc O_{P^\point}$ (Lemma~\ref{multext}). Finally, using the results of Section~\ref{GGTensorproduct} we get:
\[ 
 H^1(\AR_\F, \T_{\F}^{P^\point})\iso H^1(\AR_\F,  \T_\F\otimes_\C\mf M_{P^\point})
\iso 
H^1(\AR_\F, \T_\F)\otimes_\C\mf M_{P^\point} \iso
\mc O(P^\point, H^1(\AR_\F, \T_\F)^\point)\,,  \]
that achieves, using (\ref{isonat}) and (\ref{introiso}), the construction of the natural isomorphism (\ref{represIntro}).\\

- In Chapter~\ref{secCexUnivDef}, using that the restriction of the group-graph $\T_\F$ to $\AR_\F$  is regular (Proposition \ref{RemquchTF}) and Theorem~\ref{teobasegeom}, we specify in Theorem~\ref{defuniv} the structure of the finite dimensional universal parameter space $ H^1(\AR_\F, \T_\F)^\point$.  We also construct  a  \emph{Kodaira-Spencer map}
\[ 
\left.\frac{\partial[\F_{P^\point}]}{\partial t}\right|_{t=t_0}\,:\,T_{t_0}P\longrightarrow H^1(\AR_\F, \T_\F)
 \] associated to an equisingular  deformation $\F_{P^\point}$, that  will provide in  Theorem~\ref{unicityuniv}   an infinitesimal criterion of universality.

\section{Group-graphs}\label{secAlgTools}

 We recall that a   \emph{graph} is the data of a pair $\A=(\Ve_\A, \Ed_\A)$ where $\Ve_\A$ is a set and $\Ed_\A\subset \mc P(\Ve_\A)$ is a collection of subsets of two distinct elements $v,v'$ of $\Ve_\A$, denoted by $\langle v,v'\rangle$. 
The elements of $\Ve_\A$ are called vertices of $\A$ and those of $\Ed_\A$ are called edges of $\A$.
We denote by
\begin{equation*}\label{orientededges}
I_\A:= \{ (v,e)\in \msf{Ve}_{\A}\times \msf{Ed}_\A \;|\; 
v\in e\}
\end{equation*}
the set of \emph{oriented edges} of $\A$.
A \emph{morphism of graphs} $\varphi:\A'\to\A$ is a map
$\varphi:\Ve_{\A'}\to\Ve_\A$  such that if $e=\langle v,v'\rangle\in\Ed_{\A'}$ either $\varphi(v)\neq\varphi(v')$ and $\varphi(e):=\langle\varphi(v),\varphi(v')\rangle\in\Ed_\A$, or $\varphi(v)=\varphi(v')$ and $\varphi(e):=\varphi(v)\in\Ve_\A$.

\subsection{Notion of group-graph} 
\begin{defin}  Let $\mbf C$ be a category. A \emph{$\mbf C$-graph over  $\GA$} is a collection  $G$ of objects of  $\mbf{C}$, denoted\footnote{The notation $G(v)$ and $G(e)$ is also used in the text.} by $G_{v}$ and $G_{e}$, for each vertex $v\in \GV$ and each edge $e\in \GE$, and of $\mbf{C}$-morphisms $\rho_{v}^{e}:G_{v}\to G_{e}$ for each $(v,e)\in I_{\GA}$, which are called \emph{restriction morphisms}. 
When $\mbf C$ is the category $\mbf{Gr}$ of groups we say that $G$ is a \emph{group-graph};
if  all groups $G_\star$, $\star\in \mathsf{Ve}_\A\cup\mathsf{Ed}_\A$, are abelian, we say that $G$ is abelian and when all groups $G_\star$, $\star\in \mathsf{Ve}_\A\cup\mathsf{Ed}_\A$ are trivial we say that $G$ is the trivial group-graph and we denote it by $0$ or $1$.
\end{defin}

The  \emph{category of $\mbf C$-graphs over $\A$} is the category denoted by $\mbf{C}^\A$,   whose objects are the $\mbf C$-graphs over $\A$ and whose morphisms  $\alpha:F\to G$  are the data of $\mbf C$-morphisms $\alpha_{v}:F_{v}\to G_{v}$ and $\alpha_{e}:F_{e}\to G_{e}$, $v\in \Ve_\A$, $e\in \Ed_\A$,  such that the following diagram 
$$\begin{array}{rcl}
F_{v} & \stackrel{\alpha_{v}}{\longrightarrow} & G_{v}\\
\xi_{v}^{e}\downarrow
& & \downarrow {\rho_{v}^{e}}\\
F_{e} & \stackrel{\alpha_{e}}{\longrightarrow} & G_{e}
\end{array}$$
commutes for each $(v,e)\in I_{\GA}$, $\xi^{e}_V$ and $\rho^{e}_v$ being the restriction maps of $F$ and $G$. \\

In all the sequel we suppose that $\mbf C$ is a subcategory of the category of groups. \\

A $\mbf C$-graph $H$ is a \emph{sub-$\mbf C$-graph} of a $\mbf C$-graph $G$ if  $H_{\star}$ is a subgroup of $G_{\ast}$ for any $\ast\in \Ve_\A\cup\Ed_\A$,  the inclusion map $H_\star\hookrightarrow  G_\star$ being $\mbf C$-morphisms,  and the   restriction maps  $H_v\to  H_e$ being given by  the restriction map $\rho_{v}^{e}$ of $G$, a fortiori $\rho_v^{e}(H_v)\subset H_e$.  When each group $H_\star$ is a normal subgroup of $G_\ast$  we say that $H$ is a \emph{normal sub-$\mbf C$-graph of $G$};  then the map $\rho_v^{e}$
factorizes as  a map $\overline\rho_v^{e}:G_v/H_v\to G_e/H_e$,  defining the \emph{quotient $\mbf C$-graph} $G/H$, with $(G/H)_{\star}=G_\star/H_\star$, the maps $\overline\rho_v^{e}$ being the restriction maps.\\

 If $G$ (resp. $G'$) is a $\mbf{C}$-graph over a graph $\A$ (resp. $\A'$), a \emph{morphism of $\mbf C$-graphs $\phi:G\to G'$ over a morphism of graphs  $\varphi:\A'\to\A$} is a collection of $\mbf{C}$-morphisms 
\[\phi_\star:G_{\varphi(\star)}\to G'_\star,\qquad \star\in\Ve_{\A'}\cup\Ed_{\A'}\]
such that, if $e=\langle v,v'\rangle$ then the following diagram commutes
\[\xymatrix{G_{\varphi(v)}\ar[r]^{\phi_v}\ar[d]_{\rho_{\varphi(v)}^{\varphi(e)}}&G'_v\ar[d]^{\rho{'}_v^e}\\ G_{\varphi(e)}\ar[r]^{\phi_e} & G'_e}\]
If $\varphi(e)=\varphi(v)$ then $\rho_{\varphi(v)}^{\varphi(e)}$ is the identity.
A consequence of the commutativity of this diagram is that
$\rho_v^e$ sends the kernel of $\phi_v$ into the kernel of $\phi_e$
and
$\rho{'}_v^e$ sends the image of $\phi_v$ into the image of $\phi_e$. This allows to define
 the
 \emph{$\mbf{C}$-graph kernel} $\ker\phi$ over $\A$ by $(\ker\phi)_\star=\ker (\phi_\star)$, which is a sub-$\mbf C$-graph of $G$
 and the
 \emph{$\mbf{C}$-graph image} $\phi(G)$ over $\A'$ by $\phi(G)_\star=\phi_\star(G_{\varphi(\star)})$, which is a sub-$\mbf C$-graph of $G'$.
We can thus consider exact sequences of $\mbf{C}$-graphs over a common graph.

If  $\varphi':\A''\to\A'$ is another graph morphism and $\phi':G'\to G''$ is a $\mbf{C}$-graph morphism over $\varphi'$, then the \emph{composition} defined by
\[ 
\phi'\circ\phi:=\{\phi'_\star \circ \phi_{\varphi'(\star)} : G_{\varphi(\varphi'(\star))} \mapsto G''_\star\;\;|\;\;\star\in\Ve_{\A''}\cup\Ed_{\A''}\} 
\]
 is a $\mbf C$-graph morphism $G\to G''$ over $\varphi\circ\varphi'$. 
Hence the 
collection of all the pairs $(\A,G)$ where $\A$ is a graph and $G$ is a $\mbf C$-graph over $\A$ together with the \emph{$\mbf{C}$-graphs morphisms} consisting of the pairs $(\varphi,\phi):(\A,G)\to(\A',G')$ with $\varphi:\A'\to\A$ and $\phi:G\to G'$ over $\varphi$, forms a category that we will denote by $\mbf{CG}$. 
A $\mbf{C}$-graph morphism $(\mr{id}_\A,\phi)$ over the identity of $\A$ is just a morphism of group-graphs over $\A$ as defined previously. Thus, $\mbf{C}^\A$ is a subcategory of $\mbf{CG}$.
\begin{defin}\label{pbgrgr}
The \emph{pull-back by a graph morphism} $\varphi:\A'\to\A$ of a $\mbf{C}$-graph $G$ over $\A$ is the $\mbf{C}$-graph over $\A'$ defined by
\[(\varphi^*G)_\star=G_{\varphi(\star)},\quad \star\in\Ve_{\A'}\cup\Ed_{\A'},\]
 the restriction morphism $(\varphi^*G)_v\to(\varphi^*G)_e$ for $e=\langle v,v'\rangle\in\Ed_{\A'}$  being the restriction morphism $G_{\varphi(v)}\to G_{\varphi(e)}$  when $\varphi(e)\in\Ed_{\A}$, and
 the identity map of $G_{\varphi(v)}$ otherwise.
 We call \emph{canonical morphism} the 
$\mbf{C}$-graph morphism $\imath_{\varphi } : G\to \varphi^\ast G$ over $\varphi$ defined by the identity maps
\[
\imath_{\varphi\,\star}:=\mr{id}_{G_{\varphi(\star)}} : 
G_{\varphi(\star)}\longrightarrow  (\varphi^\ast G)_\star \,,\quad
\star\in \Ve_{\A'}\cup\Ed_{\A'}\,.
\]  
\end{defin}
\noindent  In this way, the data of a morphism of $\mbf{C}$-graphs $\phi:G\to G'$ over a morphism of graphs $\varphi:\A'\to\A$ is just the data of a morphism of $\mbf{C}$-graphs $\br\phi : \varphi^*G\to G'$ over $\A'$.

\begin{obs}\label{factorization}
Let $F:G\to G'$ be a morphism of $\mbf C$-graphs over $f:\AR'\to\A$. 
Let $r:\AR\to\A$ be a morphism of graphs. If $f$ factorizes as $f=r\circ \bar f$ for some morphism of graphs $\bar f:\AR'\to\AR$ then $F$ factorizes as $F=\bar F\circ\imath_r$ where $\bar F:r^*G\to G'$ is a morphism of $\mbf C$-graphs over $\bar f$.
Indeed, if we define $\bar F_\star:=F_\star:(r^*G)_{\bar f(\star)}=G_{r(\bar f(\star))}=G_{f(\star)}\to G'_{\star}$ for each $\star\in\Ve_{\AR'}\cup\Ed_{\AR'}$ then $F=\bar F\circ\imath_r$.
\end{obs}
 
\begin{obs}\label{factgrgr}
If $j=1,2$, let $G_j$ be a group-graph over  $\A_j$ and $K_j$  a normal sub-group-graph of $G_j$, 
then any group-graph morphism $g:G_1\to G_2$ over a graph-morphism $\varphi:\A_2\to\A_1$ sending $ K_1$ to $ K_2$ factorizes as a morphism $\bar g$  between the quotient group-graphs:
\begin{equation*}
\xymatrix{
1\ar[r] & K_1 \ar[r]\ar[d] &  G_1\ar[r]\ar[d]_{g} & G_1/ K_1\ar[r]\ar@{-->}[d] _{\bar g}&1 \\
1\ar[r] & K_2 \ar[r] & G_2\ar[r] & G_2/ K_2\ar[r]& 1
}
\end{equation*}
We easily check this property when $\A_1=\A_2$ and $\varphi=\mr{id}$. Since, by definition $\varphi^*(G_1/K_1)=\varphi^*G_1/\varphi^*K_1$, the general case follows taking the pull-back by $\varphi$ in the first row.
\end{obs}

\begin{obs}\label{graph-cat}
Every graph $\A$ can be seen as a category whose objects are the vertices and the edges of $\A$, and whose morphisms (other than the identities) are the inclusion maps $i^{b}_a: \{a\}\hookrightarrow b$ of a vertex in an edge.         
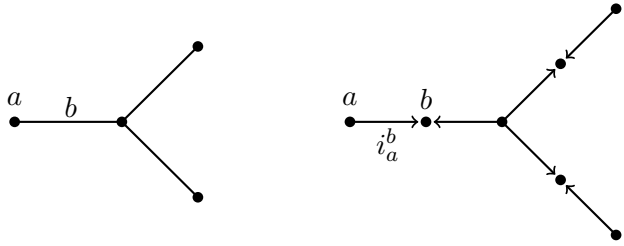
\begin{center}
\begin{tikzpicture}
\draw[thick] (-2.41,0) to (-1,0);
\fill (-2.41,0) circle [radius=2pt];
\fill (-1,0) circle [radius=2pt];
\draw[thick] (-1,0) to (0,1);
\fill (0,1) circle [radius=2pt];
\draw[thick] (-1,0) to (0,-1);
\fill (0,-1) circle [radius=2pt];
\fill (2,0) circle [radius=2pt];
\draw[thick,->] (2,0) to (2.9,0);
\fill (3,0) circle [radius=2pt];
\draw[thick,<-] (3.1,0) to (4,0);
\fill (4,0) circle [radius=2pt];
\draw[thick,->] (4,0) to (4.7,.7);
\fill (4.77,.77) circle [radius=2pt];
\draw[thick,<-] (4.84,0.84) to (5.5,1.5);
\fill (5.5,1.5) circle [radius=2pt];
\draw[thick,->] (4,0) to (4.7,-.7);
\fill (4.77,-.77) circle [radius=2pt];
\draw[thick,<-] (4.84,-0.84) to (5.5,-1.5);
\fill (5.5,-1.5) circle [radius=2pt];
\node at (-2.41,0.3) {$a$};
\node at (-1.67,0.2) {$b$};
\node at (2,0.3) {$a$};
\node at (3,0.3) {$b$};
\node at (2.5,-0.3) {$i_a^b$};
\end{tikzpicture}
\end{center}
A $\mbf C$-graph over $\A$ is just a covariant\footnote{The contravariant version leads to the dual notion of \emph{graph of $\mathbf{C}$}, for instance \emph{graph of groups} in the sense of Serre \cite{Serre}.} functor  $G : \A\to \mbf C$ and  morphisms of $\mbf C$-graphs are just morphisms (i.e. natural transformations) of functors. This explains the adopted notation  $\mbf C^\A=\{F:\A\to\mbf C \text{ covariant functor}\}$. 
Under this identification, a morphism of graphs $\varphi:\A'\to\A$ is a covariant functor between the corresponding categories. If $G\in\mbf C^\A$ then the pull-back
 $\varphi^*G\in\mbf C^{\A'}$ is the composition of functors $G\circ\varphi$
 and 
\[\varphi^*:\mbf{C}^{\A}\to\mbf{C}^{\A'},\quad G\mapsto\varphi^*G=G\circ\varphi\] 
becomes a contravariant functor defining the pull-back by $\varphi$  of a morphism of $\mbf{C}$-graphs   $\alpha:G_1\to G_2$ over $\A$ as the morphism  $\varphi^*\alpha:\varphi^*G_1\to\varphi^*G_2$  of $\mbf{C}$-graphs over $\A'$ given by
  $(\varphi^*\alpha)_\star=\alpha_{\varphi(\star)}$ for $\star\in\Ve_{\A'}\cup\Ed_{\A'}$.
\end{obs}  
In fact, the natural context to consider these notions is that of 
abstract simplicial complexes:
\begin{obs} Recall that
an abstract simplicial complex $\Delta$ is a nonempty subset of $\mc P(S)$ whose elements are called faces, such that for each $F\in\Delta$,  $0<|F|<\infty$ and if $\emptyset\neq F'\subset F$ then $F'\in\Delta$. The dimension of $F\in\Delta$ is $\dim F=|F|-1$, the dimension of $\Delta$ is $\dim\Delta=\sup\{\dim F\,:\,F\in\Delta\}$. A simplicial complex of dimension $\le 1$ is just a graph.
The $k$-skeleton $\Delta_k$ of a simplicial complex $\Delta$ is the subcomplex of $\Delta$ consisting of all faces of dimension at most $k$.
We will identify $\Delta_0$ with the set of vertices $\bigcup\limits_{F\in\Delta}F\subset S$  of~$\Delta$.
Each simplicial complex $\Delta$ can be thought of as a small category  whose objects are the elements of $\Delta$ and whose morphisms are the inclusions, i.e. if $F\subset F'\in\Delta$ then $\mr{Hom}_\Delta(F,F')=\{i_{FF'}:F\hookrightarrow F'\}$.

A simplicial map between (abstract) simplicial complexes $f:\Delta\to\Gamma$ is defined by a map $f_0:\Delta_0\to\Gamma_0$ such that $f(F):=f_0(F)\in\Gamma$ for all $F\in\Delta$. Any simplicial map $f:\Delta\to\Gamma$ can be thought of as a functor. 

The category $\mathbf{SC}$  of simplicial complexes and simplicial maps contains the full subcategory  $\mathbf{SC}_k$ of simplicial complexes of dimension $\le k$. In particular $\mathbf G:=\mathbf{SC}_1$ is the category of graphs. If  $\Delta$ is a graph then  $\Delta=\Delta_1$ and $\Delta_1\setminus\Delta_0$ is the set of edges.
Passing to the $k$-skeleton defines a functor $\mathbf{SC}\to\mathbf{SC}_k$.
For every category $\mathbf C$ we consider the collection $\mbf{CSC}$ of $\mathbf C$-simplicial complexes which are  pairs $(\Delta,G)$ with $\Delta$  a simplicial complex and 
$G\in\mathbf C^\Delta:=\{\Delta\to \mathbf C \text{ covariant functor}\}$,
i.e. $G$ is an assignment $\Delta\ni F\mapsto G(F)$ jointly with a  $\mathbf C$-morphism $\rho^G_{FF'}:G(F)\to G(F')$, that we call restriction, if $F\subset F'\in\Delta$.  We will say that $G$ is a $\mathbf{C}$-simplicial complex over $\Delta$. There is a natural definition of morphism of $\mbf C$-simplicial complexes over a map of simplicial complexes completely analogous to the one considered for $\mbf C$-graphs which makes  $\mbf{CSC}$ a category.
\end{obs}

\subsection{Group-graph associated to a sheaf}\label{gg-sh}

Let $\un{\mc S}$ be a $\mbf C$-sheaf on a topological space $\mc D$ and $\mc C$  a collection of sets of $\mc D$. Consider the following  graph  $\A$ (not necessarily  finite): its vertices are the elements  of $\mc C$ and its edges are all the sets $\langle D,D'\rangle$ formed by two distinct  elements of $\mc C$, such that $D\cap D'\neq \emptyset$. 
For any $W\subset\mc D$ (not necessarily open) we recall that the group of continuous sections of $\un{\mc S}$ over $W$ is
$\un{\mc S}(W):=\lim\limits_{\stackrel{\longrightarrow}{U\in\mathscr U_W}}\un{\mc S}(U)$, 
where $\mathscr U_W$ is the set of open neighborhoods of $W$. In the case that $W=\{p\}$, $\un{\mc S}(\{p\})$ is just the stalk $\un{\mc S}(p)$ of $\un{\mc S}$ at $p\in\mc D$. If $W'\subset W$ then $\mathscr U_{W}\subset\mathscr U_{W'}$ and the inductive limit of the restriction morphisms of $\un{\mc S}$ define a \emph{restriction morphism} $\un{\mc S}(W)\to\un{\mc S}(W')$.

We define the  \emph{$\mbf C$-graph $\mc S$ over $\A$ associated to  $\underline{\mathcal S}$}  in the following way: 
\begin{itemize}
\item ${\mc S}_D:=\un{\mc S}(D)$  for $D\in \mathsf{Ve}_{\A}$, 
\item[$\bullet$]  $\mc S_{\langle D,D'\rangle}:=\un{\mc S}(D\cap D')$  for $\langle D,D'\rangle\in \msf{Ed}_\A$,
\item the restriction maps $\rho_D^{\langle D,D'\rangle}$ are the restriction morphisms considered before.
\end{itemize}
Any morphism of $\mbf C$-sheaves over $\mc D$ induces a morphism of $\mbf C$-graphs over $\A$, defining  a covariant functor:
\begin{equation*}
\mbf {C Sh}_{\mc D}\longrightarrow \mbf {C}^\A\,,\quad
\un{\mc S}\mapsto \mc S\,,
\end{equation*}
from the category of $\mbf{C}$-sheaves over $\mc D$ to the category of $\mbf{C}$-graphs over $\A$.
We highlight that this functor is not exact in general.

Let $\mc D'$ be another topological space with a collection 
$\mc C'$ of subsets of $\mc D'$ and let $\un{\mc S}'$ be a $\mbf{C}$-sheaf over $\mc D'$.
Let $\phi:\mc D'\to\mc D$ be a homeomorphism such that 
$\phi(\mc C)=\mc C'$.
If $D\in\Ve_{\A'}$ and $\langle D,D'\rangle\in\Ed_{\A'}$ then
$\phi(D)\in\Ve_{\A}$, $\phi(D\cap D')=\phi(D)\cap\phi(D')$ and
$\phi$ induces a graph morphism 
\begin{equation*}\label{inducedgrmor}
\A_\phi:\A'\to \A
 \,,\quad \star\mapsto \phi(\star)\;;\quad \star\in\Ve_{\A'}\cup\Ed_{\A'}\,.
\end{equation*}
Given a morphism of $\mbf C$-sheaves $\un{\mc S}\to\un{\mc S}'$ over $\phi:\mc D\to\mc D'$, i.e. a morphism 
\[\un g:\phi^{-1}\un{\mc S}\to\un{\mc S}'\]
of $\mbf C$-sheaves over $\mc D'$,
we have $\mbf C$-morphisms
  \[
\un g_{D}:(\phi^{-1}\un{\mc S})(D)=\un{\mc S}(\phi(D))\to \un{\mc S}'(D)\,,
   \]
  \[ 
 \un g_{D\cap D'} : (\phi^{-1}\un{\mc S})(D\cap D')=\un{\mc S}(\phi(D\cap D'))=\un{\mc S}(\phi(D)\cap\phi(D'))\to \un{\mc S}'(D\cap D')\,,
   \]
for $D\in\Ve_{\A'}$ and $\langle D,D'\rangle\in\Ed_{\A'}$.
Since
\[(\A_\phi^*\mc S)(\langle D,D'\rangle)=\mc S(\langle \phi(D),\phi(D')\rangle)=\un{\mc S}(\phi(D)\cap \phi(D'))\]
we obtain a $\mbf C$-graph morphism \emph{associated to the sheaf morphism $\un g$}
\[g:\A_\phi^*\mc S\to \mc S'\,.\]
Notice that $\A_\phi^*\mc S$ coincides with the $\mbf C$-graph associated to the sheaf $\phi^{-1}\un{\mc S}$ over $\mc D'$, and $g$ can be seen as the $\mbf C$-graph morphism associated to the morphism of sheaves $\un g:\phi^{-1}\un{\mc S}\to\un{\mc S}'$.

The situation we will deal with in the sequel is the following: $\mc D$ is an analytic set (and more specifically a hypersurface in a complex manifold), $\mc C$ is the collection of irreducible components of~$\mc D$.
The graph $\A$ is called  the \emph{dual graph} of $\mc D$.
In this way we have a functor
\[\mbf{CSh}_\mr{an}\to\mbf{CG},\quad \un{\mc S}\mapsto\mc S\,,\]
where $\mbf{CSh}_\mr{an}$ is the subcategory of the category of $\mbf C$-sheaves over analytic sets whose morphisms are over homeomorphisms.

\subsection{Cohomology of a group-graph}\label{cohomology} This notion was introduced in \cite{MMS}. For group-graphs associated to sheaves considered in 
subsection \ref{gg-sh}, with $\mc C$ a locally finite open covering $\mc U$ of $\mc D$ and $\un{\mc S}$ abelian,
this notion will coincide with the \v Cech cohomology groups $\check{H}^i(\mc U,\un{\mc S})$, $i=0,1$.
\\

Let $G$ be a group-graph over a graph $\A$.
The $0$-cohomology set is the subgroup  $H^{0}(\GA,G)$ of $C^{0}(\GA,G):=\prod_{v\in \mathsf{Ve}_\A}G_v$ whose elements are the families  $(g_{v})$ satisfying the relations
$\rho_{v}^{e}(g_{v})=\rho_{v'}^{e}(g_{v'})$ whenever $e=\langle v,v'\rangle$.\\

In order to define the $1$-cohomology set $H^1(\A, G)$ of a group-graph $(G, (\rho^e_v)_{(e,v)\in I_\A})$ we first define the set of cocycles $Z^{1}(\GA,G)$  as the set of families 
\[
(g_{v,e})\in\prod_{(v,e)\in I_{\GA}}G_{v,e}\,,\quad \hbox{ with }\quad G_{v,e}:=G_e\,,
\]
such that $ g_{v,e}g_{v',e}=1$ whenever $e=\langle v,v'\rangle$. 
Then $H^1(\A, G)$ is the quotient set of $Z^1(\A, G)$ by the following action of $C^0(\A, G)$:
\[(g_{v})\,\star_G\,(g_{v,e}):=\left(\rho_{v}^{e}(g_{v})^{-1}\,g_{v,e}\,\rho_{v'}^{e}(g_{v'})\right)\,.\]
The set $H^1(\A,G)$ contains the privileged element $1$ defined by $g_{v,e}=1$. In this way, from now on $H^1(\A,G)$ will be consider as a pointed set.
\begin{obs}\label{abelian-case}
When $G$ is an abelian group-graph,  then  $H^1(\A, G)$ is an abelian group. Specifically, we have in this case an exact sequence of groups (with additive notations)
\begin{equation*}\label{exact-sequ-abelian}
C^0(\A,G)\stackrel{\partial^0}{\longrightarrow} Z^1(\A,G)\to H^1(\A,G)\to 0\,,
\end{equation*}
\[\partial^0((g_v))=(g_{v,e})\,,\quad g_{v,e}:= g_{v'} - g_v\,,\quad e=\langle v,v'\rangle\,.\]
More formally, $H^i(\A,G)$ is the $i$-th cohomology group of the cochain complex of abelian groups 
\[C^*(\A,G):\quad C^0(\A,G)\stackrel{\partial^0}{\to} C^1(\A,G)\stackrel{\partial^1}{\to} C^2(\A,G):=\prod_{e\in\Ed_\A}G_e\;,\]
with: $\partial^1((g_{v,e}))=(g_{v,e}+g_{v',e})$, if $e=\langle v,v'\rangle$.
\end{obs}
Every morphism $\phi:G\to G'$ of $\mbf{C}$-graphs over a graph morphism $\varphi:\A'\to\A$ induces maps
\[\phi_0:C^0(\A,G)\to C^0(\A',G'),\quad \phi_0((g_v)_v)= (\phi_{v'}(g_{\varphi(v')}))_{v'}\,,\]
\[\phi_1:C^1(\A,G)\to C^1(\A',G'),\quad \phi_1((g_{v,e}))=(g'_{v',e'}),\]
where
\[g'_{v',e'}=\left\{
\begin{array}{ll}
\phi_{e'}(g_{\varphi(v'),\varphi(e')}) & \text{if $\varphi(e')$ is an edge of $\A$,}\\
1 & \text{otherwise.}
\end{array}\right.\]
The image of the restriction $H^0(\phi)$ of the group morphism $\phi_0$ to the subgroup  $H^0(\A,G)$ is contained in $H^0(\A',G')$.
Moreover,  $\phi_1$ sends $Z^1(\A,G)$ into $Z^1(\A',G')$, the following  
 diagram is commutative
\[\xymatrix{C^0(\A,G)\times Z^1(\A,G)\ar@<-5ex>[d]^{\phi_0}\ar@<5ex>[d]^{\phi_1}
\ar[r]^{\hspace{9mm}\star_G}&Z^1(\A,G)\ar[d]^{\phi_1}\\ C^0(\A',G')\times Z^1(\A',G')\ar[r]^{\hspace{10mm}\star_{G'}}  & Z^1(\A',G')}\] 
inducing a map
\begin{equation}\label{H1phi}
H^1(\phi):H^1(\A,G)\to H^1(\A',G')\,.
\end{equation}
In this way one can check that the correspondences $(\A,G)\mapsto H^i(\A, G)$ and $(\varphi,\phi)\mapsto H^i(\phi)$ define covariant functors 
\begin{equation}\label{defH1}
H^{i} : 
\mbf{CG}\to \mbf{Set}^\point\,,\quad  i=0,1\,,
\end{equation}
from the category of $\mbf{C}$-graphs to the category of pointed sets.
Moreover when 
$\mbf C$ is one of the following sub-categories of $\mbf{Gr}$:
\begin{itemize}
\item the category $\mbf{Ab}$ of abelian groups,
\item the category $\mbf{Vec}$ of $\C$-vector spaces, and linear maps,
\end{itemize}
we obtain covariant functors with values in the same category pointed by $0$:
\begin{equation*}
H^{i}: 
\mbf{CG}\to \mbf{C}^\point\,, \quad i=0,1\,.
\end{equation*}
In particular, $H^i(\phi)$, $i=0,1$, are $\mbf C$-morphisms.

\begin{obs}
The canonical morphism $\imath_{\varphi}: G\to\varphi^\ast G$ induces  maps $H^{i}(\imath_{\varphi}):H^{i}(\A,G)\to H^{i}(\A',\varphi^\ast G)$ and we have
\[ 
H^{i}(\phi)=H^{i}(\br\phi)\circ H^{i}(\imath_{\varphi})\,,\quad  i=0,1,
 \]
 where $\br\phi:\varphi^\ast G\to G'$ is the  $\mbf{C}$-graph morphism over $\A'$ associated to $\phi$ and $\imath_\varphi : G\to \varphi^\ast G$ is the canonical morphism.
\end{obs}

\begin{prop}\label{quoti}
Let $1\to G'\stackrel{i}{\hookrightarrow} G \stackrel{p}{\to} G''\to 1$ be a short exact sequence of group-graphs over a \emph{tree} $\A$ and suppose that all  restriction maps $\rho'{}^e_v :G'_v\twoheadrightarrow G_e'$ are  surjective.  Then the induced morphism $H^1(p):H^1(\A,G)\to H^1(\A,G'')$ is an isomorphism.
\end{prop}

\begin{proof}
First we define an \emph{orientation} $\prec$  of each edge of $\A$ in the following way: we choose a vertex $v_0\in \mathsf{Ve}_\A$; as $\A$ is a tree, for each vertex $v\in \mathsf{Ve}_\A$ there is a unique \emph{geodesic in $\mathsf{Ve}_\A$ joining $v$ to $v_0$}, i.e. a unique minimal sequence of vertices $v_0,\ldots, v_\ell$, such that $v_\ell=v$ and $\langle v_{i-1}, v_i\rangle$, $i=1,\ldots \ell$, are edges of $\A$; then we set $v_{i-1}\prec v_i$. Notice that for any vertex $v\neq v_0$ there is only one edge $\langle v', v\rangle$ such that $v'\prec v$.

The surjectivity of $p_\ast:=H^1(p)$ follows from that of $p$. Indeed 
for any $(h_{v,e})\in Z^1(\A,G'')$ and each edge $e=\langle v',v''\rangle$ with $v'\prec v''$, we can choose an element $g_{v',e}\in G_{v',e}=G_e$ such that $p_e(g_{v',e})=h_{v',e}$. Setting $g_{v'',e}:=g_{v',e}^{-1}$ we obtain an  element $(g_{v,e})$ of $Z^1(\A,G)$ satisfying   $p_\ast([(g_{v,e})])=[(h_{v,e})])$.

To prove the injectivity of $p_\ast$ let us  consider two cohomological classes $
[(g_{v,e})]$ and $[(h_{v,e})] \in H^{1}({\A},G)
$
 such that $p_\ast([(g_{v,e})])=p_\ast( [(h_{v,e})] )$. The cocycles 
 $(p_{e}(g_{v,e}))$ and $(p_{e}(h_{v,e}))$ being cohomologous, there exists $(g''_v)\in C^0(\A, G'') = \prod_{v\in \mathsf{Ve}_\A} G''_v$ satisfying the following equalities in $G_e''$, for any $e=\langle v, w\rangle\in \mathsf{Ed}_\A$, $v\prec w$: 
 \begin{equation*}    
\rho''^e_v (g''_{v})^{-1}\, p_e(g_{v,e}) \,\rho''^e_w(g''_w) =p_e(h_{v,e})\,.
 \end{equation*}
 By surjectivity of $p_v: G_v\to G''_v$, $v\in \mathsf{Ve}_\A$, there are $g_v\in G_{v}$ such that $g''_v=p_v(g_v)$ and, thanks to the commutative  diagrams
\begin{equation}\label{comdiag}
\xymatrix{G'_v\ar@{->>}[d]_{\rho'^e_v}\ar@{^{(}->}[r]^{i_v}&G_v\ar[d]_{\rho_v^e}\ar[r]^{p_v}&G_v''\ar[d]^{\rho_{v}''^e}\\ G_e'\ar@{^{(}->}[r]^{i_e}&G_e\ar[r]^{p_e}&G_e''}
\end{equation}
for any $e=\langle v,w\rangle$, 
we obtain the equalities in $G_e$
 \begin{equation*}
p_e\left(\rho^e_v (g_{v})^{-1}g_{v,e} \rho^e_w(g_w) \right) =p_e(h_{v,e})\,.
 \end{equation*}
 Therefore there exists $g'_e\in G'_e$  such that 
 \begin{equation*}
  (\star_e)\qquad
 \rho^e_v (g_{v})^{-1}g_{v,e} \rho^e_w(g_w) i_e(g'_e) = h_{v,e}\,.
 \phantom{ (\star_e)\qquad}
 \end{equation*}
 We will  construct a cocycle $(k_{v})\in \prod_{v\in \mathsf{Ve}_\A} G_v$ that satisfies the equality
 \[
 (\star\star_e)\qquad\quad
  \rho^e_v (k_{v})^{-1}g_{v,e} \rho^e_w(k_w)  = h_{v,e}\,.\phantom{ (\star\star_e)\qquad}
 \]
 for each edge $e=\langle v,w\rangle$, $v\prec w$, of $\A$, using an induction process indexed by the lengths $\ell$ of the geodesics $v_0,\ldots,v_\ell=v$  joining in $\A$ any vertex  $v\in \mathsf{Ve}_\A$ to the previously chosen vertex $v_0$. One call $\ell$ the \emph{distance of $v$ to $v_0$} 
 and we denote $\ell=d_{\A}(v,v_0)$. 
 Consider the following assertion:
 \begin{enumerate}
 \item[$(H_n)$ ] there exists $(k_{v})\in \prod_{v\in \mathsf{Ve}_A, d_\A(v,v_0)\leq n}G_v$ such that:
 \begin{enumerate}
 \item[$(\alpha_n)$] the relations $(\star\star_e)$ are fulfilled for every edge $e=\langle v,w \rangle$, $v\prec w$,  with 
 $d_\A(v,v_0)$ and  $d_\A(w, v_0)\leq n$,
 \item[$(\beta_n)$]  for every $v\in \mathsf{Ve}_\A$, $1\leq d_\A(v,v_0)\leq n$,  there exists $f'_v\in G'_v$ such that $k_v=g_vi_v(f'_v)$.
 \end{enumerate}
 \end{enumerate}
 We will prove in a) that assertion $H_1$ is true, and in b) that assertion $H_{n+1}$ is true as soon as  assertion $H_{n}$ is satisfied.
 
 \indent - a) Let us consider the relation $(\star_e)$  for each edge $e=\langle v,w\rangle$,  with $v=v_0$. The restriction maps ${\rho'}^{e}_w:G'_w\to G'_e$ being surjective, we choose $g'_w\in G'_w$ such that $g'_e=\rho'{}^e_w(g'_w)$. Using again the commutativity of all diagrams (\ref{comdiag}) we deduce the equality
 \[
 \rho^{e}_{v_0} (g_{v_0})^{-1}g_{v_0,e} \rho^e_w(g_w i_w(g'_w)) = h_{v_0,e}\,.
 \]
 Setting  
 $k_{v_0}=  g_{v_0}$, 
 $k_w=g_wi_w(g'_w)$ and $f'_w=g'_w$,  
 we obtain the assertion $H_1$.
 
 \indent - b)  Now let us suppose $H_n$ satisfied, we will prove $H_{n+1}$. Let us fix families
 \[
 (g_v)\in \prod_{v\in \mathsf{Ve}_\A}G_v\quad \hbox{ and } \quad (g'_e)\in \prod_{e\in \mathsf{Ed}_\A} G'_e
 \]
fulfilling the relation $(\star_e)$ for every $e\in \mathsf{Ed}_\A$. Let us fix also a collection
\[
(f'_v)\in  \prod_{v\in \mathsf{Ve}_\A,\, d_\A(v,v_0)\leq n}G'_v
\]
such that the elements  
\begin{equation}\label{kvgv}
k_v:=g_v i_v(f'_v)\in G_v\,,\quad v\in \mathsf{Ve}_\A\,,\quad d_\A(v,v_0)\leq n\,,
\end{equation}
satisfy the relation $(\star\star_e)$ for every edge $e$ of $\A$ whose  vertices are at  distances to $v_0$ at most $n$.   
Let $w$ be a  vertex of $\A$ such that $d_{\A}(w,v_0)=n+1$. 
As noticed above, there is a unique  edge  $e_w=\langle v_w,w\rangle$ of $\A$ with $v_w\prec w$. Therefore $v_w$ is the unique vertex of $\A$ such that $d_\A(v_w, w)=n$ and $\langle v_w,w\rangle$ is an edge of $\A$.
The relations $(\star_{e_w})$ and (\ref{kvgv}) give the equality: 
\[
\rho^{e_w}_{v_w}(i_{v_w}(f'_{v_w}))\, \rho^{e_w}_{v_w}(k_{v_w})^{-1}\, g_{v_w, e_w}\, \rho^{e_w}_{w}(g_w)\, i_{e_w}(g'_{e_w})\;
 =\;
  h_{v_{w}, e_{w}}\,.
\] 
As in step a), let $g'_w\in G'_w$ such that $\rho'{}^{e_w}_w(g'_w)=g'_{e_w}$. We have: 
\[
 i_{e_w}(g'_{e_w}) = i_{e_w}(\rho'{}^{e_w}_w(g'_w))= \rho^{e_w}_w(i_{w}(g'_w))\,,
\]
thus
\[
\rho^{e_w}_{v_w}(i_{v_w}(f'_{v_w}))\, \rho^{e_w}_{v_w}(k_{v_w})^{-1}\, g_{v_w, e_w}\, \rho^{e_w}_{w}(g_w i_{w}(g'_w))= h_{v_w,e_w}\,.
\;
\]
On the other hand the element $\rho^{e_w}_{v_w}(i_{v_w}(f'_{v_w}))=i_{e_w}(\rho'{}^{e_w}_{v_w}(f'_{v_w}))\in G_{e_w}$ belongs to the normal subgroup of $G_{e_w}$
\[
\ker(p_{e_w})= i_{e_w}(G'_{e_w})=  i_{e_w}(\rho'{}^{e_w}_{w} (G'_{w}))=\rho^{e_w}_{w}(i_{w}(G'_{w})) \,.
\]
The following element of $G_{e_w}$:
\[
\wt g_{e_w}:= g^{-1}\, \rho^{e_w}_{v_w}(i_{v_w}(f'_{v_w}))\, g\,,
 \quad
g:= \rho^{e_w}_{v_w}(k_{v_w})^{-1}\, g_{v_w, e_w}\, \rho^{e_w}_{w}(g_w i_{w}(g'_w))\,,
\]  
is also an element of $\ker(p_{e_w})$. There exists $\wt g'_{w}\in G'_{w}$ such that 
\[\wt g_{e_w}=\rho^{e_w}_{w}(i_{w}(\wt g'_{w}))\,.\]
We finally obtain:
\[
 \rho^{e_w}_{v_w}(k_{v_w})^{-1}\, g_{v_w, e_w}\, \rho^{e_w}_{w}(g_w i_{w}(g'_w))\,
 \rho^{e_w}_{w}(i_{w}(\wt g'_{w})) = h_{e_w, v_w}\,,
\]
and
\[
 \rho^{e_w}_{v_w}(k_{v_w})^{-1}\, g_{v_w, e_w}\, \rho^{e_w}_{w}(g_w i_{w}(g'_w\,\wt g'_{w})) = h_{e_w, v_w}\,. 
\]
We set 
\[
k_w:=g_w i_{w}(g'_w\,\wt g'_{w})\in G_w\,,\quad f'_{w}:= g'_w\,\wt g'_{w}
\]
and we repeat this construction for each vertex whose distance to $v_0$ is $n+1$. The   family $(k_v)$, $v\in \mathsf{Ve}_\A$, $d_{\A}(v,v_0)\leq n+1$, that we obtain  satisfies assertion~$H_{n+1}$.
\end{proof}

\subsection{Pruning}\label{Subsecfunctpruning}

A \emph{path in a tree $\A$ with origin $c_0$ and extremity $c_\ell$} is a sequence $L=(c_0,\ldots,c_\ell)$, $c_j\in \Ve_\A\cup\Ed_\A$ such that: 
\begin{itemize}
\item
if $c_j$, $j<\ell$, is a vertex, then $c_{j+1}$ is %either 
an edge and $c_j\in c_{j+1}$, % or $c_{j+1}=c_j$, 
\item if $c_j$, $j<\ell$,  is an edge,  then $c_{j+1}$ is %either 
a vertex and $c_j\ni c_{j+1}$. % or $c_{j+1}=c_j$.
\end{itemize}
If $\AR$ is a sub-tree of a $\A$ we can define for any vertex $v$ of $\A\setminus\AR$ the notion of  \emph{geodesic in $\A$ from $v$ to $\AR$}, as the unique minimal path $L_v=(c_0,\ldots, c_\ell)$  in $\mathsf{Ve}_\A\cup\mathsf{Ed}_\A$ such that $c_0=v$, $c_\ell\in\mathsf{Ve}_{\AR}$  and $c_{\ell-1},\ldots, c_0\notin \mathsf{Ve}_{\AR}\cup\mathsf{Ed}_{\AR}$. When $v$ is a vertex of $\AR$, the geodesic  $L_v$ is reduced to the single element $v$. We define a \emph{partial order relation} on $\mathsf{Ve}_\A$ by setting $v\prec_{\AR} w$ if and only if the geodesic $L_v$ is contained in the geodesic $L_w$.
We will say that $\AR$ is \emph{repulsive for a group-graph $G$} over $\A$, if for every edge $e=\langle v,v'\rangle\in\Ed_\A$ with $v\prec_{\AR} v'$, the restriction map $\rho_{v'}^e:G_{v'}\to G_e$ is surjective. From \cite[Theorem~3.11 and Remark~3.12]{MMS} we have:

\begin{teo}\label{pruning}
Let $\AR$ be a subtree of a tree $\A$ that is  repulsive for a $\mbf{C}$-graph  $G$ over~$\A$. Then the map 
\[ 
 H^1(\imath_{r})\,:\, H^1(\A, G)\to H^1(\AR,r^\ast G)\,,\quad
(g_{v,e})_{v\in e\in\Ed_\A}\mapsto (g_{v,e})_{v\in e\in\Ed_{\AR}}
 \]
induced by the canonical $\mbf C$-graph morphism $\imath_r : G\to r^\ast G$ over the inclusion  graph morphism $r:\AR\hookrightarrow \A$, is a bijection of pointed sets.
Moreover, if $\mbf C=\mbf{Ab}$ or $\mbf C=\mbf{Vec}$ then $H^1(\imath_r)$ is a $\mbf C$-isomorphism.
\end{teo}

\subsection{Direct image of a $\mathbf C$-graph}
Let $\varphi:\A\to\A'$ be a morphism of graphs and let $G$ be a $\mathbf C$-graph over $\A$. We define the direct image of $G$ by $\varphi$ as the $\mathbf C$-graph $\varphi_*G$ over $\A'$ given for $v'\in\Ve_{\A'}$ and $e'\in\Ed_{\A'}$ by
\[(\varphi_*G)_{v'}:=H^0(\varphi^{-1}(v'),G)\subset\prod_{\varphi(v)=v'} G_v,\quad (\varphi_*G)_{e'}:=\prod_{\varphi(e)=e'}G_e\]
and $(\varphi_*\rho)_{v'}^{e'}((g_v)_v):=(\rho_v^e(g_v))_e$,
where $v\in\Ve_{\A}$ and $e\in \Ed_\A$. It is implicitely understood that the product over the empty set is the trivial group.

There is a canonical morphism  $j_\varphi:\varphi_*G\to G$ of $\mathbf C$-graphs over $\varphi$ defined by the natural projections $(j_\varphi)_\star:(\varphi_*G)_{\varphi(\star)}\subset\prod\limits_{\varphi(\bullet)=\varphi(\star)}
G_\bullet\to G_\star$ for every $\star\in\Ve_\A\cup\Ed_\A$. It can be checked that if $G'$ is a $\mathbf C$-graph over $\A'$ then the maps 
\[\mr{Hom}_\A(G,\varphi_*G')\stackrel{a}{\longrightarrow}\mr{Hom}_\varphi(G,G')\stackrel{b}{\longleftarrow}\mr{Hom}_{\A'}(\varphi^*G,G')\]
given by $a(\phi)=j_\varphi\circ \phi$ and $b(\br\phi)=\br\phi\circ\iota_\varphi$ are bijective.\\ 

\noindent The preimage $\varphi^{-1}(v')$ of a vertex $v'\in\Ve_{\A'}$ by a graph morphism $\varphi:\A\to\A'$ is always a subgraph of $\A$. If $G$ is a group-graph over $\A$ we will denote by $H^1(\varphi^{-1}(v'),G)$ the 1-cohomology set of the pull-back of $G$ by the inclusion map $\varphi^{-1}(v')\hookrightarrow\A$.
\begin{lema}\label{bij} 
Let $\varphi:\A\to \A'$ be a morphism of graphs, let $G$ be a group-graph over $\A$ and consider the map $H^1(j_\varphi):H^1(\A',\varphi_*G)\to H^1(\A,G)$
defined in (\ref{H1phi}).
\begin{enumerate}[(a)]
\item The image of $H^1(j_{\varphi})$ is the set of cohomology classes of $1$-cocycles  $(h_e)_e\in Z^1(\A,G)$
with $h_e=1$ if $\varphi(e)\in\Ve_{\A'}$.
\item
$H^1(j_\varphi):H^1(\A',\varphi_*G)\to H^1(\A,G)$ is always  injective. 
\item
If $H^1(\varphi^{-1}(v'),G)=1$ for all $v'\in\Ve_{\A'}$, then $H^1(j_\varphi):H^1(\A',\varphi_*G)\to H^1(\A,G)$ is surjective.
\end{enumerate}
\end{lema}
\begin{proof}
By fixing an orientation for each edge of $\A$ and $\A'$ we have bijections
\[Z^1(\A',\varphi_*G)\simeq\prod\limits_{e'\in\Ed_{\A'}}(\varphi_*G)_{e'}\quad\text{and}\quad Z^1(\A,G)\simeq\prod\limits_{e\in\Ed_\A}G_e.\] 
Under these identifications
the map $H^1(j_\varphi)$ is induced by 
\[j_\varphi^1:Z^1(\A',\varphi_*G)\simeq\prod_{e'\in\Ed_{\A'}}(\varphi_*G)_{e'}=\prod_{e'\in\Ed_{\A'}}\prod_{\varphi(e)=e'}G_e\to\prod_{e\in\Ed_\A} G_e\simeq Z^1(\A,G)\] 
which puts $1$ in the factor $G_e$ when $\varphi(e)\notin\Ed_{\A'}$, this proves assertion (a). 
To prove assertion
(b) let us fix $(g_{e'})_{e'},(h_{e'})_{e'}\in Z^1(\A',\varphi_*G)$ and $(k_v)_v\in C^0(\A,G)$ safisfying $(k_v)\star_G j_\varphi^1(g_{e'})=j^1_\varphi(h_{e'})$ in $Z^1(\A,G)$. For any $v'\in\Ve_{\A'}$ we check that 
 \[k_{v'}:=(k_v)_{v\in\varphi^{-1}(v')\cap\Ve_\A}\in H^0(\varphi^{-1}(v'),G).\]
 Then $k_{v'}\in (\varphi_*G)_{v'}$ and
$(k_{v'})_{v'}\in C^0(\A',\varphi_*G)$  satisfies $(k_{v'})\star_{\varphi_*G}(g_{e'})=(h_{e'})$ in $Z^1(\A',\varphi_*G)$.
To prove assertion (c),  let us fix a $1$-cocycle $(g_e)_e\in Z^1(\A,G)$.
Since $H^1(\varphi^{-1}(v'),G)=1$  for
 each $v'\in\Ve_{\A'}$ 
 there is $(k_v)_{v\in\varphi^{-1}(v')\cap\Ve_\A}\in C^0(\varphi^{-1}(v'),G)$ such that $(k_v)\star(g_e)=1$ in $Z^1(\varphi^{-1}(v'),G)$. 
Then $(k_v)_{v\in\Ve_\A}\in C^0(\A,G)$ satisfies $(k_v)\star (g_e)=(h_e)$ with $h_e=1$ if $\varphi(e)\in\Ve_{\A'}$. We conclude using assertion (a).
\end{proof}

\subsection{Regular group-graph}

%\subsection{Characteristic group-graphs}\label{charggrgr}
The \emph{support} of a group-graph $G$ over a graph $\msf{A}$ is the set of vertices and edges where the corresponding group is non-trivial:
\begin{equation}\label{defsuppgrgr}
\mathrm{supp} (G)=\{\star\in \msf{Ve}_{\msf{A}}\cup\msf{Ed}_{\msf{A}}\;|\; G_\star\neq\{1\}\}\,,
\end{equation}
with $1$ denoting the identity element.

\begin{obs}\label{aretes-support}
Let $G$ be a group-graph over $\A$ and let $\A'$ be a subgraph of $\A$ obtained by removing some edges of $\A$ which are not in the support of $G$. Then the morphism $H^1(\imath_j): H^1(\A,G)\iso H^1(\A',j^*G)$ induced by the canonical morphism $\imath_j:G\to j^*G$ over the inclusion $j:\A'\hookrightarrow\A$ is an isomorphism.
\end{obs}

\begin{defin}\label{regular}
We will say that a group-graph $G$ over $\A$ is \emph{regular} if 
%the complementary of its support $\mr{supp}(G)$ is a subgraph of $\A$ and 
 the restriction morphisms $\rho^e_v: G_v\to G_e$  are isomorphisms as soon as $v, e\in\mr{supp}(G)$. 
\end{defin}

Let $\A'$ be a subtree of a tree $\A$.
An edge $e=\langle v,v'\rangle\in\Ed_\A$ is \emph{adjacent} to $\A'$ if $v\in\Ve_\A\setminus\Ve_{\A'}$ and $v'\in\Ve_{\A'}$.
We define the \emph{contraction} $\A/\A'$ as the tree whose vertices are 
\[\Ve_{\A/\A'}=(\Ve_{\A}\setminus\Ve_{\A'})\sqcup\{v_{\A'}\}\] 
and whose edges are the edges of $\A$ which do not belong to $\A'$ and are not adjacent to $\A'$, jointly with an additional edge $\tilde{e}=\langle v,v_{\A'}\rangle$ for each adjacent edge $e=\langle v,v'\rangle$ to $\A'$ with $v\notin\Ve_{\A'}$,
\begin{equation}\label{bij-edge}
 \Ed_\A\setminus\Ed_{\A'}\iso\,\Ed_{\A/\A'}, \qquad e\mapsto e \text{ or } \tilde e.
 \end{equation}
There is a natural surjective graph morphism $c_{\A'}:\A\to \A/\A'$ given by $c_{\A'}(v)=v_{\A'}$ if $v\in\Ve_{\A'}$ and $c_{\A'}(v)=v$ otherwise.

If $\A''\subset\A'\subset\A$ are subtrees of a tree $\A$ then we have a natural isomorphism 
\begin{equation}
\label{contraction-iteree}
j:\A/\A'\iso(\A/\A'')/(\A'/\A'') \text{ such that } c_{\A'/\A''}\circ c_{\A''}=j\circ c_{\A'}.
\end{equation}

If $G$ is a $\mathbf C$-graph over $\A$ the direct image $\tilde G:=(c_{\A'})_*G$ over $\A/\A'$
satisfies $\tilde G_{v_{\A'}}=H^0(\A',G)$, $\tilde G_{\tilde{e}}=G_e$ if $e\in\Ed_{\A}$ is adjacent to $\A'$ and $\tilde G_\star=G_\star$ otherwise. 
\begin{lema}\label{reg}
Let $G$ be a regular $\mathbf C$-graph over a tree $\A$ and let $\A'$ be a subtree of $\A$  such that all its edges are contained in the support of $G$. Then
$(c_{\A'})_*G$ is a regular $\mathbf C$-graph over $\A/\A'$.
\end{lema}
\begin{proof}
It is easy to check when $\A'$ has only one edge. In the general case, we proceed by induction  on the number of edges of $\A'$ using  isomorphisms~(\ref{contraction-iteree}).
\end{proof}

We call \emph{active edge} of a regular $\mathbf C$-graph $G$ over a tree $\A$ any edge $\msf a=\langle \msf v,\msf v'\rangle\in \Ed_\A$ such that 
$G_{\msf a}\neq \{0\}$ and $G_{\msf v'}=\{0\}$.  If $G_{\msf v}\neq 0$, the vertex $\msf v$ will be called \emph{active vertex associated to $\msf a$} and denoted by $\msf v_{\msf a}$. If $G_{\msf v}= G_{\msf v'}=\{0\}$ and $G_{\msf a}\neq \{0\}$, we select  one of the two vertices $\msf v$ or $\msf v'$ as  active vertex associated to $\msf a$. 
Let $(S_\alpha)_{\alpha\in I}$ be the collection of  \emph{path connected components} of $\mr{supp}(G)$, i.e. the maximal subsets of $\mr{supp}(G)$ such that any two  elements can be joined by a path in $\mr{supp}(G)$. We say that $S_\alpha$ is an \emph{active component} if  it contains an active edge, or equivalently an active vertex. We denote by $I'$ the set of indices $\alpha\in I$ such that $S_\alpha$ is active and not reduced to a single edge.

Let $\ms A$ be the collection  of all active edges.
Now, let us choose one edge $\msf a_\alpha$ in each active component $S_\alpha$, $\alpha\in I'$, and  let us write
\[
\ms A':=\ms A\setminus \{\msf a_\alpha\;;\;\alpha\in I'\}\,.
\]
\begin{teo}\label{teobasegeom} Let $G$ be a regular $\mathbf C$-graph over a tree~$\A$.   
If $\mc A'=\emptyset$ then $H^1(\A,G)=1$, otherwise we
consider the map 
\[[\delta_{{}_G}]:\prod\limits_{\msf a\in\ms A'}G_{\msf a}\to H^1(\A,G)\] 
induced by $\delta_{{}_G}:\prod\limits_{\msf a\in\ms A'}G_{\msf a}\to Z^1(\A,G)$ defined by
$\delta_{{}_G}((g_{\msf a})_{\msf a})=(g_{\msf v,\msf e})$ with $g_{\msf v,\msf e}= 1$ if $\msf e\notin\ms A'$ and 
\[
g_{\msf v_{\ag},\msf a}= g_{\msf a}^{-1}\,,\quad g_{\msf v',\msf a}= g_{\msf a}\]
for $\msf a=\langle \msf v_{\msf a},\msf v'\rangle\in\ms A'$.
Then $[\delta_{{}_G}]$ is bijective and 
 if $\mathbf C=\mbf{Ab}$ or $\mathbf C=\mbf{Vec}$ then $[\delta_{{}_G}]$ is 
a $\mbf C$-isomorphism. Moreover,  
if $\mbf{C=Vec}$ and all the vector  spaces  $G_\star$,  $\star\in\mr{supp}(G)$, have the same dimension $d$ then
\[\dim H^1(\A,G) =(a- p) \cdot d\]
where $a$ is the number of active edges, $p$ is the number of active connected components of $\mr{supp}(G)$ not reduced to a single edge. 
\end{teo} 

\begin{proof}
We reason by induction on the number $n(\A,G)$ of path connected components of $\mr{supp}(G)$ not reduced to a single vertex or a single edge. If $n(\A,G)=0$ the statement is clear. If $n(\A,G)>0$ we consider a path connected component $S_\alpha$ of $\mr{supp}(G)$ not reduced to a single edge nor a single vertex. It contains a nonempty maximal subgraph $C_\alpha'$.
If $S_\alpha$ is an active component
we consider the graph $C_\alpha$ given by $\Ed_{C_\alpha}=\Ed_{C_\alpha'}\cup\{\msf a_\alpha\}\subset\mr{supp}(G)$ and 
 $\Ve_{C_\alpha}=\Ve_{C_\alpha'}\cup\{\msf v_{\msf a_\alpha},\msf v'\}$ where $\msf a_\alpha=\langle \msf v_{\msf a_\alpha},\msf v'\rangle$ is the active edge previously chosen to define~$\ms A'$. 
 If $S_\alpha$ is not an active component then we set $C_\alpha:=C'_\alpha$.
 Let $c:\A\to\tilde\A=\A/C_\alpha$ be the contraction of the subtree $C_\alpha\subset\A$. By Lemma~\ref{reg} the $\mbf C$-graph $\tilde G=c_*G$ over $\tilde{\A}$ is regular and $\ms A'\simeq\tilde{\ms A}'$ under the bijection~(\ref{bij-edge}). 
Moreover we have the following commutative diagram
\[\xymatrix{\prod\limits_{\msf a\in\ms A'}G_{\msf a}\ar[r]^{\delta_{{}_G}}\ar@/^2pc/[rr]^{[\delta_{{}_G}]}\ar[d]& Z^1(\A,G)\ar[r] & H^1(\A,G)\\ \prod\limits_{\tilde{\msf a}\in\tilde{\ms A}'}\tilde G_{\tilde{\msf a}}\ar[r]^{\delta_{{}_{\tilde G}}}\ar@/_2pc/[rr]^{[\delta_{{}_{\tilde{G}}}]} & Z^1(\tilde\A,\tilde G)\ar[u]^{j^1_c}\ar[r] & H^1(\tilde\A,\tilde G)\ar[u]^{H^1(j_c)}}\]
where the left vertical arrow, induced by the bijection (\ref{bij-edge}) using that $\tilde G_{\tilde{\msf a}}=G_{\msf a}$, is the identity. It is clear that every vertex of $C_\alpha\cap\mr{supp}(G)$ is repulsive for the restriction of $G$ to $C_\alpha$.
By applying Theorem~\ref{pruning} we deduce that $H^1(C_\alpha,G)=1$ so that hypothesis (c) in Lemma~\ref{bij} is fulfilled for the contraction map $c:\A\to\tilde{\A}$. Consequently $H^1(j_c)$ is bijective (or a $\mbf C$-isomorphism when $\mbf C=\mbf{Ab}$ or $\mbf C=\mbf{Vec}$).
It is easy to see that if $S_\alpha$ is an active component then $\tilde v:=c(C_\alpha)\in\Ve_{\tilde{\A}}$ does not belong to the support of $\tilde G$, i.e. $\tilde{G}_{\tilde v}=H^0(C_\alpha,G)=1$. If $S_\alpha$ is not active then $\{\tilde v\}$ is a path connected component of $\mr{supp}(\tilde G)$. In both cases  $n(\tilde \A,\tilde G)=n(\A,G)-1$.
By the inductive hypothesis $[\delta_{{}_{\tilde{G}}}]$ is bijective (or a $\mbf C$-isomorphism). Therefore $[\delta_{{}_G}]$ is bijective (or a $\mbf C$-isomorphism). The last assertion is trivial.
\end{proof}

\subsection{Tensor product}\label{GGTensorproduct}
If $T$ is a $\mbf{Vec}$-graph over a graph $\A$  and $W$ is a $\C$-vector space we can define the $\mbf{Vec}$-graph $T\otimes_\C W$ in an obvious way and we obtain a functor
\[\otimes_\C:\VSG\times\mbf{Vec}\to\VSG\,,\]
$\VSG$ being the category of $\C$-vector space-graphs.
The commutative property  between tensor product and direct sum gives  an isomorphism between the functors 
\[(T,W)\mapsto C^*(\A,T\otimes_\C W)\quad\text{and}\quad (T,W)\mapsto C^*(\A,T)\otimes_\C W,\] 
from $\VSG\times\mbf{Vec}$ to the category of vector space complexes. It induces  an  isomorphism 
\begin{equation}\label{tensorprodfunctIso}
\left( (T,W)\mapsto H^1(\A,T\otimes_\C W)\right) 
\iso
\left( (T,W)\mapsto H^1(\A,T)\otimes_\C W\right)
\end{equation}
between functors from the category $\VSG\times\mbf{Vec}$ to $\mbf{Vec}$.

\bigskip

\section{Equisingular deformations of foliations} \label{SectEquising}

\subsection{Deformations of foliations} 
Consider a germ $\F$ of singular foliation at the origin of $\C^2$, given by a germ  $Z=a(x,y)\partial_x+b(x,y)\partial_y$ of holomorphic vector field with $\{a(x,y)=b(x,y)=0\}=\{0\}$.  Let $Q^\point=(Q,u_0)$ be a germ of manifold. A   \emph{deformation of $\F$ over  $Q^\point$} 
is a germ  of foliation $\F_{Q^\point}$ on $(\C^2\times Q,(0,u_0))$ defined by a germ of 
\emph{vertical} (tangent to the fibers of the canonical projection $\mr{pr}_Q : \C^2\times Q\to \C^2$) vector field $X=A(x,y,u)\partial_x+B(x,y,u)\partial_y$,  whose restriction to $\C^2\times \{u_0\}$ is equal to $\F$,
\[ 
D \,\mr{pr}_Q\cdot  X=0\,,\qquad
\iota^\ast\F_{Q^\point}=\F\,,\quad
 \iota:\C^2\hookrightarrow \C^2\times Q\,,\quad\iota(x,y):=(x,y,u_0)\,.
 \]
 The germ $Q^\point$ is called \emph{parameter space} of $\F_{Q^\point}$.
 If $\lambda$ is a germ of holomorphic map from a germ of manifold $P^\point=(P,t_0)$ to $Q$ satisfying $\lambda(t_0)=u_0$, the \emph{pull-back} of $\F_{Q^\point}$ by $\lambda$ is the deformation $\lambda^\ast\F_{Q^\point}$ of $\F$ over $P^\point$,   defined by the vector field $\lambda^\ast X:=A(x,y,\lambda(t))\partial_x+B(x,y,\lambda(t))\partial_y$. When $Q=\{u_0\}$, $\lambda$ is the constant map and $\lambda^\ast\F_{Q^\point}$ is called \emph{constant deformation over} $P^\point$ and is denoted by $\F_{P^\point}^{\mr{ct}}$. 
 %We have the relations $(\lambda\circ\mu)^\ast \F_{Q^\point}=\mu^\ast\lambda^\ast  \F_{Q^\point}$.
 
Two deformations   $\F_{Q^\point}$ and $\F'_{Q^\point}$ of $\F$ with same parameter space $Q^\point$  are \emph{topologically conjugated}, or \emph{$\mc C^0$-conjugated},
if there is a  germ of homeomorphism $\Phi$ that is a \emph{deformation of  $\mr{id}_{\C^2}$}, that sends the leaves of $\F_{Q^\point}$ on that of $\F'_{Q^\point}$
\[ 
\Phi:(\C^2\times Q,(0,u_0))\iso (\C^2\times Q,(0,u_0))\,,\quad  \mr{pr}_Q\circ\Phi=\mr{pr}_Q\,,\quad \Phi\circ\iota=\iota\,,\quad \Phi(\F_{Q^\point})=\F_{Q^\point}';
\]
%sending the leaves of $\F_{Q^\point}$ on that of $\F'_{Q^\point}$; 
we will say that $\Phi$ is a \emph{conjugacy of deformation} from $\F_{Q^\point}$ to $\F'_{Q^\point}$ and we will denote $\Phi: \F_{Q^\point}\to \F'_{Q^\point}$. We will say that a deformation  is  \emph{trivial} if it is conjugated to the constant deformation.\\

\begin{obs}\label{pullbackmorphism} (a) If $\Phi: \F_{Q^\point}\to \F'_{Q^\point}$,  the \emph{pull-back $\lambda^\ast\Phi$ of $\Phi$} by a map germ $\lambda:P^\point\to Q^\point$, defined by 
\[
\lambda^\ast\Phi : (\C^2\times P,(0,t_0))\iso (\C^2\times P,(0,t_0))\,,\quad
\lambda^\ast\Phi(x,y,t):=\Phi(x,y,\lambda(t))\,,
\]
is a conjugacy from the deformation $\lambda^\ast\F_{Q^\point}$ to $\lambda^\ast\F'_{Q^\point}$. (b) If $\mu : N^\point\to P^\point$ is a germ of holomorphic map,  we have the relation $(\lambda\circ\mu)^\ast \F_{Q^\point}=\mu^\ast\lambda^\ast  \F_{Q^\point}$.
\end{obs}
%We will say that a deformation  is  \emph{$\mc C^0$-trivial} if it is $\mc C^0$-conjugated to the constant deformation.\\

%\color{red}
%\subsection{equireducib}
Let us recall  that a deformation $\F_{Q^\point}$ is called \emph{equireducible}  if there exists a map germ  called \emph{equireduction map}

\begin{equation}\label{notequi}
 E_{\F_{Q^\point}} : (M_{\F_{Q^\point}}, \mc E_{u_0}) \to (\C^2\times Q, (0,u_0))
 \end{equation}
obtained by composition of proper holomorphic map germs \[
E_{\F_{Q^\point}}=E_1\circ\cdots\circ E_k\,,\quad
E_j: (M_{j}, K_{j}) \to(M_{j-1},K_{j-1})\,,
\]
 \[
(M_0,K_0)= (\C^2\times Q, (0,u_0))\,,\quad (M_k,K_k)=(M_{\F_{Q^\point}}, \mc E_{u_0}),
 \]
fulfilling the following properties (\ref{regsing})-(\ref{redparam}) below:  for $1\leq j\leq k$ let us write
  \[
 E^j:=  E_1\circ\cdots\circ E_{j}:(M_j,K_j)\to (\C^2\times Q, (0,u_0))\,,\quad 
 \pi_j:=\mr{pr}_Q\circ E^j : M^j\to Q,
  \]
 and let us denote by $ \F^j_{Q^\point}$ the foliation
$(E^j)^{-1}(\F_{Q^\point})$ on $M_j$, then  for  $j=1,\ldots,k$, we must have:
\begin{enumerate}[(i)]
\item \label{regsing}
on  an open neighborhood of $K_j$ in $M_j$  the singular locus   of $\F_{Q}^j$ is regular and  the restriction of $\pi_j$ to it  is a covering map over an open neighborhood of $u_0$ in~$Q$;
 \item\label{bl}
$E_j$ is a blow-up map germ with center a union $C_j$   of components of the singular locus of $\F^{j-1}_{Q^\point}$ 
and $K_j=E_j^{-1}(K_{j-1})$; moreover   $C_1$ is the singular locus $\mr{Sing}(\F_{Q^\point})$ of $\F_{Q^\point}$;
 \item\label{redparam} 
there is an open neighborhood $U\subset Q$ of $u_0$ such that for any $u\in U$ the restriction of $\F^k_{Q^\point}$ to $\pi_k^{-1}(u)$ is a reduced foliation at each of its singular points; moreover  the restriction of $E^k$ to $\pi_k^{-1}(u)$ is  the minimal reduction map of the germ at $\mr{pr}_Q^{-1}(u)\cap \mr{Sing}(\F_{Q^\point})$ of the restriction of  $\F_{Q}$ to $\mr{pr}_Q^{-1}(u)$.

%, at the intersection point of $\mr{pr}_Q^{-1}(u)\cap \mr{Sing}(\F_{Q^\point})$ and $\mr{Sing}(\F_{Q^\point})$.
 %
\end{enumerate}
We will write:
\begin{equation}\label{notequiexcep}
\mc E_{\F_{Q^\point}}:=E_{\F_{Q^\point}}^{-1}(C_1)\,, \quad \pi^\sharp:=\pi_k:(M_{\F_{Q^\point}}, \mc E_{u_0})\to Q^\point \,, \quad \F^\sharp_{Q^\point}:=\F_{Q^\point}^k\,;
\end{equation}
%we have
%\begin{equation}\label{specialfibers}
% \mc E_{u_0}=E^{-1}_{\F_{Q^\point}}(0,u_0)=\mc E_{\F_{Q^\point}}\cap \pi^{\sharp\,-1}(u_0)\,.
%\end{equation}
By induction on $j=1,\ldots,k$, we check that $\pi^\sharp$ is a submersion. 
The \emph{exceptionnal divisor} $\mc E_{\F_{Q^\point}}$ is an hypersurface with normal crossing  and the restriction of $\pi^\sharp$ to each of its irreducible components is a holomorphically trivial fibration with fiber $\mb P^1$. Its \emph{special fiber} 
\begin{equation}\label{specialfibers}
 \mc E_{u_0}=E^{-1}_{\F_{Q^\point}}(0,u_0)=\mc E_{\F_{Q^\point}}\cap \pi^{\sharp\,-1}(u_0)\,.
\end{equation}
 is a curve with  normal crossings and   irreducible components  biholomorphic to $\mb P^1$; the restriction of $E_{\F_{Q^\point}}$ to the \emph{special fiber} $M_{u_0}:=\pi^{\sharp\,-1}(u_0)$ of $M_{\F_{Q^\point}}$ is identified to the \emph{reduction map} $E_{\F}:(M_\F,\mc E_\F)\to\C^2$ of $\F$,
 \begin{equation}\label{identif}
 E_\F\simeq E_{\F_{Q^\point}|M_{u_0}} : (M_{u_0},\mc E_{u_0})\longrightarrow \C^2\times \{u_0\}\simeq\C^2\,,
\quad (M_{u_0},\mc E_{u_0})\simeq (M_\F,\mc E_\F)\,,
 \end{equation}
% \[ 
%E_\F\simeq E_{\F_{Q^\point}|M_{u_0}} : (M_{u_0},\mc E_{u_0})\longrightarrow \C^2\times \{u_0\}\,,\quad% \hbox{ with } 
% \C^2\times \{u_0\}\simeq\C^2\,,
%\quad (M_{u_0},\mc E_{u_0})\simeq (M_\F,\mc E_\F)\,,
% \] 
and the \emph{special fiber of $\F_{Q^\point}^\sharp$} ,
\begin{equation}\label{specialfiberfol}
\F^\sharp_{u_0}:=\F_{Q^\point | M_{u_0}}^{\sharp}\,,
\end{equation}
is identified to the reduced foliation $\F^\sharp:=E_\F^{-1}(\F)$ on $M_\F$.
 %By induction on $j=1,\ldots,k$, the restrictions of $\pi_j$ to the centers $C_j$  of  blow-ups are submersions.
 Notice that any constant deformation $\F_{Q^\point}^{\mr{ct}}$ is equireducible and its reduction map is  the product map of the reduction map of $\F$ with the identity map of $Q$: 
\[
E_{\F_{Q^\point}^{\mr{ct}}}=E_{\F}\times \mr{id}_Q: (M_{\F} \times Q,\mc E_{\F}\times\{u_0\})\to (\C^2\times Q, (0,u_0))\,,\quad (m,u)\mapsto (E_{\F}(m),u)\,;
\]
Using the fact that pull-back process induces biholomorphisms at the fibers level  one  checks the following property:

\begin{prop}
The pull-back $\mu^\ast\F_{Q^\point}$  of an equireducible deformation $\F_{Q^\point}$ over $Q^\point$ of a foliation $\F$ by a holomorphic map germ $\mu: P^\point\to Q^\point$, is an equireducible  deformation of $\F$   over $P^\point$ and its equireduction map is the pull-back $\mu^\ast E_{\F_{Q^\point}}$ of the equireduction map of  
$\F_{Q^\point}$.
\end{prop}

%\begin{proof} Properties (\ref{regsing}) and (\ref{bl}) for $\mu^\ast\un\F_{Q^\point}$ result from the following facts: 
% if $\varpi :N\to Q$ is a submersion, $C\subset N$ is a submanifold and  the restriction of $\varpi$ to $C$ 
%also is a submersion, then: a)  the blow-up map $E_C:N^\sharp\to N$ with center $C$ is a \emph{map over $Q$}, that is $\varpi^\sharp:=\varpi\circ E_C$ is a submersion; b) the restriction of $\varpi^\sharp$ to the exceptional divisor $\varpi^{-1}(C)$ is also  a submersion; c)  the pull-back $\mu^\ast E_C : \mu^\ast N^\sharp\to\mu^\ast N$ of $E_C$, as map over $Q$, is equal to the blow-up map with center  $\mu^\ast C\subset \mu^\ast N$. On the other hand, the pull-back process inducing biholomorphisms at the fibers level, property~(\ref{redparam}) is satisfied. 
%\end{proof}

For equireducible deformations we may consider a special class of $\mc C^0$-conjugacies:

\begin{defin}\label{excellent}
Let
$\F_{Q^\point}$ and $\F'_{Q^\point}$ be two deformations over $Q^\point=(Q,u_0)$ of a foliation $\F$ and let $F: (\C^2\times Q,(0,u_0))\iso (\C^2\times Q,(0,u_0))$, $\mr{pr}_Q\circ F=\mr{pr}_Q$,  be a homeomorphism   that sends the leaves of $\F_{Q^\point}$ to the leaves of $\F'_{Q^\point}$. We will say that~$F$ is \emph{excellent} or of class \Cex, if
\begin{enumerate}
\item\label{holsing} $F$ lifts through the reduction maps of these foliations 
\[E_{\F_{Q^\point}}:(M_{\F_{Q^\point}},\mc E_{u_0})\to\C^2\times Q\,,\quad
E_{\F'_{Q^\point}}:(M_{\F'_{Q^\point}},\mc E'_{u_0})\to\C^2\times Q\,,
\]
i.e. there is a (unique) germ of  homeomorphism 
$F^\sharp : (M_{\F_{Q^\point}},\mc E_{u_0})\to(M_{\F'_{Q^\point}},\mc E'_{u_0})$ satisfying $E_{\F'_{Q^\point}}\circ F^\sharp=F\circ E_{\F_{Q^\point}}$,
\item\label{anlocnonoeud}  $F^\sharp$ is holomorphic at each point of $\mr{Sing}(\mc E_{u_0})\cup\mr{Sing}(\F^\sharp_{u_0})\subset \mc E_{u_0}$, except perhaps at the singular points of $\mc E_{u_0}$ that are nodal singularities of the special fiber  $\F^\sharp_{u_0}$ of $\F^\sharp_{Q^\point}$, cf. (\ref{specialfiberfol}).
\end{enumerate}
\end{defin}

\begin{obs}\label{trhol}
According to Camacho-Sad index Theorem, there is a non-nodal singular point of $\F^{\sharp}_{Q^\point}$ in  each invariant  component of the special fiber $\mc E_{u_0}$ of the exceptional divisor of the reduction of $\F_{Q^\point}$; consequently the holomorphy property~(\ref{anlocnonoeud}) in Definition~\ref{excellent} induces the   transversal holomorphy  of $F^\sharp$ at any regular point of the foliation $\F^\sharp_{Q^\point}$.
\end{obs}

\begin{obs}\label{pullbakconjugacy}
If $\mu:P^\point\to Q^\point$ is a holomorphic map germ and $F$ is a \Cex-conjugacy between two equireducible deformations $\F_{Q^\point}$ and $\G_{Q^\point}$ of the same foliation $\F$, then $\mu^\ast F$ is a \Cex-conjugacy between the deformations $\mu^\ast\F_{Q^\point}$ and $\mu^\ast\G_{Q^\point}$. 
\end{obs}

\subsection{Equisingular deformations}\label{Equising}
%\color{red}
 Let us  consider  an equireducible foliation $\F_{Q^\point}$,   over a germ of manifold $Q^\point=(Q,u_0)$, of a foliation $\F$ on $(\C^2,0)$. We  keep all 
 previous   notations (\ref{notequiexcep})-(\ref{specialfiberfol}). We will denote by $\mr{Diff}(\C\times Q,(0,u_0))$  the group of germs of    holomorphic automorphisms of $(\C\times Q,(0,u_0))$ fixing the point $(0,u_0)$ and  by 
\begin{equation}\label{diffQ}
 \mr{Diff}_Q(\C\times Q,(0,u_0)) := \{  h\in \mr{Diff}(\C\times Q,(0,u_0))
 \; | \; \mr{pr}_Q\circ h=\mr{pr}_Q \}\,,
\end{equation}
the subgroup of \emph{automorphisms over $Q$}.

Now let us fix a point $o_D$ in each \emph{$\F_{u_0}$-invariant}  component $D$ of   $\mc E_{u_0}$  that is a non-singular point of this foliation and let us   choose  a germ of holomorphic submersion 
\begin{equation*}\label{submtrans}
g_D : (M_{\F_{Q^\point}}, o_D)\to (\C\times Q,(0,u_0))\,,\quad 
g_D(o_D)=(0,u_0)\,,
\end{equation*}
that is a \emph{map over $Q^\point$}, i.e.  $\mr{pr}_Q\circ g_D=\pi^\sharp$, constant on the leaves of $\F^\sharp_{Q^\point}$. We will say that $g_D$ is a \emph{transversal factor}  to $\F_{Q^\point}^\sharp$  at the point $o_D$.
%
%
%Let us denote by $\mr{Diff}(\C\times Q,(0,u_0))$  the group of germs of    holomorphic automorphisms of $(\C\times Q,(0,u_0))$ fixing the point $(0,u_0)$ and consider the subgroup of automorphisms over $Q$
%\begin{equation}\label{diffQ}
% \mr{Diff}_Q(\C\times Q,(0,u_0)) := \{  h\in \mr{Diff}(\C\times Q,(0,u_0))
% \; | \; \mr{pr}_Q\circ h=\mr{pr}_Q \}\,, 
%\end{equation}
%
% and  then a holomorphic section
%of $g_D$ 
%\begin{equation*}\label{transversales}
%%\sigma_{D} :  (\C\times Q, (0,u_0))\hookrightarrow (M_{\F_{Q^\point}}, o_D)\,,\quad \sigma_D(0,u_0)=o_D\,.
%%\end{equation*}
%Clearly $\sigma_D$ is  a map   over $Q^\point$ and its image is a manifold transversal to $D$ and therefore to the leaves of $\F_{Q}$; we will say that $(g_D,\sigma_D)$ is a \emph{transversal factor over $Q^\point$ to $\F_{Q^\point}^\sharp$  at $o_D$}.
%
Classically the \emph{holonomy} of $\F^\sharp_{Q^\point}$ along $D$ 
\emph{realized} on $g_D$ is the group representation of the  fundamental group of the \emph{punctured component} 
$D^\ast:=D\setminus \mr{Sing}(\F_{Q}^\sharp)$
 \begin{equation}\label{holonomiefam}
  \mc H_{D}^{\F^\sharp_{Q^\point}} : \pi_1(D^\ast, o_D)\to \mr{Diff}_Q(\C\times Q,(0,u_0))
 \end{equation}
 that associates to the class of a  loop $\gamma$ in $D^\ast$, $\gamma(0)=o_D$,  the  automorphism $h_\gamma$ over $Q^\point$ such that $g_D\circ h_\gamma^{-1}$ is the analytic extension (equivalently the extension  as first integral of $\F^\sharp_{Q^\point}$) of $g_D$ along $\gamma$. %(equivalently the extension of $g_D$ along $\gamma$ as first integral of $\F^\sharp_{Q^\point}$).  
Up to composition by  inner automorphisms of  $\mr{Diff}_Q(\C\times Q,(0,u_0))$, this representation does not depend on the choice of the point  $o_D$ in $D^*$ or that of the transversal factor $g_D$.
%\begin{obs}\label{defHolFol} 

For  a germ of holomorphic map $\mu : P^\point\to Q^\point$ we will  identify to $M_\F$ the special fibers of the reductions of $\F_{Q^\point}$ and of $\mu^\ast\F_{Q^\point}$, see (\ref{identif}). The pull-back by $\mu$ of a submersion over $Q^\cdot$,  resp. a first integral over $Q^\point$ of $\F^\sharp_{Q^\point}$, being a submersion over $P^\point$, resp. a first integral over $P^\point$ of $\mu^\ast\F^\sharp_{Q^\point}$, we have:
\begin{itemize}%[-]
\item
the pull-back $\mu^\ast g_D$ of a transversal factor $g_D$ to $\F_{Q^\point}^\sharp$, considered as a map over $Q^\point$, 
is a  transversal  factor to  $\mu^\ast\F_{Q^\point}^\sharp$ at the same point   of  the same invariant  component  $D$ of $\mc E_{\F}$, and the holonomy of $\mu^\ast\F^\sharp_{Q^\point}$ represented on it is
\begin{equation}\label{imrechol}
\mc H^{\mu^\ast \F^\sharp_{Q^\point}}_D= \mu^* \circ \mc H^{\F^\sharp_{Q^\point}}_D\,,
\end{equation}
where %$\mu^*$ is  the following group morphism:
\[\mu^* : 
\mr{Diff}_Q(\C\times Q, (0,u_0))\to\mr{Diff}_P(\C\times P, (0,t_0))\,,\quad h\mapsto \big(\mu^\ast h:(z,t)\mapsto h(z,\mu(t)\big)\,;
% \,,\quad \mu^\ast (z,t)\mapsto h(z,t):=h(z,\mu(t))\,.
\]
%$\mu^\ast h$ being the pull-back by $\mu$  of $h$ considered as a map over $Q^\point$;

%

\item if $H_D$ denotes the \emph{holonomy group}
of $\F^\sharp_{Q^\point}$ along $D$, i.e. the  image of the morphism $\mc H^{\F^\sharp_{Q^\point}}_D$, then $\mu^*( H_D)$
is the holonomy group  of $\mu^\ast \F^\sharp_{Q^\point}$ along $D$.
\end{itemize}
%\end{obs}

%\color{red}

Let us denote by $\mr{Diff}(\C,0)\times \{\mr{id}_{Q}\}\subset  \mr{Diff}_Q(\C\times Q,(0,u_0))$ the subgroup of automorphisms  that do not depend on $u\in Q$.
\begin{defin}\label{SL-def} 
We say that a  deformation $\F_{Q^\point}$ of $\F$ over $Q^\point$   is \emph{equisingular},  if it is equireducible and 
the holonomy representation of the reduced foliation $\F^\sharp_{Q^\point}$ along any invariant component  $D$ of the special fiber $\mc E_{u_0}$  of the exceptional divisor $\mc E_{\F_{Q^\point}}$ 
is conjugated to a morphism with values in $\mr{Diff}(\C,0)\times \{\mr{id}_{Q}\}$: there exists $\psi_D\in \mr{Diff}_Q(\C\times Q,(0,u_0))$ such that 
\begin{equation*}\label{trivholon}
\tau_{\psi_D}\circ \mc H^{\F^\sharp_{Q^\point}}_D 
 : \;\pi_1(D^\ast, o_D)\to \mr{Diff}(\C,0)\times \{\mr{id}_{Q}\}\subset
 \mr{Diff}_Q(\C\times Q,(0,u_0))
\end{equation*}
where $\tau_{\psi_D}$ is the inner automorphism $\phi\mapsto\psi_D\circ \phi\circ\psi_D^{-1}$ of $ \mr{Diff}_Q(\C\times Q,(0,u_0))$.
\end{defin}
\noindent In other words, an  equireducible foliation  $\F_{Q^\point}$ is equisingular if and only if for any invariant component $D$ of $\mc E_{u_0}$, the holonomy representation $ \mc H^{\F^\sharp_{Q^\point}}_D $  is conjugated to the holonomy representation along $D$  of the constant foliation $\F_{Q^\point}^{\mr{ct}\,\sharp}$, i.e.
\begin{equation}\label{conjholcte}
\tau_{\psi_D}\circ \mc H^{\F^\sharp_{Q^\point}}_D =
\mc H_D^{ \F_{Q^\point}^{\mr{ct}\,\sharp}}.
\end{equation}
for an appropriate $\psi_D\in \mr{Diff}_Q(\C\times Q,(0,u_0))$. 
 \begin{prop} The pull-back by a holomorphic map germ  $\mu:P^\point\to Q^\point$ of an equisingular  deformation $\F_{Q^\point}$ over $Q^\point$ is  an equisingular deformation over $P^\point$.
\end{prop}
  
\begin{proof} 
Let us suppose equality (\ref{conjholcte}) satisfied, and let us denote by $\kappa_{P\point}: P^\point\to P^\point$ the constant map $t\mapsto t_0$. Since $\kappa_{P^\point}^\ast \mu^\ast \F_{Q^\point}$ is the constant deformation of $\F$ over $P^\point$,  it suffices to prove the equality
\begin{equation}\label{egdesire}
\tau_{\mu^\ast \psi_{D}} \circ \mc H_D^{\mu^\ast \F_{Q^\point}^\sharp} = 
\mc H_D^{\kappa_{P^\point}^\ast \mu^\ast \F_{Q^\point}^\sharp}\,, 
\end{equation}
$\kappa_{P^\point}: P^\point\to P^\point$ being the constant map $t\mapsto t_0$. 
Trivially we have: $\tau_{\mu^\ast \psi_D}\circ \mu^*=\mu^*\circ \tau_{\psi_D}$. Hence, it follows from (\ref{imrechol}) and (\ref{conjholcte}):
\[
\tau_{\mu^\ast \psi_{D}} \circ \mc H_D^{\mu^\ast \F_{Q^\point}^\sharp} =
\tau_{\mu^\ast \psi_{D}} \circ \mu^* \circ \mc H_D^{\F_{Q^\point}^\sharp}=
\mu^* \circ \tau_{\psi_D}
\circ \mc H_D^{\F_{Q^\point}^\sharp}=
\mu^*\circ \mc H_D^{\kappa^\ast_{Q^\point}\F_{Q^\point}^\sharp}=\mc H_D^{\mu^\ast\kappa^\ast_{Q^\point}\F_{Q^\point}^\sharp},
\]
the last equality follows from the fact that the constant deformation $\kappa^\ast_{Q^\point}\F_{Q^\point}$ is equisingular and thus fulfills the corresponding relation (\ref{imrechol}).
Equality (\ref{egdesire}) results from the trivial relation $\kappa_{Q^\point}\circ \mu=\mu\circ \kappa_{P^\point}$ that gives 
$\mu^\ast {\kappa_{Q^\point}^\ast \F_{Q^\point}}
=\kappa_{P^\point}^\ast \mu^\ast \F_{Q^\point}$.
\end{proof}

%\section{Equisingular deformations of foliations} \label{SectEquising}

\subsection{Good trivializing system}\label{subsecMarkFol} 
In all the sequel we will make the hypothesis that the considered foliations $\F$ are \emph{generalized curves}, i.e.  the reduced foliations $\F^\sharp$ have no saddle-node singularities. 
Consequently at each singular point $s$ of $\F^\sharp$ in an invariant component $D$ of $\mc E_\F$, the holonomy around $s$ and the
 \emph{Camacho-Sad index}  $\mr{CS}(\F^\sharp,D,s)$  determine the analytical type of the germ of $\F^\sharp$ at $s$. We will see that this property will  imply the ``$\mc C^{\mr{ex}}$-rigidity'' of $\F^\sharp$ along each component $D$ of $\mc E_\F$, in the meaning that the germ along $D$ of the reduced foliation associated to any equisingular deformation of $\F$, is $\mc C^{\mr{ex}}$-conjugated to that of the constant deformation. \\

Let us consider an equisingular deformation $\F_{Q^\point} $ 
of $\F$. Let us  keep the 
 previous     notations (\ref{notequiexcep})-(\ref{specialfiberfol}) and let us denote by 
 \begin{equation}\label{imcanlifted}
 \iota^\sharp:(M_\F,\mc E_\F)\hookrightarrow (M_{\F_{Q^\point}},\mc E_{u_0})\,,\quad 
 E_{\F_{Q^\point}}\circ \iota^\sharp = \iota\circ E_\F\,,
 \end{equation}
the lifting throught the reduction and equireduction maps of the canonical  immersion 
\begin{equation}\label{immcan}
\iota:(\C^2,0)\hookrightarrow (\C^2\times Q,(0,u_0))\,,\quad (x,y)\mapsto (x,y,u_0)\,.
\end{equation}
We will also denote by $j^\sharp : M_\F\hookrightarrow M_\F\times Q$ the canonical immersion $m\mapsto(m,u_0)$, by $\mr{pr}_Q :\C^2\times Q\to Q$ and $\mr{pr}^\sharp_Q:M_\F\times Q\to Q$ the canonical projections, and we again write $\pi^\sharp:=\mr{pr}_Q\circ E_{\F_{Q^\point}} : (M_{\F_{Q^\point}},\mc E_{u_0})\to Q$.

\begin{teo}\label{psiD} %[Good trivializing system]
If $\F$ is a generalized curve, then we can associate to  each irreducible component $D$ of $\mc E_\F$,  a homeomorphism germ %over $Q^\point$
\[
\Psi_D : 
(M_{\F_{Q^\point}}, \iota^\sharp(D)) \iso (M_\F\times Q, D\times \{u_0\})\,,
%\quad \mr{pr}_Q\circ\Psi_D=\mr{pr}_Q\circ E_{\F_{Q^\point}}\,,
\]
so that:
\begin{enumerate}[(i)]
\item\label{overQ} $\Psi_D$ is a map over $Q^\point$, i.e. $\mr{pr}^\sharp_Q\circ\Psi_D=\pi^\sharp$, and corresponds to the identity map over $u_0$, i.e. 
$\Psi_D\circ \iota^\sharp=j^\sharp$;
\item\label{locholexcept}  $\Psi_D$ is holomorphic at each 
point of $\mr{Sing}(\mc E_{u_0})\cup\mr{Sing}(\F^\sharp_{u_0})$ except perhaps at the singular  points of 
$\mc E_{u_0}$  
that are nodal singularities of $\F^\sharp_{u_0}$; 
\item\label{conjfeuilletages} $\Psi_D$ conjugates the foliation $\F_{Q^\point}^\sharp$  to 
the foliation $\F_{Q^\point}^{\mr{ct}\,\sharp}$ obtained after equireduction of the constant deformation $\F_{Q^\point}^{\mr{ct}}$;
\item\label{collagenoeuddic} the germ of $\Psi_{D}\circ\Psi_{D'}^{-1}$ at the intersection point  $\{s_{DD'}\}=(D\cap D')\times\{u_0\}$ of two irreducible components $D$ and $D'$,  is the identity when either  $s_{DD'}$ is a nodal singular point of $\F^\sharp_{u_0}$ or $s_{DD'}$ is a regular point of $\F^\sharp_{u_0}$.
\end{enumerate}
\end{teo}
\noindent The collection $(\Psi_D)_D$ of these homeomorphisms indexed by the components of $\mc E_\F$ is called \emph{good trivializing system} for $\F_{Q^\point}$.
\begin{proof} 
We will proceed in five steps. \\

\noindent\textit{-Step 1: construction of $\Psi_D$ on a neighborhood $\Omega$ of 
$\iota^\sharp(D\setminus \mr{Sing}(\F^\sharp))$ with $D$ invariant.} 
Let us fix a point  $o_D\in D\setminus \mr{Sing}(\F^\sharp)$ and  a transversal factor to $\F_{Q^\point}^\sharp$ 
\begin{equation*}\label{submintpro}
g:(M_{\F_{Q^\point}}, \iota^\sharp(o_D)) \to (\C\times Q,(0,u_0))\,.
\end{equation*}
Let us also fix a $\mc C^\infty$ submersion 
\begin{equation*}\label{retactD}
\rho:W\to \iota^\sharp(D)
\end{equation*}
defined on a neighborhood $W$ of $\iota^\sharp(D)$ in $M_{\F_{Q^\point}}$, such that:
\begin{enumerate}[(i)]
\item\label{resrho} the restriction of $\rho$ to $ \iota^\sharp(D)$ is the identity map,
\item the restriction $\rho_0$  of $\rho$ to the special fiber $M_{u_0}:=\pi^{\sharp\,-1}(u_0)$ is a submersion,
\item $\rho$ is holomorphic at $\iota^\sharp(o_D)$ and  also at each point $s\in \mr{Sing}(\mc E_{u_0})\cup \mr{Sing}(\F_{u_0}^\sharp)$,
\item\label{invfibers} the fibers $\rho^{-1}(s)$, $s\in  \mr{Sing}(\mc E_{u_0})\cup \mr{Sing}(\F_{u_0}^\sharp)$,  are  invariant by $\F_{Q^\point}^\sharp$.
\end{enumerate}
There is a unique section  $\sigma: (\C\times Q,(0,u_0))\to (M_{\F_{Q^\point}}, \iota^\sharp(o_D))$  of $g$, whose image coincides with the fiber $\rho^{-1}(\iota^\sharp(o_D))$. 
We do  a similar construction for the constant deformation. First, at the point 
$\br o_D:= j^\sharp(o_D)$ we have the following transversal factor 
\[
\br g= \br g_0\times \mr{id}_Q : (M_\F\times Q,(\br o_D,u_0))\to (\C\times Q,(0,u_0))\,,\quad \br g_0 : =\mr{pr}_\C\circ g\circ \iota^\sharp\,,
\]
with $\mr{pr}_{\C}:\C\times Q\to \C$ the first projection. Next, we define the following submersion $\br\rho$ onto $D\times\{u_0\}$
\[
\br\rho :  \iota^{\sharp\,-1}(W)\times 
Q\to D\times\{u_0\}\,,
\quad
(m,u)\mapsto (\iota^{\sharp\,-1} \circ\rho_0\circ\iota^\sharp(m), u_0)\,.
\]
Finally we consider the section $\br\sigma$ of $\br g$ whose image coincides with  $\br\rho^{-1}(\br o_D)$.

Now let us fix an element  $\psi_D\in \mr{Diff}_{Q}(\C\times Q,(0,u_0))$ that  conjugates the holonomy representation  along $\iota^\sharp(D)$
of  $\F^\sharp_{Q^\point}$ realized on $g$, to that of $\F_{Q^\point}^{\mr{ct}\,\sharp}$  realized on $\br g$:
\[
\tau_{\psi_D}\circ \mc H_ D^{\F_{Q^\point}^{\sharp\,}} =\mc H_ D^{\F_{Q^\point}^{\mr{ct}\,\sharp}} \,,\quad \tau_{\psi_D}(\phi):=\psi_D\circ\phi\circ\psi_D^{-1},
\] 
as in Definition~\ref{SL-def} and equation (\ref{conjholcte}).
By classical theory of path lifting in leaves of regular $1$-dimensional foliations, there is a homeomorphism $\Psi:\Omega \to \br\Omega$ where $\Omega$ is an open neighborhood of $\iota^\sharp(D\setminus \mr{Sing}(\F^\sharp))$ in $W\subset M_{\F_{Q^\point}}$ and $\br\Omega$ is an open neighborhood of $(D\setminus \mr{Sing}(\F^\sharp))\times \{u_0\}$ in $M_\F\times Q$, satisfying the following properties:
\begin{itemize}
\item when restricted  to $\iota^\sharp(D\setminus \mr{Sing}(\F^\sharp))$, $\Psi$ coincides with the map 
\begin{equation*}\label{psibemol}
\Psi^\flat
: \iota^{\sharp}(D)\iso D\times \{u_0\}\,,
\quad
p\mapsto (\iota^{\sharp\,-1} (p),u_0)\,,
\end{equation*}
\item $\Psi$ sends the fiber $\rho^{-1}(\iota^\sharp(o_D))$ to the fiber $\br\rho^{-1}(\br o_D)$ and its restriction to $\rho^{-1}(\iota^\sharp(o_D))$ is equal to $\br\sigma\circ\psi_D\circ g$,
\item $\Psi$  conjugates the restriction of  $\F_{Q^\point}^\sharp$ to $\Omega$ to that of $\F_{Q^\point}^{\mr{ct}\,\sharp}$
to $\br\Omega$,
\item $\Psi$ is a lift of $\Psi^\flat$, that is $\br\rho\circ\Psi=\Psi^\flat\circ \rho$. 
\end{itemize}

By construction, $\Psi$ is a map over $Q^\point$, i.e. $\mr{pr}^\sharp_Q\circ \Psi=\pi^\sharp$ and its germ along $\iota^\sharp(D\setminus \mr{Sing}(\F^\sharp))$ is unique.   Moreover, $\rho$ being holomorphic at the singular points, $\Psi$ is also holomorphic on the intersection of $\Omega$ with neigborhoods of these points.\\

\noindent\textit{-Step 2: extension at a non-nodal singular point.}   
The proof of Mattei-Moussu's theorem~ \cite{MatMou} given in \cite[Theorem 5.2.1]{Frank}  shows that the closures  of $\Omega$ and $\br\Omega$ at the non-nodal singular points of $\F^\sharp_{u_0}$ are neighborhoods of these points; in fact, the estimates made in \cite{Frank} are uniform in the parameters, see also \cite{DiawLoray}. Since $\Psi$ constructed in Step~1 is holomorphic near these singularities we conclude that $\Psi$  extends  holomorphically at these points  by classical Riemann's theorem. \\

\noindent\textit{-Step 3: construction of $\Psi_D$ when $D$ is dicritical.} Classically, the holomorphic type of $\F_{Q^\point}^\sharp$ along a dicritical divisor $\iota^\sharp(D)$ only depends on the self-intersection number of $\iota^\sharp(D)$ in the special fiber 
$\pi^{\sharp\,-1}(u_0)$. 
Thus there exists a germ of biholomorphism $\Psi : 
(M_{\F_{Q^\point}}, \iota^\sharp(D)) \iso (M_\F\times Q,D\times \{u_0\})$ over $Q^\point$ that conjugates $\F_{Q^\point}^\sharp$ to $\F_{Q^\point}^{\mr{ct}\,\sharp}$. Up to conjugating by a biholomorphism of $(M_\F\times Q,D\times \{u_0\})$ leaving $\F_{Q^\point}^{\mr{ct}\,\sharp}$ invariant we may also suppose that $\Psi\circ \iota^\sharp=j^\sharp$. It remains to modify $\Psi$ at each point  where $\iota^\sharp(D)$ meets another component $\iota^\sharp(D')$ so that  at this point the germ of $\Psi$ coincides with that of the homeomorphism   constructed in Step 1 for $D'$. This follows from the following remark:

\begin{obs}\label{collagedicr}
Let us consider  two germs of biholomorphisms over $\C^q$ 
\[g^j:\left(\C^2\times \C^q,\,\overline{\mb D}_{1}\times \{0\}\right)\iso\left(\C^2\times \C^q,\,g^j(\overline{\mb D}_{1}\times \{0\})\right)\,,
\]
$j=1,2$,  of the following form: 
\[
g^j (x,y,u)=(g^j_1(x,u), g^j_2(x,y, u), u)\,,
\quad
u=(u_1,\ldots,u_q)\,,
\]
with $g_1^j:\left(\C\times \C^q,\,\overline{\mb D}_{1}\times \{0\}\right)\to \C$, satisfying
\begin{equation}\label{condsaxes}
g^j_1(0,u) = g^j_2(x,0,u)=0
\,,\quad
g^j_1(x, 0)=x\,,\quad g^j_2(x,y, 0)=y\,.
\end{equation}
Here $\overline{\mb D}_1$ denotes the closed unit disk on $\C$.
Then   for suitable real numbers $0<r_1<R_1<1$, 
there exists a  homeomorphism germ    
\[g:(\C^2\times \C^q,\,\overline{\mb D}_{1}\times \{0\})\iso(\C^2\times \C^q,\, g(\overline{\mb D}_{1}\times \{0\})),
\]
\[
g (x,y,u)=(g_1(x,u), g_2(x,y, u), u)\,,
\]
of the same form,  
satisfying also Properties (\ref{condsaxes}), such that 
\[
g(x,y,u)=g^1(x,y,u) \hbox{ if } |x| \leq r_1\,,\qquad
g(x,y,u)=g^2(x,y,u) \hbox{ if } R_1\leq |x|\leq 1\,.
\]
\end{obs}
\begin{dem2}{of the remark}
Let $0<r_1<R_1<1$ and $\epsilon>0$ be  reals numbers such that 
\[\sup_{\tiny\begin{array}{c}|x|\leq r_1\\ |u|\le\epsilon\end{array}}|g_1^1(x,u)| < \inf_{\tiny\begin{array}{c} R_1\leq |x|\leq 1\\ |u|\le\epsilon\end{array}}|g_1^2(x,u)|\,,
\]
and let us choose  $r_2$, $R_2\in\R$ satisfying $r_1<r_2<R_2<R_1$. 
Similarly to \cite[Proposition~5.13]{MMS} one can prove that
there is  a homeomorphism germ over $\C^q$
 \[
 G:(\C\times\C^q,\overline{\mb D}_1\times\{0\})\iso
 (\C\times\C^q,\overline{\mb D}_1\times\{0\})\,,
 \quad
 G(z,u)=(G_1(z,u),u)\,,
 \]
such that: 
$G_1(z,0)=z$, $G_1(z,u)=g_1^1(z,u)$ if $|z|\leq r_1$ and 
$ G_1(z,u)=g_1^2(z,u)$   if $R_2\leq |z|\leq 1$.
Let us fix  continuous maps 
\[
\zeta: \{r_1\leq|z|\leq r_2\}\to \{|z|\leq r_1\}\,,\quad
\xi: \{R_2\leq|z|\leq R_1\}\to \{|z|\leq R_1\}\,,
\]
that induce homeomorphisms when restricted to the interior of these compact annuli,  
such that $\zeta(z)=z$  if $|z|=r_1$,  $\zeta(z)=0$ if $|z|=r_2$ and $\xi(z)=0$ if  $|z|=R_2$, $\xi(z)=z$ if $|z|=R_1$. Let us also  fix
a continuous family  $F_{\tau}(z,u)$, $\tau\in [r_2,R_2]$, of function germs  at $(0,0)\in\C\times\C^{q}$
that are defined and holomorphic on a common domain, such that 
$\partial_zF_\tau(0,0)\neq 0$,
$F_{r_2}(z,u)=g^1_2(0,z,u)$ and $F_{R_2}(z,u)=g^2_2(0,z,u)$. 
Then 
we set 
$g_1(x,u)=G_1(x,u)$ and
 \[g_2(x,y,u)=\left\{\begin{array}{llr}
 g^1_2(x,y,u) & \text{if} &  |x|\leq r_1\,,\\
g^1_2(\zeta(x),y,u) & \text{if} &  r_1\leq |x|\leq r_2\,,\\
F_{|x|}(y,u) & \text{if} & r_2\leq |x|\leq R_2\,,\\
 g^2_2(\xi(x),y,u) & \text{if} &  R_2\leq |x|\leq R_1\,,\\ 
 g^2_2(x,y,u)  &\text{if} & R_1\leq |x|\leq 1\,.
 \end{array}\right.\]
\end{dem2}

\noindent\textit{-Step 4: Extension  at a nodal singular point  $s\notin \mr{Sing}(\mc E_{u_0})$.}  The extension of $\Psi$ will be done ``by linearity'' as follows. Let  $\zeta=(\zeta_1,\cdots,\zeta_q):(Q,u_0)\to(\C^q,0)$ be a chart on $Q^\point$. Since  the holonomy around $s$ is a trivial family, Camacho-Sad index of $\F_{Q^\point}^\sharp$ restricted to the fibers of $\pi^\sharp$ is constant along the singular locus. By linearization (with parameters) there is a local chart
\begin{equation*}\label{linearchart}
\chi=(w_1, w_2, z_1,\ldots,z_q):(M_{\un\F_{Q^\point}},s)\to (\C^2\times \C^q,0)\,,
\quad
z_j=\zeta_j\circ \pi^\sharp\,,
\end{equation*}
such that   $\F_{Q^\point}^\sharp=\chi^{-1}(\mc L)$, where $\mc L$ is the one dimensional  foliation on $\C^{q+2}_{x,y,u_1,\ldots,u_{q}}$, with singular set $\{(0,0)\}\times \C^q$,  given by the linear differential equations system 
\begin{equation}\label{linearsystemdif}
xdy-\alpha y dx=du_1=\cdots=du_{q}=0\,,
\quad 
\alpha\in \R_{>0}\,.
\end{equation}
We may suppose that the $x$-axis corresponds to  $\iota^\sharp(D)$ and that $\rho$ corresponds to the linear projection on the first coordinate $w_1$ in $\C^2$. At the point  $\br s:= (\iota^{\sharp\,-1}(s),u_0)\in M_\F\times Q$, with the local chart
\begin{equation*}\label{linearchartct}
\br\chi=(w_1\circ \iota^\sharp, w_2\circ \iota^\sharp, \zeta_1,\ldots,\zeta_q)
:(M_{\F}\times Q,\br s) \to (\C^2\times \C^q,0)\,,
\end{equation*}
the component $D\times \{u_0\}$ corresponds again to  the $x$-axis,  $\br\rho$ is the linear projection and we have: $\F^{\mr{ct}\,\sharp}_{Q^\point}=\br\chi^{-1}(\mc L)$. Notice that $\br\chi\circ\Psi\circ\chi^{-1}$ is a  holomorphic automorphism  leaving invariant the foliation $\mc L$, defined  on a neigbourhood in $\C^{q+2}$ of a punctured disk $\mc D^\ast=\{0<|x|\leq\varepsilon, y=0,u=0\}$. It has the following expression:
\begin{equation*}\label{exprePsiChi1}
\br\chi\circ\Psi\circ\chi^{-1}(x,y,u)= 
\left( x, \wt\Psi(x,y,u),u\right)\,,
\quad
u=(u_1,\ldots,u_q)\,,
\end{equation*}
\begin{equation*}\label{exprePsichi2}
\wt\Psi(x,0,u)=0\,,
\quad
\wt\Psi(x,y,0)=(x,y,0)\,.
\end{equation*}
On  $\{x\}\times \C^{q+1}$, $x\in \mc D^\ast$,  the holonomy of $\mc L$ along the loop $ \gamma_{x}(t)=(e^{2\pi it}x, 0,\ldots, 0)$, $t\in [0,1]$,  is the linear automorphism $h(x,y,u)=(x,e^{2\pi \alpha i}y,u)$. The commutativity of $\br\chi\circ\Psi\circ\chi^{-1}$ with these holonomy maps, 
\[\wt \Psi(x,e^{2\pi \alpha i}y, u)=e^{2\pi \alpha i}\wt\Psi(x,y,u)\,,\]
 gives 
\[
\wt\Psi(x,y,u)=A(x,u) y\,,\quad A(x,u)\neq 0\,,
\]
where $A$ is a holomorphic map defined  on an open set  of $ \C^{q+1}_{x,u_1,\ldots,u_{q}}$ that contains  the compact set defined by $\varepsilon/2 \leq |x|\leq \varepsilon$, $|u_j|\leq \eta$ for $j=1,\ldots,q$. By the  invariance of $\mc L$ under $\br\chi\circ\Psi\circ\chi^{-1}$, we have  the equality: 
\[
(-\alpha \frac{dx}{x} +\frac{d\wt \Psi}{\wt\Psi})\wedge (-\alpha \frac{dx}{x} +\frac{dy}{y})\wedge du_1\wedge\cdots\wedge  du_q=0\,.
\]
Hence:
\[
\frac{dA}{A}\wedge (-\alpha \frac{dx}{x} +\frac{dy}{y})\wedge du_1\wedge\cdots\wedge  du_q=0\,.
\]
Since the differential form $-\alpha \frac{dx}{x} +\frac{dy}{y}$ in $\C^2$  posseses only constant holomorphic first integrals,  $A$ does not depend on the variable $x$. It  extends trivially to a holomorphic map defined on $\{|x|\leq \varepsilon, |u_j|\leq \eta, j=1,\ldots,q\}$. 
Thus the automorphism   $\br\chi\circ\Psi\circ\chi^{-1}$  extends to a neighborhood of the origin in $\C^{q+2}$, as a holomorphic automorphism $\un\Psi$  leaving  $\mc L$ invariant. We conclude that the desired extension of $\Psi$ is given by $\br\chi^{-1}\circ \un\Psi\circ\chi$.\\

\noindent\textit{-Step 5: Extension  at a nodal singular point  $s\in \mr{Sing}(\mc E_{u_0})$.} % 
If  by Step 3  we extend at a such a point $s$  the homeomorphisms   along the components $D$ and $D'$ meeting at $s$ constructed in Step 1,  we obtain two germs at $s$ of biholomorphisms  $\Psi$ and $\Psi'$ that do not  fulfill  the requested property (\ref{collagenoeuddic}). Thanks to the following remark, we  modify them so that they coincide as germs at $s$. 
\begin{obs}\label{collanodal}
Let $g^j : \overline{\mb D}^2_1\times \mb D_{\eta}^q \iso \mc W_j$, $j=1,2$,
be two biholomorphisms leaving invariant the linear foliation $\mc L$ defined by (\ref{linearsystemdif}), such that 
\begin{enumerate}%[(i)]  
\item\label{0fixe} $g^j(x,y,0)=(x,y,0)$, 
\item\label{fibrationinv} $g^1(x,y,u)=(x,g^1_2(x,y,u),u)$, with  $g^1_2(x,0,u)=0$,
\item\label{idemautreaxe}  $g^2(x,y,u)=(g^2_1(x,y,u),y,u)$ with $g^2_1(0,y,u)=0$,
\end{enumerate}
where $\mb D_\eta=\{|z|<\eta\}\subset\C$. Then for $\eta>0$ small enough, there are  suitable real numbers $0<C_1<C_2<1<C_2'<C'_1$ such that 
there exists a  homeomorphism germ    
\[g : \overline{\mb D}_1^2\times \mb D_{\eta}^q \iso \overline{\mb D}_1^2\times \mb D_{\eta}^q\,,
\quad
(x,y,u)\mapsto (g_1(x,y,u), g_2(x,y, u), u)
\]
satisfying also Properties (\ref{0fixe})-(\ref{idemautreaxe}) above, 
that is equal to $g_1$ when $|y| \leq C_1\,|x|^\alpha$, to $g_2$ when  $|y|\geq C'_1 |x|^\alpha$ and to the identity map when $C'_2|x|^\alpha<|y|  \leq C_2\,|x|^\alpha$.
\end{obs}
\begin{dem2}{of the remark}
As we have seen in  Step 4, the invariance of a linear foliation by these biholomorphisms implies that $g^1_2(x,y,u)=A_1(u)y$ and $g^2_1(x,y,u)=A_2(u)x$ with $A_1$, $A_2 : \overline{\mb D}_{\eta}^q\to\C^\ast$ holomorphic functions.  
 Let us choose the real numbers $C_j$, $C'_j$, $j=1,2$ so that 
 \[C'_1 >  \sup \{ |A_2(u)|\;;\;u\in \mb D_\eta^q\}\cdot C'_2>
 C_2 >  \sup \{ |A_1(u)|\;;\;u\in \mb D_\eta^q\}\cdot C_1\,.
 \]
 The  continuous functions 
\[
B:\{1\}\times\overline{\mb D}_{1}\times \mb D_{\eta}^{q}\to \{1\}\times\mb \overline D_{1}\times \mb D_{\eta}^q\,,
\quad
B(1,y,u)=(1,R(|y|,u)\,e^{i\theta(|y|,u)},u)\]
defined by the following interpolation
\begin{itemize}
\item  $R(r,u)= |A(u)| r$, if $r\leq C_1$,
\item $R(r,u)= \frac{C_2-C_1|A(u)|}{C_2-C_1}\cdot(r -C_1)+ C_1 |A(u)|$, if $C_1\leq r\leq C_2$,
\item $R(r,u)= r$,   if $C_2<r<1$
\item $\theta(r,u)=\arg z +\arg A(u)$,  if $r\leq C_1$,
\item $\theta(r,u)= \arg z+\frac{-\arg A(u)}{C_2-C_1}\cdot(r -C_1)+ \arg A(u)$, if $C_1\leq r\leq C_2$, 
\item $\theta(r,u)= \arg z$,   if $C_2<r<1$,
\end{itemize}
is a homeomorphism that is equal to $g^1(1,y,u)$  if $|y|<C_1$ and to the identity map   if $|y|>C_2$. 
On each line  $L_{u}=\{1\}\times\mb D_{1}\times\{u\}$ the holonomy map of $\mc L$, which is a rotation,  commutes with the restriction   $B|_{L_{u}}$. Therefore $B$ 
extends in a unique way to a homeomorphism  defined on the open set
\[
\{|y| < |x|^\alpha\}
\subset  \mb D_{1}^2\times \mb D_{\eta}^q
\]
 obtained by saturation of $\{1\}\times\mb D_{1}\times \mb D_{\eta}^{q}$ by  $\mc L$. This homeomorphism  leaves $\mc L$  invariant and fixes each line $\{x\}\times\mb D_1\times\{\wt y\}$.
Thanks to the uniqueness of this   extension it  is equal to $\br\chi\circ\Psi\circ\chi^{-1}$ on $\{ |y|<C_1|x|^\alpha \}$ and by construction it is equal to the identity map on $\{ C_2|x|^\alpha < |y| < |x|^\alpha \}$. 
It extends trivially as the identity map on $\{ C_2|x|^\alpha < |y|  \}$. Finally the obtained extension  is a homeomorphism  $G:\mb D_1^2\times \mb D_{\eta}^q\iso \mb D_1^2\times \mb D_{\eta}^q$ with support in $\{|y|\leq C_2|x|^\alpha\} $, again equal to   $g^1$ on $\{ |y|<C_1|x|^\alpha \}$. 

Performing the same construction along the $y$ axis we end up with a homeomorphism  $G':\mb D_1^2\times \mb D_{\eta}^q\iso \mb D_1^2\times \mb D_{\eta}^q$ with support in $\{|y|\geq C'_2|x|^\alpha\}$,  equal to   $\br\chi\circ\Psi'\circ\chi^{-1}$ on $\{ |y|>C'_1|x|^\alpha \}$. The supports of $G$ and $G'$  being disjoints, the homeomorphism  $g:= G\circ G'=G'\circ G$ fulfills the required properties.
\end{dem2}
This achieves the proof of Theorem~\ref{psiD}.
\end{proof}

\subsection{Deformation functor}\label{subsecDefSpaFunct} Let us consider   the \emph{pointed set} 
\begin{equation*}\label{notDef}
\mr{Def}_{\F}^{Q^\point}:=\left.\{[\F_{Q^\point}]\;:\;\F_{Q^\point}\ \hbox{equisingular deformation of }  \F\}\right/ \approx_{\mc C^\mr{ex}}
\end{equation*}
of all \Cex-conjugacy classes $[\F_{Q^\point}]$ of germs of equisingular deformations $\F_{Q^\point}$ over $Q^\point$ of a fixed foliation $\F$. This set is pointed by the class of the constant deformation.\\

The assignment $Q^\point\mapsto\Def_{\F}^{Q^\point}$  is a  contravariant functor, because according to Remark~\ref{pullbackmorphism},  
 to a germ $\mu:P^\point\to Q^\point$ corresponds the well defined pull-back map
 \[
 \mu^*:\Def_{\F}^{Q^\point}\to\Def_{\F}^{P^\point}\,,\quad [\F_{Q^\point}]\mapsto [\mu^\ast \F_{Q^\point}]\,.
  \]
\begin{teo}\label{famconsigma}
Let  $\phi : (\C^2,0) \iso (\C^2, 0)$ be a homeomorphism  germ that is a \Cex-conjugacy  between two germs of  foliations $\G$ and $\F=\phi(\G)$ which are generalized curves.  Let $Q^\point=(Q,u_0)$ be a germ of manifold. Then there exists a  bijective map 
\[
\phi^\ast  :  {\Def}_{\F}^{Q^\point} \stackrel{\sim}{\longrightarrow}
{\Def}_{\mc {G}}^{Q^\point}
\]
defined by the following property:
\begin{enumerate}
\item[$(\star)$]$\phi^\ast([\F_{Q^\point}])=[\G_{Q^\point}]$ if and only if there exists a germ of homeomorphism over~$Q$
\[ 
\Phi : (\C^2\times Q,(0,u_0))\iso (\C^2\times Q,(0,u_0))\,,\quad \mr{pr}_Q\circ\Phi=\mr{pr}_Q\,,
 \]
that sends the leaves of $\G_{Q^\point}$ on that of $\F_{Q^\point}$, is excellent, and satisfies 
\[\Phi(x,y,u_0)=(\phi(x,y),u_0)\,.\]
\end{enumerate}
Moreover, if $\psi:(\C^2,0)\iso(\C^2,0)$, $\psi(\mc K)=\G$,  is a \Cex-conjugacy between a germ of foliation $\mc K$ and $\G$, then 
\begin{equation}\label{fonctsigmaphi}
(\phi\circ\psi)^\ast=\psi^\ast\circ\phi^\ast:
 {\Def}_{\F}^{Q^\point} \iso
{\Def}_{\mc K}^{Q^\point}\,.
\end{equation}
\end{teo}
\begin{proof}
Under the hypothesis of the theorem, let us consider a class $\mf c\in \mr{Def}_{\mc {F}}^{Q^\point}$ and an equisingular  deformation 
$\F_{Q^\point}$ of $\F$ in $\mf c$. 
In a first step we will construct an equisingular deformation $\G_{Q^\point}$ of $\G$ and a \Cex-homeomorphism
$\Phi$ satisfying $\Phi(\G_{Q^\point})=\F_{Q^\point}$, such that $\Phi\circ\iota=\iota\circ\phi$, with $\iota:\C^2\hookrightarrow \C^2\times Q$, $\iota(x,y):=(x,y,u_0)$. Then in a second step we will verify that the class $[{\G}_{Q^\point}] \in \mr{Def}_{\G}^{Q^\point}$ does not depend on the choice of the deformation ${\F}_{Q^\point}$ in $\mf c$. Finally in a third step we check that the map $\phi^\ast  $ that associate to each class $\mf c=[\F_{Q^\point}]\in \mr{Def}_{\mc {F}}^{Q^\point}$ the class of the deformation $\G_{Q^\point}$ defined  in the first step, fulfills the property $(\star)$ and the functorial relation.\\

\textit{Step 1. } We again  denote by $\iota^\sharp$ 
the lifting (\ref{imcanlifted}) of $\iota$ 
through the reduction and equireduction maps $E_\F$  and $E_{\F_{Q^\point}}$,   by $j^\sharp:M_\F\hookrightarrow M_\F\times Q$ the lifting of $\iota$ through 
through $E_\F$ and $E_{\F^{\mr{ct}}_{Q^\point}}$,  that is  $j^\sharp(m):=(m,u_0)$, and finally by
\[ 
\phi^\sharp:(M_\G,\mc E_\G)\to(M_\F,\mc E_\F)
\,,\quad
E_\F\circ\phi^\sharp=\phi\circ E_\G, 
\]
the lifting of $\phi$ through $E_\F$ and the reduction map $E_\G:(M_\G,\mc E_\G)\to(\C^2,0)$ of $\G$. The following homeomorphism
\begin{equation*}\label{phiQlift}
\phi^\sharp_{Q^\point} : 
 (M_\G\times Q, \mc E_{\G}\times \{u_0\}) 
\longrightarrow
(M_\F\times Q, \mc E_{\F}\times \{u_0\})\,, \quad
(m,u)\mapsto (\phi^\sharp(m),u)\,,
\end{equation*}
is excellent and sends the reduced constant foliation $\G_{Q^\point}^{\mr{ct}\,\sharp}$ over $Q^\point$ with special fiber $\G^\sharp$, to the constant foliation $\F_{Q^\point}^{\mr{ct}\,\sharp}$.
According to Theorem~\ref{psiD}, let us fix a good trivializing system for $\F_{Q^\point}$
\[
\Psi_D : 
(M_{\F_{Q^\point}}, \iota^\sharp(D)) \iso (M_\F\times Q,D\times \{u_0\})\,,\quad
\Psi_D(\F_{Q^\point}^\sharp)=\F_{Q^\point}^{\mr{ct}\,\sharp}
\,,
\quad
\Psi_D\circ \iota^\sharp=j^\sharp\,,
\]
 indexed by the irreducible components $D$  of $\mc E_\F$.  
At the intersection points
$ \{s_{DD'}\}:= (D\cap D')\times \{u_0\}$, $D\cap D'\neq \emptyset$, 
the \emph{cocycles} 
\begin{equation}\label{cocyclerecolmnt}
\Phi_{DD'}:= (\phi^\sharp_{Q^\point})^{-1}\circ \Psi_{D}\circ\Psi_{D'}^{-1}\circ \phi_{Q\point}^{\sharp} : (M_\G\times Q, s_{DD'}) \iso (M_\G\times Q, s_{DD'})
\end{equation}
are germs of biholomorphisms over $Q^\point$ fulfilling the properties
\begin{equation*}\label{relgluing}
\Phi_{DD'}(\G_{Q^\point}^{\mr{ct}\,\sharp})=\G_{Q^\point}^{\mr{ct}\,\sharp}\,,
\quad
\Phi_{DD'}\circ j^\sharp=j^\sharp\,.
\end{equation*}
Indeed according to (\ref{locholexcept}) and (\ref{collagenoeuddic}) in Theorem \ref{psiD}, if the intersection point $D\cap D'$ is not a nodal singular point of $\F^\sharp$, the germs of $\phi^{\sharp}_{Q^\point}$ at the point $s_{DD'}$ and of  $\Psi_{D'}$ at $\Psi_D^{-1}(s_{DD'})$ are holomorphic; otherwise, at $\Psi_D^{-1}(s_{DD'})$ the germs $\Psi_D$ and  $\Psi_{D'}$ coincide and $\Phi_{DD'}$ is the identity map. 

Let us consider the manifold  germ%over $Q^\point$ 
\[
( N,{\mc E}'_\G) := 
\left.\sqcup_D (M_\G\times Q, D\times\{u_0\})\right/
(\Phi_{DD'})\,,\quad
\theta : ( N, {\mc E}'_\G)\to Q^\point\,,
\]
obtained by gluing neighborhoods  in $M_\G\times Q$ of the irreducible components $j^\sharp(D)$ using these cocycles, and endowed with the germ of  submersion 
$\theta$  obtained by gluing the germs of the canonical projection $\mr{pr}_Q:(M_\G\times Q, D\times \{u_0\})\to Q$. Since  $\Phi_{DD'}$ are the identity on the special fiber $M_{\G}\times \{u_0\}$,  $j^\sharp$ induces an embedding 
\[ \Delta :(M_\G,\mc E_\G) \hookrightarrow (N, {\mc E}'_{\G})
\]
that is a biholomorphism germ onto  $(\theta^{-1}(u_0),\mc E'_{\G})$. The gluing maps  leaving  invariant the constant foliation $\G^{\mr{ct}\,\sharp}_{Q^\point}$, they define in the ambient space $(N, {\mc E}'_\G)$ a foliation germ   $\G'_{Q^\point}$  tangent to the fibers of $\theta$,  that coincides with $\Delta(\G^\sharp)$ on $\theta^{-1}(u_0)$. Thanks to the relations 
$\Psi_{D'}^{-1}\circ\phi_{Q^\point}^{\sharp}\circ\Phi_{DD'}^{-1}=\Psi_{D}^{-1}\circ\phi_{Q^\point}^{\sharp}$ given by (\ref{cocyclerecolmnt}), the collection of homeomorphisms
\[\Phi_D:=\Psi_D^{-1}\circ\phi_{Q^\point}^\sharp:
(M_\G\times Q,j^\sharp(D))\to(M_{\F_{Q^\point}},\iota^\sharp(D)),\quad \Phi_D(\G_{Q^\point}^{\mr{ct}\,\sharp})=\F_{Q^\point}^\sharp,\]
glue as a homeomorphism over $Q^\point$
\[\Phi' : (N,\mc E'_{\G})\iso (M_{\F_{Q^\point}}, \iota^\sharp(\mc E_{\F}))\,,
\quad
\mr{pr}_Q\circ \Phi'=\theta\,,
\]
that send the leaves of $\G'_{Q^\point}$ to that of $\F^\sharp_{Q^\point}$. As the maps $\phi^\sharp$ and $\Psi_D$, this map is excellent in the meaning that it is  also holomorphic at the  non-nodal points of the corresponding foliation. It satisfies:
\begin{equation}\label{Phi'Delta}
\Phi'\circ \Delta=\iota^\sharp \circ \phi^\sharp\,;
\end{equation}
On the other hand, the preimage by  $\Phi'$ of the exceptional divisor $\mc E_{\F_{Q^\point}}:=E_{\F_{Q^\point}}^{-1}(\{0\}\times Q)$ is an hypersurface $\mc E_Q$ which is also  exceptional  in $N$ (see \cite[p. 306]{M}): there is a holomorphic map  germ 
\[ C:(N,\mc E'_\G)\to (\C^2\times Q,(0,u_0))\quad 
\text{such that}\quad 
\mr{pr}_Q\circ C=\theta \,,\quad
C(\mc E_Q)=\{0\}\times Q\,,
\]
that is a biholomorphism from complementary of $\mc E_Q$ to the complementary of $\{0\}\times Q$. This last property allows to define a germ of holomorphic foliation $\G_{Q^\point}$ on $(\C^2\times Q,(0,u_0))$, that is the direct image  of $\G'_{Q^\point}$ by $C$. 
Up to perform an additional biholomorphism we also require that $\Delta$ contracts  to the embedding $\iota$, i.e. $C\circ \Delta=\iota\circ E_\G$, so that 
\[ 
\G_{Q^\point|\C^2\times \{0\}} = C(\G'_{Q^\point|\theta^{-1}(u_0)})=C(\Delta(\G^\sharp))=\iota(E_{\G}(\G^\sharp))=\iota(\G).
 \]
In other words, $\G_{Q^\point}$ is a deformation of $\G$. By construction this deformation is equisingular and more precisely there is a biholomorphism germ 
\[ 
F: (N,\mc E'_\G)\iso ( M_{\G_{Q^\point}},\mc{ E}'_{u_0})\,,\quad
 \]
such that 
\begin{equation}\label{FDelta}
E_{\G_{Q^\point}}\circ F=C\,,\quad
F({{\G}'_{Q^\point}})=\G_{Q^\point}^\sharp\,,
\quad
F\circ\Delta =k^\sharp\,,
\end{equation}
$k^\sharp$ being the lifting of $\iota$ through the reduction map $E_\G$   and the equireduction map $E_{ \G_{Q^\point}}:(M_{\G_{Q^\point}},\mc{ E}'_{u_0}) \to (\C^2\times Q,(0,u_0))$ of the deformation $\G_{Q^\point}$,
\[\xymatrix{&(N,\mc E'_\G)\ar[dr]^C\ar[dd]_F&\\
(M_\G,\mc E_\G)\ar[ur]^{\Delta}\ar[dr]^{k^\sharp}&&(\C^2\times Q,(0,u_0))\;.\\
&(M_{\G_{Q^\point}}\ar[ur]_{E_{\G_{Q^\point}}},\mc E_{u_0}')&
}
\]
Now let us notice that since $C(\mc E_Q)=\{0\}\times Q$, the homeomorphism  germ $\Phi'$ contracts through $C$ and $E_{\F_{Q^\point}}$to a germ of map
\[ 
\Phi:(\C^2\times Q,(0,u_0))\to (\C^2 \times Q,(0,u_0))\,,\quad E_{\F_{Q^\point}}\circ\Phi'=\Phi\circ C\,, \]
 that by construction is a germ of homeomorphism satisfying:
  \[ 
  \mr{pr}_Q\circ \Phi=\mr{pr}_Q\,,\quad
  \Phi(\G_Q)=\F_Q\,,\quad
  \Phi\circ \iota=\iota\circ \phi\,.
  \]
  To achieve the Step 1, it remains to check that $\Phi$ is excellent. 
Indeed, $\Phi'\circ F^{-1}$ is a 
lifting of $\Phi$,
\[
\Phi'\circ F^{-1}:(M_{ \G_{Q^\point}},\mc E'_{u_0})\to (M_{\F_{Q^\point}},\mc E_{u_0})\,,\quad 
E_{\F_{Q^\point}}\circ (\Phi'\circ F^{-1})=\Phi\circ C\circ F^{-1}=\Phi\circ E_{\G_{Q^\point}}\,.
\]
Since $\Phi'$ is excellent we deduce that $\Phi$ is also  excellent.\\

  \textit{Step 2. } Notice first that up to $\mc C^{\mr{ex}}$-conjugacy the deformation $\G_{Q^\point}$ obtained by this construction does not depend on the choice of the good trivializing system $(\Psi_D)_D$. If $(\br{N},\br{\mc E}_\G')$, 
 $\br{\G}_{Q^\point}' $ and $\br{\G}_{Q^\point}$ are similarly obtained from another
 good trivializing system   $(\br{\Psi}_D)_D$  then the homeomorphisms $\Psi_D\circ\br{\Psi}_D^{-1}:(M_\F\times Q, D\times\{u_0\})\to(M_\F\times Q, D\times\{u_0\})$ glue to an excellent homeomorphism that conjugates $\br{\G}'_{Q^\point}$ and $\G_{Q^\point}'$ and contracts to an excellent conjugacy between the deformations $\br{\G}_{Q^\point}$ and $\G_{Q^\point}$ of $\G$.

Now let us show that $[\G_{Q^\point}]$ does not depend on the choice of the representative $\F_{Q^\point}$ of $\mf c\in\mr{Def}_\F^{Q^\point}$.
Let  ${\br{\mc F}}_{Q^{\point}}$ be another representative of  $\mf c$,
$\br{\G}_{Q^\point}$ a deformation of $\G$ and  $\br{\Phi}:(\C^2\times Q,(0,u_0))\to (\C^2\times Q,(0,u_0))$ a germ of excellent homeomorphism such that 
$\br{\Phi}({\br{\G}}_{Q^\point})={\br{\F}}_{Q^\point}$, $\mr{pr}_Q\circ\br{\Phi}=\mr{pr}_Q$ and  $\br{\Phi}\circ{\iota} ={\iota}\circ \phi$. 
Then $\br{\G}_{Q^\point}$ is $\mc C^\mr{ex}$-conjugated to $\G_{Q^\point}$. Indeed,
if $\xi$ is an \Cex-homeomorphism such that $\xi({\br{\mc F}}_{Q^{\point}})={\mc F}_{Q^{\point}}$ and  $\xi\circ{\iota}=\iota$, then the \Cex-homeomorphism $\Upsilon:=\Phi^{-1}\circ\xi\circ\br{\Phi}$  trivially satisfies $\Upsilon({\br{\mc G}}_{Q^{\point}})={{\mc G}}_{Q^{\point}}$ and $\Upsilon\circ {\iota}=\iota$.
This implies that the map $\phi^*$ is well-defined.\\

\textit{Step 3. } The direct implication of $(\star)$ is clear. To see the converse, we apply the previous argument to the case $\br{\F}_{Q^\point}=\F_{Q^\point}$.
The functorial relation follows directly from $(\star)$ and if $\phi=\mr{id}_{\C^2}$ then $\phi^*$ is the identity map on $\mr{Def}^{Q^\point}_\F$. This implies that $\phi^*$ is bijective and $(\phi^*)^{-1}=(\phi^{-1})^*$.
\end{proof}

We check that for any holomorphic map germ $\mu : P^\point\to Q^\point$ and any  deformation $ \G_{Q^\point}\in \phi^\ast([\F_{Q^\point}])$  we have:
\begin{equation*}
\phi^\ast([ \mu^\ast \F_{Q^\point}  ])=[ \mu^\ast \G_{Q^\point}]\,,
\end{equation*}
i.e. the following diagram is commutative:
\begin{equation}\label{comsigmalambda}
\xymatrix{\mr{Def}^{Q^\point}_\F\ar[r]^{\phi^*}\ar[d]^{\mu^*}&\mr{Def}^{Q^\point}_\G\ar[d]^{\mu^*}\\ \mr{Def}^{P^\point}_\F\ar[r]^{\phi^*}&\mr{Def}_\G^{P^\point}\,.}
\end{equation}

\begin{lema}\label{3.13}
Under the assumptions of Theorem~\ref{famconsigma},
if $\mu:P^\point\to Q^\point$ and $\lambda:R^\point\to P^\point$ are holomorphic maps between germs of manifolds, $\phi:\G\to\F$ and $\psi:\mc K\to\G$ are $\mc C^{\mr{ex}}$-conjugacies 
and if we write 
\[
(\mu, \phi)^\ast :=\phi^\ast \circ\mu^\ast: 
\mr{Def}^{Q^\point}_{\F}\to \mr{Def}^{P\point}_{\G}\,,
\quad
\]
then we have
$(\lambda, \psi)^\ast\circ(\mu, \phi)^\ast=(\mu\circ\lambda, \phi\circ \psi)^\ast$.
\end{lema}
\begin{proof}
It suffices to check that the following diagram is commutative using  (\ref{fonctsigmaphi}),  diagram~(\ref{comsigmalambda}) and Remark \ref{pullbakconjugacy},
\begin{equation}\label{def-fun}
\xymatrix{
\Def_{\F}^{Q^\point}\ar[d]_{\mu^*}\ar@/_3pc/[dd]_{(\mu\circ\lambda)^*}\ar[rd]^{(\mu, \,\phi)^\ast}\ar@/^4pc/[rrdd]^{(\mu\circ\lambda,\,\phi\circ\psi)^\ast}& &\\ 
\Def_{\F}^{P^\point}\ar[d]_{\lambda^*}\ar[r]^{\phi^\ast}& 
\Def_{\G}^{P^\point}\ar[d]_{\lambda^*}\ar[rd]^{(\lambda,\, \psi)^\ast}& \\
\Def_{\F}^{R^\point}\ar[r]^{\phi^\ast}\ar@/_2pc/[rr]_{(\phi\circ\psi)^\ast}&
\Def_{\G}^{R^\point}\ar[r]^{\psi^\ast}&\Def_{\mc K}^{R^\point}}
\end{equation}
\end{proof}

Let us denote now by 
\begin{itemize}
\item $\fol$ the category  whose objects are the germs of foliations on $(\C^2,0)$  which are \emph{generalized curves} and whose morphisms $\phi : \G\to \F$ are the germs of \emph{\Cex-conjugacies}, $\phi(\G)=\F$;
\item $\mbf{Set}^\point$ the category of \emph{pointed sets} whose objects are the pairs $(A,a)$ formed by  a set  and a point of this set, the morphisms $F : (A,a)\to (B,b)$ being  maps  from $A$ to $B$ such that   $F(a)=b$;
\item $\mbf{Man}^{\point}$  the subcategory of  $\mbf{Set}^\point$, consisting of pairs $(A,a)$ with $A$ endowed with a complex manifold structure, the morphisms  being 
holomorphic pointed sets morphisms
$\mu:P^\point\to Q^\point$.
\end{itemize}

\begin{defin}\label{def411} The \emph{deformation functor}  is the contravariant
functor
\[
\Def :\mbf{Man^\point}\times\fol\to \mbf{Set}^\point\,,
\quad
(Q^\point, \F)\mapsto \Def_{\F}^{Q^\point}
\]
defined by associating to any morphism $(\mu, \phi) :(P^\point, \G)\to (Q^\point, \F), $ the \emph{pull-back map} 
\[
(\mu,\phi)^\ast :   \Def_{\F}^{Q^\point}\to \Def_{\G}^{P^\point}
 \,,\quad
 [\F_{Q^\point}]\mapsto
  \phi^\ast(\mu^\ast([\F_{Q^\point}  ]))= \phi^\ast([ \mu^\ast \F_{Q^\point}  ])\,.
\] 
The fact that $\Def$ is a functor follows from Lemma~\ref{3.13}.
\end{defin}
As  a direct consequence of Theorem \ref{famconsigma},     
if $[\G_{P^\point}]=(\mu, \phi)^\ast([\F_{Q^\point}])$ with $\mu : P^\point:=(P,t_0)\to Q^\point$,  then for $t\in P$ sufficiently close to $t_0$ the foliations $\G_{P^\point}|_{\C^2\times\{t\}}$ and  $\F_{Q^\point}|_{\C^2\times\{\mu(t)\}}$ are \Cex-conjugated.

\color{black}

 \section{Group-graphs of automorphisms and  transversal symmetries}\label{sectAutSymGrGr}

 \subsection{Group-graph of \Cex-automorphisms}\label{sec:aut}  Given a  foliation 
 $\F$ and  a germ of manifold $Q^\point=(Q,u_0)$,  let us consider   the following  sheaf $\un{\mr{Aut}}_{\F}^{Q^\point}$ over  the exceptional divisor $\mc E_{\mc \F}$ of the reduction of $\F$: if  $U$ is an open subset of $\mc E_\F$, then $\underline{\aut}^{Q^\point}_{\F}(U) $ is the group of germs along $U\times \{u_0\}$ of \Cex-homeomorphisms over $Q^\point$
\[\Phi:(M_{\F}\times Q, U\times \{ u_0\})\longrightarrow (M_{\F}\times Q,U\times \{ u_0\})
\]
leaving invariant the constant  family $\F^{\sharp\, \mr{ct}}_{Q^\point}$ with fiber the reduced foliation $\F^\sharp$ and moreover being  the identity map on the special fiber $M_\F\times \{u_0\}$. The same definition works when $U$ is   not open in $\mc E_\F$ and in that case $\underline{\aut}^{Q^\point}_{\F}(U) $ coincides with the inductive limit of $\underline{\aut}^{Q^\point}_{\F}(V) $ for $V$ open subset of $\mc E_\F$ containing $U$, cf. Section~\ref{gg-sh}.
The property  ``excellent'' means here that at  each  point $m$ in an invariant component of $\mc E_\F$ the germ $\Phi_m$ 
of  $\Phi$ is a  holomorphic germ if $m\in \mr{Sing}(\mc E_\F)\cup\mr{Sing}(\F^\sharp)$,  except perhaps if  $m$  is a nodal singularity of $\F^\sharp$ belonging to $\mr{Sing}(\mc E_\F)$, and that $\Phi_m$ is transversely holomorphic if $m$ is a regular point of $\F^\sharp$. 
According to  \cite{CamachoRosas} if $D$ is an invariant component of $\mc E_\F$ and if one saturates by  $\F^{\sharp}$  a neighborhood of $\mr{Sing(\F^\sharp})\cap D$, one obtains  a set that contains all the regular points of $\F^\sharp$ in $D$. Therefore when $U$ contains $D$, the  above transversal holomorphy property  is automatically induced by the holomorphy at the singular points; for this reason we did not need  to require it in  Definition~\ref{excellent} of \Cex-conjugacy.

\begin{defin}\label{defAutQ}
We call \emph{group-graph  of automorphisms over $Q^\point$ of $\F$} and we denote by $\mr{Aut}_{\F}^{Q^\point}$  the following group-graph over the \emph{dual graph $\A_\F$ of $\mc E_{\mc \F}$}:
\begin{enumerate}[(i)]
\item\label{DnD}  $\mr{Aut}_{\F}^{Q^\point}(D)=\un{\mr{Aut}}_{\F}^{Q^\point}(D)$, if $D\in \Ve_{\A_\F}$ is invariant;
\item\label{Ddicr} $\mr{Aut}_{\F}^{Q^\point}(D)=\{I_D\}$, if $D\in \Ve_{\A_\F}$  is  dicritical\,;
\item\label{endnn} $\mr{Aut}_{\F}^{Q^\point}(\msf e)$ is the stalk $\un{\mr{Aut}}_{\F}^{Q^\point}(s)$  of the sheaf $\un{\mr{Aut}}_{\F}^{Q^\point}$  at the point $s$ defined by  $\msf e=\langle D,D'\rangle$,  $D\cap D'=\{s\}$, if  $s$ is neither a regular point nor a nodal singular point  of $\F^\sharp$;
\item\label{trivialedge} $\mr{Aut}_{\F}^{Q^\point}(\msf e)=\{I_{\msf e}\}$, if $\msf e=\langle D,D'\rangle$,  $D\cap D'=\{s\}$ and $s$ is either a regular point or a nodal singular point  of $\F^\sharp$;
\item the restriction map $\rho^{\msf e}_{D} : \mr{Aut}_{\F}^{Q^\point}(D)\to\mr{Aut}_{\F}^{Q^\point}(\msf e)$ is the restriction map of the sheaf  $\un{\mr{Aut}}_{\F}^{Q^\point}$ when $D$ is invariant and $\msf e$ fulfills condition (\ref{endnn}); $\rho^{\msf e}_{D}$ is  the trivial map $\mr{Aut}_{\F}^{Q^\point}(D)\to\{I_{\msf e}\}$ otherwise;
\end{enumerate}
where $I_D$, resp. $I_{\msf e}$, denotes the germ along $D\times\{u_0\}$, resp. at the point $(s,u_0)$, of the identity map  $\mr{id}_{M_\F\times Q}$.
\end{defin}

\begin{obs}\label{suppAut} 
Notice that restricted  to its support, see (\ref{defsuppgrgr}), $\mr{Aut}_{\F}^{Q^\point}$ coincides with the group-graph associated to the sheaf $\un{\mr{Aut}}_{\F}^{Q^\point}$  defined in Section \ref{gg-sh}.
The elements of $\Ve_{\A_{\F}}\cup\Ed_{\A_{\F}}$ not belonging to this support   are exactly the elements given by (\ref{Ddicr}) and (\ref{trivialedge}): the vertices that are dicritical components of $\mc E_\F$,  the edges $\langle D,D'\rangle$ with $D$ or $D'$ dicritical and the edges $\langle D,D'\rangle$ for witch  $\F^\sharp$ has a nodal singularity at the  point $D\cap D'$.  Clearly $\mr{supp}({\mr{Aut}}_{\F}^{Q^\point})$  is a sub-graph of $\mc E_\F$ called \emph{cut-graph of $\F$}.  We denote by 
$\mr{supp}({\mr{Aut}}_{\F}^{Q^\point})=
\bigsqcup_{\alpha\in \mc A} \A_{\F}^\alpha$  its decomposition  into connected components which we call  \emph{cut-components of $\A_\F$}. We have:
\begin{equation}\label{h1decSoDir}
H^1(\A_\F,{\aut}^{Q^\point}_{\F})=
\prod_{\alpha\in\mc A} H^1(\A_{\F}^\alpha,{\aut}^{Q^\point}_{\F})\,.
\end{equation}
This decomposition,  produced by the points (\ref{Ddicr}) and (\ref{trivialedge}) and Remark~\ref{aretes-support} in the above definition, may seem artificial. However the cocycles $(\Psi_{D}\circ\Psi_{D'}^{-1})$ that we will consider are constructed using  good trivializing systems $(\Psi_D)_D$ provided by Theorem~\ref{psiD}. Consequently the property (\ref{collagenoeuddic}) of 
that theorem
guarantees that  $\Psi_{D}\circ\Psi_{D'}^{-1}$ is trivial when $D$ or $D'$ is dicritical or when $\F^\sharp$ has a nodal singularity at $D\cap D'$.
\end{obs}

 %\end{document}

Now let us consider    a germ of \Cex-homeomorphism $\phi: (\C^2,0)\iso(\C^2,0)$ which  conjugates two   foliations $\G$ and  $\F$, $\phi(\G)=\F$, and the corresponding  \Cex-conjugacy 
\begin{equation*}\label{phiQ}
\phi^\sharp_{Q^\point} :(M_\G\times Q,\mc E_\G\times \{u_0\})
\iso
(M_\F\times Q,\mc E_\F\times \{u_0\}) \,,
\quad
(p,u)\mapsto(\phi^\sharp(p),u)\,,
\end{equation*}
between the contant families  $\G_{Q^\point}^{\mr{ct}\,\sharp}$ and $\F_{Q^\point}^{\mr{ct}\,\sharp}$. 
Let us denote by $\phi_{\ltE}:\mc E_\G\to\mc E_\F$ the restriction of  $\phi^{\sharp}$ to the exceptional divisors. 
If $U\subset \mc E_\G$ is an open set and   $\Phi$ belongs to $\phi_{\ltE}^{-1}\underline{\aut}^{Q^\point}_{\F}(U)=\underline{\aut}^{Q^\point}_{\F}(\phi_{\ltE}(U)) $, then $\phi_{Q^\point}^{\sharp\, -1}\circ\Phi\circ\phi_{Q^\point}^{\sharp}$   belongs to 
$\underline{\aut}^{Q^\point}_{\G}(U)$. 
As described in Section~\ref{gg-sh},
the homeomorphism $\phi_{\ltE}$ induces an isomorphism between the dual graphs of $\mc E_\G$ and $\mc E_\F$
\begin{equation}\label{Aphi}
\A_\phi : \A_\G\to \A_\F\,,\quad
D\mapsto \phi_{\ltE}(D)\,,\quad
\langle D,D'\rangle\mapsto \langle \phi_{\ltE}(D),\phi_{\ltE}(D')\rangle\,. 
\end{equation}
We thus obtain the following isomorphism of group-graphs over $\A_\phi$:
\begin{equation*}\label{transportAut}
\phi^\ast: 
{\aut}^{Q^\point}_{\F}\to {\aut}^{Q^\point}_{\G}\,,
\end{equation*}
\[
{\aut}^{Q^\point}_{\F}(\A_\phi(\star))\ni\Phi
\;\;\mapsto \;\;
\phi_{Q^\point}^{\sharp\, -1}\circ\Phi\circ\phi_{Q^\point}^{\sharp}
\in {\aut}^{Q^\point}_{\G}(\star) \,,
\quad
\star\in \Ve_{\A_\F}\cup\,\Ed_{\A_\F}\,.
\]
On the other hand let $\mu: P^\point\to Q^\point$ be  a holomorphic map  between germs of manifolds. The pull-back being a functor and, by definition, $\F_{Q^\point}^{\mr{ct}\,\sharp}$ being the pull-back by a constant map, it follows:
\[
\mu^*\F_{Q^\point}^{\mr{ct}\,\sharp}=\F_{P^\point}^{\mr{ct}\,\sharp}\quad\hbox{ and } \quad\mu^*\phi_{Q^\point}^\sharp=\phi_{P^\point}^\sharp\,.
\]
Thus we  have 
the equality $\mu^\ast(
\phi_{Q^\point}^{\sharp\, -1}\circ\Phi\circ\phi_{Q^\point}^{\sharp}
)=
\phi_{P^\point}^{\sharp\, -1}\circ\mu^*\Phi\circ\phi_{P^\point}^{\sharp}
$.  We finally obtain the following commutative diagram of group-graph morphisms

\begin{equation}\label{sq2}
\xymatrix{\aut_{\F}^{Q^\point}\ar[d]_{\mu^*}\ar[r]^{\phi^\ast}&\aut_{\G}^{Q^\point}\ar[d]^{\mu^*}\\ \aut_{\F}^{P^\point}\ar[r]^{\phi^\ast}&\;\aut_{\G}^{P^\point}\,.}
\end{equation}

Using the relations $\phi^\ast\circ\psi^\ast=(\psi\circ\phi)^\ast$
 and $(\mu\circ\lambda)^*=\lambda^*\circ\mu^*$ 
we deduce  as in (\ref{def-fun}) that the following assignments 
\[
(Q^\point, \F)\mapsto(\A_\F, \aut_{\F}^{Q^\point})\,,
\]
\begin{equation}\label{moraut}
\big( (\mu,\phi):(P^\point,\G)\to (Q^\point,\F)\big)
\;\mapsto 
\big((\mu, \phi)^\ast := 
\mu^\ast\circ \phi^\ast
:\aut_{\F}^{Q^\point}\to \aut_{\G}^{P^\point}
\big)\,,
\end{equation}
define a contravariant functor with values in the category $\grgr$ of group-graphs.
When restricted to generalized curves this functor is denoted by
\begin{equation}\label{functorAut}
\aut:\mbf{Man^\point}\times\fol\to\grgr\,.
\end{equation}
%\medskip

From now on $\F$ will be a {\sc{generalized curve}}.\\

For  any deformation
$\F_{Q^\point}$ of $\F$ over $Q^\point$, 
let us choose   a good trivializing system $(\Psi_{D})_{D\in\Ve_{\A_\F}}$   meaning that the properties (\ref{overQ})-(\ref{collagenoeuddic}) of Theorem~\ref{psiD} are satisfied. The family $(\Phi_{D,\msf e})_{D\in \msf e}$, defined by 

\begin{equation}\label{relco}
\Phi_{D,\msf e}=\Psi_D\circ\Psi_{D'}^{-1}\,,
\quad
\msf e=\langle D,D'\rangle\,,
\end{equation}
is an  element of $Z^1(\A_\F, \aut_{\F}^{Q^\point})$.

\begin{lema}\label{isonatDefH1}
The cohomology class $C(\F_{Q^\point})\in H^1(\A_\F, \aut_{\F}^{Q^\point})$ of the above cocycle $(\Phi_{D,\msf e})_{D\in \msf e}$ does not depend on the choice of a good trivializing system; moreover it only depends on the $\mc C^{\mr{ex}}$-class $[\F_{Q^\point}]\in \mr{Def}_{\F}^{Q^\point}$. 
\end{lema}

\begin{proof}
We check that if $(\Psi_{D})_{D\in \Ve_{\A_\F}}$ and $(\Psi'_{D})_{D\in \Ve_{\A_\F}}$ are two good trivializing systems for $\F_{Q^\point}$, then  the homeomorphisms  $\Psi_D\circ\Psi_D'^{-1}$ belong to $\aut_{\F}^{Q^\point}(D)$ and define a $0$-cocycle whose action on the cocycle $(\Psi_D\circ\Psi_{D'}^{-1})$ gives the cocycle
$(\Psi'_D\circ\Psi{'}_{D'}^{-1})$. Hence $C(\F_{Q^\point})$ is well defined. On the other hand if $\Phi$ is an \Cex-conjugacy between another deformation $\G_{Q^\point}$ of $\F$ over $Q^\point$ and $\F_{Q^\point}$, $\Phi(\G_{Q^\point})=\F_{Q^\point}$, we easily verify that $(\Psi_D\circ\Phi^\sharp)_{D\in \Ve_{\A_\F}}$ is a good trivializing system  for $\G_{Q^\point}$ with the same  associated cocycle.
\end{proof}

\begin{teo}\label{cocycle} Let $Q^\point$ be a germ of manifold and $\F$ a foliation which is a generalized curve.
Then the maps  
\[
C_{\F}^{Q^\point} :\mr{Def}_{\F}^{Q^\point}\to H^1(\A_\F, \aut_{\F}^{Q^\point})^\point\,,
\qquad
[\F_{Q^\point}]\mapsto C(\F_{Q^\point})\,,
\]
are bijective. Moreover they define a natural isomorphism
\[
C: \mr{Def}\iso H^1\circ \mr{Aut}\,,
\]
between the contravariant functor $\mr{Def}:\mbf{Man^\point}\times\fol\to \mbf{Set}^\point$ introduced in Definition~\ref{def411}  and the contravariant functor $(Q^\point,\F) \mapsto H^1(\A_\F, \aut_{\F}^{Q^\point})$ obtained by composing 
 the contravariant functor $\mr{Aut}:\mbf{Man^\point}\times\fol\to\grgr$ with the covariant
cohomological functor 
$H^1:\grgr\to \mbf{Set}^\point$ defined in (\ref{defH1}) (pointed by the class of the identity). 
\end{teo}

\begin{proof} The maps $C_\F^{Q^\point}$ are well defined thanks to Lemma~\ref{isonatDefH1}. We proceed in three steps:\\

\textit{Step 1: functoriality of $C$. } 
We must prove that, given a germ of holomorphic map $\mu:P^\point\to Q^\point$ and an \Cex-conjugacy $\phi : \G\to \F$ between  generalized curves, the following diagram is commutative:
\[\xymatrix{
&& \Def_{\F}^{Q^\point}\ar'[d]^{\mu^*}[dd]\ar[dll]_{C_{\F}^{Q^\point}}\ar[rrr]^{\phi^\ast}&&&\Def_{\G}^{Q^\point}\ar[dd]^{\mu^*}\ar[dll]_{C_{\G}^{Q^\point}}\\ 
H^1({\A_\F},\aut_{\F}^{Q^\point})\ar[rrr]^{H^1(\phi^\ast)\hphantom{AAA}}\ar[dd]^{H^1(\mu^*)}&&&H^1({\A_\G},\aut_{\G}^{Q^\point})\ar[dd]^<(0.2){H^1(\mu^*)}&&\\ 
&&\Def_{\F}^{P^\point}\ar'[r]^{\phi^\ast}[rrr]\ar[dll]_{C_{\F}^{P^\point}}&&&\Def_{\G}^{P^\point}\ar[dll]_{C_{\G}^{P^\point}}\\
 H^1({\A_\F},\aut_{\F}^{P^\point})\ar[rrr]^{H^1(\phi^\ast)}&&&H^1({\A_\G},\aut_{\G}^{P^\point})&&}
\]
Let us check first the commutativity of the lateral faces:
If $(\Psi_D)_{D\in \Ve_{\A_\F}}$ is a good trivializing system for $\F_{Q^\point}$   then $(\mu^*\Psi_D)_{D\in\Ve_{\A_\F}}$ is also a good trivializing system for $\mu^\ast\F_{Q^\point}$. Consequently we have:
\[C_{\F}^{P^\point}([\mu^*\F_{Q^\point}])=[\mu^*\Psi_D\circ\mu^*\Psi_{D'}^{-1}]=H^1(\mu^*)([\Psi_D\circ\Psi_{D'}^{-1}])=H^1(\mu^*)\circ C_{\F}^{Q^\point}([\F_{Q^\point}]).\]
To check the commutativity of the top face, we notice that by definition $\mf c:=H^1(\phi^\ast)\circ C^{Q^\point}_{\F}([\F_{Q^\point}])$ is the cohomology class in $H^1({\A_\G},\aut_{\G}^{Q^\point})$ of the cocycle $(\phi^{\sharp\,-1}_{Q^\point}\circ\Psi_D\circ\Psi_{D'}^{-1}\circ\phi^{\sharp}_{Q^\point})$. It coincides with the cocycle~(\ref{cocyclerecolmnt})
used in the proof of Theorem \ref{famconsigma} to construct the deformation $\G_{Q^\point}\in\phi^\ast([\F_{Q^\point}])$. Therefore $\mf c=C^{Q^\point}_{\G}([\G_{Q^\point}])$. The same arguments give the commutativity of the lower face. That of the back and front faces of the cube results from the relations (\ref{comsigmalambda}) and (\ref{sq2}) respectively. 
\\

\textit{Step 2: injectivity of $C_{\F}^{Q^\point}$. }
Let  $(\Psi_{D})_{D\in \Ve_{\A_\F}}$ resp. $(\Psi'_{D})_{D\in \Ve_{\A_\F}}$ be  good trivializing systems for two equisingular deformations  $\F_{Q^\point}$, resp. $\F'_{Q^\point}$, inducing  the same cohomology class in $H^1({\A_\F},\aut_{\F}^{Q^\point})$. There exist   $\Phi_D\in \aut_{\F}^{Q^\point}(D)$, $D\in \Ve_{\A_\F}$,
such that the following relation: 
\[\Phi_D\circ\Psi_D\circ\Psi_{D'}^{-1}\circ\Phi_{D'}^{-1}=\Psi{'}_D\circ\Psi{'}_{D'}^{-1}
\]
is satisfied for any pair $(D, D')$ of irreducible components of $\mc E_\F$ such that $\{s_{DD'}\}=D\cap D'$
is neither a nodal singularity or a regular point of $\F^\sharp$. This relation also means that the homeomorphisms $K_D:=\Psi{'}_D^{-1}\circ \Phi_D\circ\Psi_D$ defined on neighborhoods of $D\times\{u_0\}$ coincide on neighborhoods of $(s_{DD'},u_0)$ and induce a \Cex-conjugacy between $\F_{Q^\point}$ and $\F'_{Q^\point}$.\\

\textit{Step 3: surjectivity of $C_{\F}^{Q^\point}$.} Given a cocycle $(\Phi_{D,\ge})\in Z^1(\A_\F,\mr{Aut}^{Q^\point}_{\F})$, the construction of an equisingular deformation $\F_{Q^\point}$ equipped with a good trivializing system satisfying (\ref{relco}), may be done by a gluing process  as in the proof of Theorem~\ref{famconsigma}.
\end{proof}

\subsection{Sheaf of transversal symmetries}
 Let us fix again a foliation $\F$ and a germ of manifold $Q^\point=(Q,u_0)$. For an open set $U\subset \mc E_\F$ we will say that an automorphism $\Phi\in \underline{\aut}^{Q^\point}_{\F}(U) $ \emph{fixes the leaves}, if it leaves invariant the codimension one foliation $\F^\sharp\times Q$. 
We denote by $\underline{\fix}^{Q^\point}_{\F}\subset\underline{\aut}^{Q^\point}_{\F}$,  the subsheaf of normal subgroups consisting of  these automorphisms.
We will describe in an explicit way
the quotient sheaf
\begin{equation*}
\underline{\sym}^{Q^\point}_{\F}=\underline{\aut}^{Q^\point}_{\F}/\underline{\fix}^{Q^\point}_{\F}\,.
\end{equation*}
To do that let us consider the normal subgroup \begin{equation*}
\DiffQz= \{\phi\in\DiffQ\;|\;\phi(z,u_0)\equiv (z,u_0)\},
\end{equation*}
of the group $\DiffQ$ defined in (\ref{diffQ}), 
and for any subgroup 
\[ G\subset\DiffQ
 \]
 let us  adopt the following notations:
%\color{blue}
\begin{itemize}
\item $C_{Q^\point}(G)$ is the \emph{centralizer of $G$}, i.e. the subgroup of  $\DiffQ$ whose elements commute with any element of $G$;
\item
$C_{Q^\point}^0(G)=C_{Q^\point}(G)\cap \DiffQzero$;
\item in  the 
monogenous case  $G=\langle h\rangle$, we  write $C_{Q^\point}(h)$  and  $C_{Q^\point}^0(h)$  instead of  $C_{Q^\point}(\langle h\rangle)$  and $C_{Q^\point}^0(\langle h\rangle)$.
\end{itemize}
Now let us fix an invariant component $D$ of $\mc E_\F$.
For $m\in D\setminus \mr{Sing}(\F^\sharp)$, 
let us choose a germ of holomorphic submersion 
\[g: (M_\F, m)\longrightarrow (\C,0)
\]
constant on the leaves of $\F^\sharp$. 
Any $\phi\in \underline{\aut}^{Q^\point}_{\F}(m)$
factorizes  through $g\times\mr{id}_Q$, defining an element 
$g_{\ast}(\phi)\in\DiffQzero$
such that
\[g_{\ast}(\phi)\circ (g\times\mr{id}_Q)=(g\times\mr{id}_Q)\circ \phi\,.\]
 The holomorphy of $g_{\ast}(\phi)$ results from  the fact that   $\phi$ is transversely holomorphic by definition.
Clearly 
\begin{equation}\label{g*}
g_{\ast}:\un{\mr{Aut}}_{\F}^{Q^\point}(m)\to\DiffQzero
\end{equation}
 is a surjective group morphism. 
\begin{lema}\label{remfactgermes}
The following sequence
\begin{equation}\label{factgermes}
1\to
 \underline{\fix}^{Q^\point}_{\F}(m)
\to
 \underline{\aut}^{Q^\point}_{\F}(m)
\stackrel{g_{\ast}}{\to} \DiffQzero\to 1
\end{equation}
is exact.
\end{lema}
\begin{proof}
For the exactness at the central term, 
let us first notice that the germ at $(m,u_0)$ of an element
$\phi\in\un{\aut}^{Q^\point}_{\F}(m)$ preserves the codimension one foliation $\F^\sharp\times Q$  if and only if there is a factorization $g_\ast(\phi)^\flat$:
\[
\xymatrix{
(M_\F\times Q,(m,u_0))\ar[r]^{\hphantom{aaa}g\times\mr{id}_Q}\ar[d]_{\phi} \;& (\C\times Q,(0,u_0))\ar[d]^{g_\ast(\phi)}\ar[r]^{\phantom{aaaa}\mr{pr}_\C} & (\C,0)\ar@{.>}[d]^{g_\ast(\phi)^\flat}\\
(M_\F\times Q,(m,u_0))\ar[r]^{\hphantom{aaa}g\times\mr{id}_Q} \;& (\C\times Q,(0,u_0))\ar[r]^{\phantom{aaaa}\mr{pr}_\C} &(\C,0)
}
\]
where $\mr{pr}_\C(z,u)=z$. Since $g_\ast(\phi)(p,u)=(\wt{\phi}(p,u),u)$, $g_\ast(\phi)^{\flat}$ exists if and only if $\wt{\phi}(p,u)$ does not depend on~$u$.
But $\wt\phi(z,u_0)=z$, therefore $g_\ast(\phi)^{\flat}$ exists if and only if $g_\ast(\phi)=\mr{id}_{\C\times Q}$.
\end{proof}

 \begin{lema} If $U\subset D$ is open\footnote{$U$ may not be open in $\mc E_\F$.} and connected and $p\in U$,  then we have the exact sequence:
 \begin{equation}\label{FAS}
 1\to\un{\mr{Fix}}_{\F}^{Q^\point}(U)\to\un{\mr{Aut}}_{\F}^{Q^\point}(U)\to\un{\mr{Sym}}_{\F}^{Q^\point}(p)\,.
 \end{equation}
\end{lema}
\begin{proof}
The statement is trivial 
if $U=W\cap D$ and $W\subset M_\F$ is an open subset  trivializing the foliation $\F^\sharp$.  If $U\cap\mr{Sing}(\F^\sharp)=\emptyset$ we cover $U$ by open subsets in $M_\F$ trivializing $\F^\sharp$  and we conclude by  connectedness of $U$. For the last case $p\in\mr{Sing}(\F^\sharp)$ we take a  point $q\in U\setminus \mr{Sing}(\F^\sharp)$ close to $p$ and we note that if the germ of an element $\phi\in\un{\mr{Aut}}_{\F}^{Q^\point}(U)$ at $p$ is in $\un{\mr{Fix}}_{\F}^{Q^\point}(p)$ then the germ of $\phi$ at $q$ also belongs to $\un{\mr{Fix}}_{\F}^{Q^\point}(q)$. 
By applying the exactness of sequence (\ref{FAS}) substituting $U$ and $p$ by $U\setminus\mr{Sing}(\F^\sharp)$ and $q$ respectively, we deduce that $\phi\in\un{\mr{Fix}}_{\F}^{Q^\point}(U\setminus\mr{Sing}(\F^\sharp))$. It remains to see that the germ of $\phi$ at $p'\in U\cap\mr{Sing}(\F^\sharp)$ belongs to $\un{\mr{Fix}}_{\F}^{Q^\point}(p')$.
For this we use the holomorphy of $\phi$ at $p'$ and the following characterization: $\phi\in \underline{\fix}^{Q^\point}_{\F}(p')\Leftrightarrow(\phi^*\omega)\wedge\omega\equiv 0$, where $\omega$ is the germ at $p'$ of  a holomorphic 1-differential form defining the codimension one foliation $\F^\sharp\times Q$.
\end{proof}  
Let us fix an invariant component $D$  of $\mc E_\F$ and let us  denote by  $i_D:D\hookrightarrow \mc E_\F$ the inclusion map. Let us also fix a transverse fibration $\rho:(M_\F\times Q^\point,D)\to D$ 
satisfying properties (\ref{resrho})-(\ref{invfibers})  described in the step 1 of the proof of Theorem~\ref{psiD} and let us consider the subsheaf over $D$
\[\un{\mr{Aut}}_{\F,\rho}^{Q^\point}\subset i_D^{-1}\un{\mr{Aut}}_{\F}^{Q^\point}\]  
of automorphisms preserving the fibration $\rho$.

 \begin{lema}\label{aut-rho}
 If $\F$ is a generalized curve, for any connected open set $U$ of $D$ and any point $m\in U \setminus\mr{Sing}(\F^\sharp)$,
 the following assertions hold:
\begin{enumerate}[(i)]
\item\label{loccstsaf} The sheaf  $\un{\mr{Aut}}_{\F,\rho}^{Q^\point}$ is locally constant over $D\setminus\mr{Sing}(\F^\sharp)$ and the morphism   $g_{*}$ defined in (\ref{g*}) induces an isomorphism  
\[\un{\mr{Aut}}_{\F,\rho}^{Q^\point}(m)\simeq \DiffQzero\,;\]
\item\label{gstariso}  The restriction map $\un{\mr{Aut}}_{\F,\rho}^{Q^\point}(U)\to\un{\mr{Aut}}_{\F,\rho}^{Q^\point}(U\setminus\mr{Sing}(\F^\sharp))$ is an isomorphism and  $g_{*}$ induces an isomorphism $\un{\mr{Aut}}_{\F,\rho}^{Q^\point}(U)\simeq C^0_{Q^\point}(H_U)$, where $H_U$ is  the holonomy group 
\[H_U:= \mc H_D^{\un\F_{Q^\point}^{\mr{ct}\,\sharp}}( \pi_1(U\setminus \mr{Sing}(\F^\sharp), m ))\subset\DiffQzero\,;
\]
\item\label{respamaut} For any $p\in U$,   the natural map $\un{\mr{Aut}}_{\F,\rho}^{Q^\point}(U)\to\un{\mr{Aut}}_{\F,\rho}^{Q^\point}(p)$ is injective.
\end{enumerate}
\end{lema}
\begin{proof}
Assertion (\ref{loccstsaf}) follows from the fact that the restriction of $g\times \mr{id}_Q$ to each fiber of $\rho$ is a local diffeomorphism onto $(\C\times Q,(0,u_0))$.
Assertion (\ref{gstariso}) is a consequence of Mattei-Moussu's Theorem as in step 2 of the proof of Theorem~\ref{psiD}. 
To prove assertion (\ref{respamaut}), let us assume  that the germ of $\phi\in\un{\mr{Aut}}_{\F,\rho}^{Q^\point}(U)$ at $p$ is the identity. If $p\notin\mr{Sing}(\F^\sharp)$ then $\phi_{|U\setminus\mr{Sing}(\F^\sharp)}=\mr{id}$ by assertion~(\ref{loccstsaf})  and $\phi=\mr{id}$ using also assertion~(\ref{gstariso}). If $p\in\mr{Sing}(\F^\sharp)$ then there is $q\notin\mr{Sing}(\F^\sharp)$ close to $p$ such that the germ of $\phi$ at $q$ is the identity; we apply the previous case and we conclude by the holomorphy of the germ of $\phi$ at $p$. 
\end{proof}
 \begin{prop}\label{prop59} If $\F$ is a generalized curve, then
the composition of the group sheaves morphisms 
\begin{equation}\label{comp}
\un{\mr{Aut}}_{\F,\rho}^{Q^\point}\hookrightarrow i_D^{-1}\un{\mr{Aut}}_{\F}^{Q^\point}\to i_D^{-1}\un\sym_{\F}^{Q^\point}
\end{equation}
is an isomorphism.
\end{prop}
\begin{proof}
We have to see that $\un{\mr{Aut}}_{\F,\rho}^{Q^\point}(p)\to\un\sym_{\F}^{Q^\point}(p)$ is an isomorphism for each $p\in D$. 
The case  $p\in D\setminus\mr{Sing}(\F^\sharp)$ follows from assertion (i) in Lemma~\ref{aut-rho} and the exact sequence (\ref{factgermes}) in Lemma~\ref{remfactgermes}. Next, we fix $p\in\mr{Sing}(\F^\sharp)$ and we take $[\phi_p]\in\un\sym_{\F}^{Q^\point}(p)$. There is a neighborhood $U$ of $p$ in $D$ and $\phi_U\in\un{\mr{Aut}}_{\F}^{Q^\point}(U)$ such that $[\phi_U]\mapsto[\phi_p]$. In the commutative diagram below
\[\xymatrix{ \tilde\phi_p\in\un{\mr{Aut}}_{\F,\rho}^{Q^\point}(p)\ar[r]& \un\sym_{\F}^{Q^\point}(p)\ni[\phi_p] \\ 
\tilde\phi_U\in \un{\mr{Aut}}_{\F,\rho}^{Q^\point}(U)\ar[r]\ar[u]\ar[d]_{b}^{\rotatebox{90}{\ensuremath{\sim}}}& 
\frac{\un{\mr{Aut}}_{\F}^{Q^\point}(U)}{\un{\mr{Fix}}_{\F}^{Q^\point}(U)}\ni[\phi_U]\ar[u]\ar@{^{(}->}[d]^{c} \\ 
 \un{\mr{Aut}}_{\F,\rho}^{Q^\point}(U\setminus\mr{Sing}(\F^\sharp))\ar[r]^{\sim}_{a}& \un\sym_{\F}^{Q^\point}(U\setminus\mr{Sing}(\F^\sharp))
}\]
the arrow $a$ is an isomorphism by the regular case already considered and the arrow $b$ is also an isomorphism by assertion (ii) of Lemma~\ref{aut-rho}.
Hence there is $\tilde\phi_U\in\un{\mr{Aut}}_{\F,\rho}^{Q^\point}(U)$ such that $[\tilde\phi_U]$ and $[\phi_U]$ are sent to the same element in $\un{\sym}_{\F}^{Q^\point}(U\setminus\mr{Sing}(\F^\sharp))$. Using the exact sequence (\ref{FAS}) we deduce that the arrow $c$ is injective and consequently $\tilde\phi_U$ is sent to $[\phi_U]$. By the commutativity of the top square the germ $\tilde{\phi}_p$ of $\tilde\phi_U$ at $p$ projects onto $[\phi_p]$. This shows that the composition~ (\ref{comp}) is surjective at $p$.
The injectivity of the composition~(\ref{comp}) at $p$ follows, as in the proof of assertion~(iii) in Lemma~\ref{aut-rho}, using the holomorphy of $\tilde\phi_U$ and the  injectivity at the regular points, which has already been shown.
\end{proof}

\begin{cor}\label{cor510}
If $\F$ is a generalized curve,  for any connected open set $U$ of $D$, the following assertions hold:
\begin{enumerate}[(i)]
\item\label{symloccst} The sheaf $\un\sym_{\F}^{Q^\point}$ is locally constant on $D\setminus\mr{Sing}(\F^\sharp)$;
\item\label{isosymCent} The morphism $g_{*}$ induces an isomorphism $\un\sym_{\F}^{Q^\point}(U)\simeq C^0_{Q^\point}(H_U)$;
\item\label{FixAutSym} We have the exact sequence:
\[1\to \un{\mr{Fix}}_{\F}^{Q^\point}(U)\to \un{\mr{Aut}}_{\F}^{Q^\point}(U)\to \un{\sym}_{\F}^{Q^\point}(U)\to 1\,.\]
\end{enumerate}
\end{cor}
\begin{proof} Assertions (\ref{symloccst}) and (\ref{isosymCent}) are obvious from  the isomorphism $\un{\mr{Aut}}_{\F,\rho}^{Q^\point}\simeq \un\sym_{\F}^{Q^\point}$ and assertions (\ref{loccstsaf}) and (\ref{gstariso}) in Lemma~\ref{aut-rho}.
To check the exactness of the sequence in assertion (\ref{FixAutSym}) it only remains to show the surjectivity of $\un{\mr{Aut}}_{\F}^{Q^\point}(U)\to \un{\sym}_{\F}^{Q^\point}(U)$. This is so because the composition
\[\un{\mr{Aut}}_{\F,\rho}^{Q^\point}(U)\hookrightarrow \un{\mr{Aut}}_{\F}^{Q^\point}(U)\to \un\sym_{\F}^{Q^\point}(U)\]
is an isomorphism thanks to Proposition~\ref{prop59}.
\end{proof}

\subsection{Group-graph of transversal symmetries} 
 Let us again fix a  foliation $\F$ that is a generalized curve.
We consider the normal subgroup-graph
 $\fix^{Q^\point}_{\F}\subset\aut^{Q^\point}_{\F}$  defined by 
 \[\fix^{Q^\point}_{\F}(\star)=\aut^{Q^\point}_{\F}(\star)\cap\un{\fix}^{Q^\point}_{\F}(\star)\,,\
 \quad
 \star\in \Ve_{\A_\F}\cup\Ed_{\A_\F}\,,
\]
where $\un\fix_{\F}^{Q^\point}(\msf e)$ denotes $\un\fix_{\F}^{Q^\point}(D\cap D')$ if $\msf e=\langle D,D'\rangle\in\Ed_{\A_\F}$.
\begin{defin}
The \emph{group-graph of transversal symmetries} is the 
 quotient group-graph $\sym_{\F}^{Q^\point}$  defined by the group-graph exact sequence
\begin{equation}\label{sym}
1\to\fix_{\F}^{Q^\point}\to\aut_{\F}^{Q^\point}\stackrel{\pi_{\F}^{Q^\point}}{\longrightarrow} \mr{Sym}_{\F}^{Q^\point}={\aut}^{Q^\point}_{\F}/{\fix}^{Q^\point}_{\F}\to 1\,.
\end{equation} 
\end{defin}
\noindent For each invariant component $D\in\Ve_{\A_\F}$, using
 the exact sequence (\ref{FixAutSym}) in Corollary~\ref{cor510} with $U=D$,
we have a natural\footnote{
If $A\to A'$ is a morphism of sheaves of groups over $X$ sending a normal subgroup $F$ into $F'$  then for any open subset $U\subset X$ the following diagram is commutative:
\[\xymatrix{A(U)/F(U)\ar[r]\ar[d]&(A/F)(U)\ar[d]\\ A'(U)/F'(U)\ar[r]&(A'/F')(U)}\]} isomorphism:
\begin{equation}\label{symD}
\sym_{\F}^{Q^\point}(D)=\un\aut_{\F}^{Q^\point}(D)/\un\fix_{\F}^{Q^\point}(D)\iso\un\sym_{\F}^{Q^\point}(D)\,,
\end{equation}
when $\F$ is a generalized curve.

We check that if $(\mu,\phi):(P^\point,\G)\to(Q^\point,\F)$  
is a morphism in the category
$\mbf{Man^\point}\times\fol$,
then the morphism $(\mu,\phi)^*$
defined in (\ref{moraut}) sends the group-graph ${\fix}^{Q^\point}_{\F}$ into ${\fix}^{P^\point}_{\G}$ and it factorizes (see Remark~\ref{factgrgr})  as a  morphism of group-graphs over the graph morphism $\A_\phi:\A_\G\to\A_\F$ defined in (\ref{Aphi}), that we also denote by 
\[(\mu,\phi)^* :  {\mr{Sym}}_{\F}^{Q^\point} \to {\mr{Sym}}_{\G}^{P^\point}.\] 
 This allows to define  a contravariant functor from $\mbf{Man^\point}\times\fol$ to $\grgr$
\begin{equation*}\label{functSym}
\mr{Sym}:
(Q^\point,\F)\mapsto (\A_\F,\mr{Sym}^{Q^\point}_{\F})\,,
\quad
(\mu,\phi)\mapsto 
(\mu,\phi)^*\,.
\end{equation*}
The collection $\{\pi_{\F}^{Q^\point}\}$
 of quotient maps (\ref{sym}) defines a natural transformation
\begin{equation*}\label{natTrAutSym}
\aut\to\mr{Sym}\,.
\end{equation*}
By applying the functor $H^1:\grgr\to \mbf{Set}^\point$ to the morphisms $\pi_{\F}^{Q^\point}$
we obtain 
maps
\begin{equation}\label{auttosym}
H^1(\A_\F, {\aut}^{Q^\point}_{\F})^\point \to H^1(\A_\F,{\mr{Sym}}^{Q^\point}_{\F})^\point
\end{equation}
defining a natural transformation 
$H^1\circ \mr{Aut}\to H^1\circ \mr{Sym}\,.$
It follows immediately from Lemma~\ref{extension} below and Proposition~\ref{quoti} applied to the exact sequence (\ref{sym}) that:
\begin{prop}\label{isoH1AutSym}
For any germ of manifold $Q^\point$ and any generalized curve $\F$, the map (\ref{auttosym}) 
is bijective and consequently the natural transformation 
\[H^1\circ \mr{Aut}\to H^1\circ \mr{Sym}\]
is an isomorphism of contravariant functors from  $\mbf{Man^\point}\times\fol$ to $\mbf{Set}^\point$.
\end{prop}

\begin{lema}\label{extension}
Assume that $\F$ is a generalized curve. For any edge $\msf e=\langle D,D'\rangle$ of $\A_\F$ with $D$ invariant, the restriction map 
${\fix}^{Q^\point}_{\F}(D)\to{\fix}^{Q^\point}_{\F}({\msf e})$ is surjective. 
\end{lema}
\begin{proof}
  At the point $\{s\}=D\cap D'$ we take local coordinates $(x,y):(M_\F,s)\to(\C^2,0)$  
such that the foliation $\F^\sharp$ is defined by a vector field
$x\partial_x+yB(x,y)\partial_y$ with $B(0,0)\neq 0$. Let us consider $\Phi\in{\fix}^{Q^\point}_{\F}({\msf e})=\un\fix^{Q^\point}_{\F}(s)$. 

Let $u=(u_1,\ldots,u_q)$ be a centered   coordinate system on $Q^\point$. In the chart $\chi=(x,y,u)$ the foliation $\F_{Q^\point}^{\mr{ct}\,\sharp}$ is given by the vector field $Z=x\partial_x+yB(x,y)\partial_y$ and the foliation $\F^\sharp\times Q$ is defined by  $yB(x,y)dx-xdy=0$. Let us denote by $\varphi=\chi\circ \Phi\circ \chi^{-1}$ the expression of $\Phi$ in this chart. Since $\varphi(x,y,0)=(x,y,0)$ and the points $(x,y,u)$ and $\varphi(x,y,u)$ belong to the same leaf $L_{x,y,u}$ of $\F^\sharp\times Q$ the function $\tau(x,y,t)=\int_{(x,y,t)}^{\varphi(x,y,t)}\frac{dx}{x}\big|_{L_{x,y,u}}=\int_{(x,y,t)}^{\varphi(x,y,t)}\frac{dy}{yB(x,y)}\big|_{L_{x,y,u}}$ is well defined and holomorphic in an open neighborhood $\Omega$ of $C=\{(x,y,u)\,:\, \varepsilon\le|x|\le 2\varepsilon,\ |y|\le\varepsilon,\ |u|\le\delta\}$ for $0<\delta\ll \varepsilon$ small enough, moreover $\tau(x,y,0)=0$. By definition, the flow $\Phi^Z_t$ of $Z$ satisfies $\Phi^Z_{\tau(p)}(p)=\varphi(p)$ for $p\in C$. Let $\alpha:\C\to\R$ be a $\mathcal C^\infty$ function with compact support on $x(\Omega)$, that is equal to $1$ in a neighborhood of $\{\varepsilon\le |x|\le2\varepsilon\}$. The map $p\mapsto \xi(p):=\Phi^Z_{\alpha(x(p))\tau(p)}(p)$ is a $\mathcal C^\infty$ diffeomorphism, because its restriction to $u=0$ is the identity and moreover it is a local diffeomorphism as it can be easily checked by computing its Jacobian matrix. Clearly the map $\phi=\xi^{-1}\circ\varphi$ coincides with $\varphi$ on a neighborhood of $s$, it preserves the codimension $1$ foliation $\F^\sharp\times P$ and $\phi(x,y,u)=(x,y,u)$ for $\varepsilon\le |x|\le2\varepsilon$. Thus, $\Phi$ extends to a neighborhood of $D$ as the identity and defines an element of ${\fix}^{Q^\point}_{\F}(D)$.
\end{proof}

Now we will give an explicit expression of the group-graph $\sym_{\F}^{Q^\point}$ which will depend on the choice of 
the following additional data:

\begin{defin}\label{geometric-system}
A \emph{geometric system} for an invariant component $D$ of $\mc E_\F$ consists in:
\begin{itemize}
\item a point 
$o_D\in D\setminus(\mr{Sing}(\mc E_\F)\cup\mr{Sing}(\F^\sharp))$ and a germ of holomorphic submersion $g:(M_\F,  o_D)\to(\C,0)$
which is constant along the leaves of $\F^\sharp$;
\item a collection $\{U_p\}_{p\in\mr{Sing}(\F^\sharp)\cap D}$ of  connected and simply connected open subsets of $D$ such that $U_p\cap\mr{Sing}(\F^\sharp)=\{p\}$ and $ o_D\in\bigcap_{p\in\mr{Sing}(\F^\sharp)\cap D}U_p$. 
\end{itemize}
For $\ge=\langle D,D'\rangle$ with $D\cap D'=\{p\}$
we denote by 
\begin{equation}\label{hHD}
h_{D,\ge}\in H_D\subset\DiffQ
\end{equation}
the holonomy of $\F_{Q^\point}^{\mathrm{ct}\,\sharp}$ along of a path in $U_p\setminus\{p\}$ of index $1$ with respect to~$p$, which belongs to the holonomy group $H_D$ image of the morphism $\mc H_D^{\F_{Q^\point}^{\mr{ct}\,\sharp}}$ in~(\ref{holonomiefam}).
\end{defin}

\begin{prop}\label{G} Assume that $\F$ is a generalized curve.
If $D\in\Ve_{\A_\F}$ and $\ge=\langle D,D'\rangle\in\Ed_{\A_\F}$, after choosing a geometric system for $D$, the morphism (\ref{g*}) with $m= o_D$ induces isomorphisms
\[G_{D,\ge}:\sym_{\F}^{Q^\point}(\ge)\iso C^0_{Q^\point}(h_{D,\ge})\quad\text{and}\quad G_{D}:\sym_{\F}^{Q^\point}(D)\iso C^0_{Q^\point}(H_D)\,.\]
Under these isomorphisms the restriction map $\sym_{\F}^{Q^\point}(D)\to\sym_{\F}^{Q^\point}(\ge)$ is just the inclusion $C^0_{Q^\point}(H_D)\hookrightarrow C^0_{Q^\point}(h_{D,\ge})$.
\end{prop}
\begin{proof}
We have:
$\sym_{\F}^{Q^\point}(\ge)=\un\sym_{\F}^{Q^\point}(\ge)\simeq\un\sym_{\F}^{Q^\point}(U_p)$, thanks to assertion (\ref{symloccst}) of Corollary~\ref{cor510}, where $\{p\}=D\cap D'$. By assertion (\ref{isosymCent}) in 
Corollary~\ref{cor510} with $U=U_p$, $g_*$ induces an isomorphism
$\un\sym_{\F}^{Q^\point}(U_p)\simeq C^0_{Q^\point}(h_{D,\ge})$. The second isomorphism follows immediately from (\ref{symD}) and assertion (\ref{isosymCent}) of Corollary~\ref{cor510} with $U=D$.
\end{proof}
\bigskip

\section{Finite type foliations and infinitesimal transversal symmetries}\label{sectFinTypeInfTrSym}

\subsection{Finite type foliations}\label{secft} Given a foliation $\F$ which is a generalized curve,  
we will say that a vertex $D$, resp. an edge $\langle D,D'\rangle$, belonging to a cut-component $\A_\F^\alpha$, $\alpha\in \mc A$, of  $\A_\F$ (see Remark \ref{suppAut}) is \emph{red for} $\F$ if, using the notations in (\ref{hHD}) with $Q=\{u_0\}$,
the holonomy group $H_D$ of $\F^\sharp$ is not finite, resp. the holonomy diffeomorphism $h_{D,\ge}$ (or equivalently $h_{D',\ge}$) is not periodic.
Classically a vertex $D$, resp. an edge $\langle D,D'\rangle$, is red if every holomorphic first integral of $\F^\sharp$ defined in a neighborhood of $D$, resp. $D\cap D'$, is constant.\\

Notice that the \emph{red part} $\msf{R}^\alpha_{\F}$ of $\A^\alpha_{\F}$ is a sub-graph. When it is connected and non-empty,   we consider  the partial order relation $\prec_{\AR^\alpha_{\F}}$ on $\Ve_{\A_{\F}^\alpha}$ defined in Subsection \ref{Subsecfunctpruning}. When
$\AR_{\F}^\alpha=\emptyset$ we will consider the partial order relation $\prec_{\{v\}}$  on $\Ve_{\A_{\F}^\alpha}$ defined by the subgraph $\{v\}$ reduced to some single vertex $v$.

\begin{defin}\label{def-type-fini} We say that $\F$ is \emph{of finite type} if for each $\alpha\in \mc A$ one of the following conditions holds:
\begin{enumerate}[(i)]
\item $\AR^\alpha_{\F}\neq\emptyset $ is connected and for any edge $\msf e=\langle D,D'\rangle\in(\Ed_{\A_{\F}^\alpha}\setminus \Ed_{\AR_{\F}^\alpha})$ with $D' \prec_{\AR^\alpha_{\F}} D$, the holonomy group $H_D$ is generated by the holonomy map $h_{D,\msf e}\,$;
\item $\AR^\alpha_{\F}=\emptyset$ and $\A_{\F}^\alpha$ contains a vertex $v$ such that we have: $H_D=\langle h_{D,\msf e}\rangle$ for any edge $\msf e=\langle D,D'\rangle\in\Ed_{\A_{\F}^\alpha}$ with 
$D'\prec_{\{v\}} D$. 
\end{enumerate}
We will denote by $\mbf{Fol}_\mbf{ft}\subset\fol$ the full subcategory of finite type foliations.
\end{defin}

When $\F$ is of finite type,  for every germ of manifold~$Q^\point$ the subgraph $\AR^\alpha_{\F}$ is  $\sym_{\F}^{Q^\point}$-repulsive  in $\A^\alpha_{\F}$ in the meaning of Section \ref{Subsecfunctpruning}. Indeed for $D\in \Ve_{\A^\alpha_{\F}}$ and $\msf e=\langle D,D'\rangle\in \Ed_{\A_{\F}^\alpha}$,  thanks to Proposition~\ref{G}, we have
 isomorphisms ${\mr{Sym}}^{Q^\point}_{\F}(D)\simeq C^0_{Q^\point}(H_D)$ and ${\mr{Sym}}^{Q^\point}_{\F}(\msf e)\simeq C^0_{Q^\point}(h_{D,\msf e})$. 
As we will see later
the cohomology of $\sym_{\F}^{Q^\point}$  is given by its restriction to  the subgraph 
\[\AR_{\F}:=\bigsqcup_{\alpha\in\mc A}\AR_\F^\alpha\subset\A_\F.\]

\begin{defin}\label{redsym}
We call \emph{restricted group-graph of transversal symmetries} the  group-graph ${\mr{RSym}}^{Q^\point}_{\F}=r_\F^*\sym_\F^{Q^\point}$ 
over $\AR_\F$ defined as the pull-back by the inclusion $r_\F:\AR_\F\hookrightarrow\A_\F$:
\[
{\mr{RSym}}^{Q^\point}_{\F}(\star)={\mr{Sym}}^{Q^\point}_{\F}(\star)\,,\qquad \star\in \Ve_{\AR_{\F}}\cup\Ed_{\AR_{\F}}\,,
\]
\end{defin}

Notice that for  any morphism  $\phi:\G\to\F$ in the category $\fol$, the graph isomorphism $\A_\phi:\A_\G\to\A_\F$ restricts to a graph isomorphism $\AR_\phi:\AR_{\G}\iso \AR_{\F}$.

If $\mu:P^\point \to Q^\point$ is a morphism in $\mbf{Man}^\point$, we consider the left diagram of group-graphs morphisms over the right diagram of graph morphisms: 
\[\xymatrix{\mr{Sym}_\F^{Q^\point}\ar[r]^{(\mu,\phi)^*}\ar[d]_{\imath_{r_\F}}\ar[rd]^F&\mr{Sym}_\G^{P^\point}\ar[d]^{\imath_{r_\G}}\\ \mr{RSym}_\F^{Q^\point}\ar@{-->}[r]^{\bar F}&\mr{RSym}_\G^{P^\point}}\qquad\begin{array}{c}{}\\[10mm] \text{over}\end{array}\qquad\xymatrix{\A_\F^{\vphantom{Q^\point}}&\ar[l]_{\A_\phi}\A_\G^{\vphantom{Q^\point}}\\ \AR_\F^{\vphantom{Q^\point}}\ar@{^{(}->}[u]^{r_\F}&\ar@{..>}[l]_{\AR_\phi}\ar@{^{(}->}[u]_{r_\G}\AR_\G^{\vphantom{Q^\point}}\ar[lu]_f}\]
where $\iota_{r_\F}$ and $\iota_{r_\G}$ denote the canonical morphisms, see Definition \ref{pbgrgr}.
Since $\A_\phi(\AR_\G)\subset \AR_\F$, the morphism $F=\imath_{r_\G}\circ(\mu,\phi)^*$ over $f=\A_\phi\circ r_\G$ factorizes 
through $\imath_{r_\F}$, according to Remark~\ref{factorization}, and defines a morphism of group-graphs $\bar F:\mr{RSym}_\F^{Q^\point}\to\mr{RSym}_\G^{Q^\point}$ over $\AR_\phi$. 
By abuse of notation we will denote $\bar F$ as $(\mu,\phi)^*$.
This allows  to consider the contravariant functor
\[\mr{RSym}: \mbf{Man^\point}\times\fol \to \grgr,\quad (Q^\point, \F)\mapsto (\AR_\F,\mr{RSym}^{Q^\point}_{\F}),\quad(\mu,\phi)\mapsto(\mu,\phi)^*\,.\]
The collection of canonical morphisms $\imath_{r_\F}:\sym_\F^{Q^\point}\to\mr{RSym}_\F^{Q^\point}$ of group-graphs over the graph morphisms $r_\F:\AR_\F\hookrightarrow \A_\F$ defines
a natural transformation
\begin{equation*}\label{NatTrRsym}
R : \mr{Sym}\to \mr{RSym}
\end{equation*}
between contravariant functors from 
$\mbf{Man^\point}\times\fol$ to $\grgr$. 
It induces a natural transformation 
\begin{equation}\label{H1NatTrRsym}
\mc R:=H^1(R) : H^1\circ \mr{Sym}\to H^1\circ \mr{RSym}
\end{equation}
between contravariant functors
from $\mbf{Man^\point}\times\fol$ to $\mbf{Set}^\point$.
By applying (\ref{h1decSoDir}) and Theorem~\ref{pruning}
 to each subtree $\AR_{\F}^\alpha\subset\A_{\F}^\alpha$, $\alpha\in\mc A$,
we directly obtain:

\begin{teo}\label{aut-Raut/fix} For any germ of manifold $Q^\point
$ and any  finite type foliation which is a generalized curve, the map
\[\mc R_{\F}^{Q^\point} : 
H^1(\A_\F,{\mr{Sym}}^{Q^\point}_{\F})^\point\iso H^1(\AR_\F,{\mr{RSym}}^{Q^\point}_{\F})^\point
\]
is bijective and  the natural transformation  $\mc R$ considered in (\ref{H1NatTrRsym})   is an isomorphism of  contravariant functors when restricted to the subcategory $\mbf{Man^\point}\times\mbf{Fol}_\mbf{ft}$.
\end{teo}

We will see in the next section that 
the group-graph ${\mr{RSym}}^{Q^\point}_{\F}$  is abelian, so that
the two  functors  in (\ref{H1NatTrRsym}) 
restricted to $\mbf{Man^\point}\times\mbf{Fol}_\mbf{ft}$ are  isomorphic and  take values in the category $\mbf{Ab}$ of abelian groups, which can be seen as a subcategory of $\mbf{Set}^\point$ by pointing by zero, see Section~\ref{cohomology}. 

\subsection{Sheaf of infinitesimal transversal symmetries}\label{subsecinfsym}
Given a foliation $\F$ let us consider now the following sheaves  $\un {\mc X}_{\F}\subset\un{\mc B}_{\F}$ over $\mc E_\F$  of \emph{tangent} and \emph{basic} holomorphic vector fields of $\F^\sharp$: the stalk $\un{\mc B}_{\F}(m)$ of $\un{\mc B}_{\F}$ at $m\in\mc E_\F$ is the $\C$-vector space of germs at $m$ of holomorphic vector fields in $M_\F$  leaving invariant the foliation $\F^\sharp$ and the divisor $\mc E_\F$; $\un {\mc X}_{\F}(m)$ is the subspace of $\un{\mc B}_{\F}(m)$ consisting of vector fields tangent to $\F^\sharp$.  The quotient sheaf $\un {\mc T}_{\F}:= \un{\mc B}_{\F} / \un {\mc X}_{\F}$ is called \emph{sheaf of infinitesimal transversal symmetries of~$\F^\sharp$}.\\

Similarly, given $Q^\point=(Q,u_0)$ a germ of manifold, we  define  $\un{\mc B}^{Q^\point}_{\F}$ the sheaf over $\mc E_\F$ of $\mc{O}_{Q,u_0}$-modules whose stalks  are the spaces $\un{\mc B}^{Q^\point}_{\F}(m)$ of germs at $(m,u_0)$ of holomorphic  vector fields in $M_\F\times Q$ leaving invariant the constant foliation $\F^{\mr{ct}\,\sharp}_{Q^\point}$ and the divisor $\mc E_\F\times Q$, that are vertical (i.e. tangent to the fibers of the projection $M_\F\times Q\to Q$) and zero on the special fiber $M_\F\times \{u_0\}$;   $\un {\mc X}^{Q^\point}_{\F}\subset\un{\mc B}^{Q^\point}_{\F}$ is the subsheaf  consisting of vector fields which are tangent to $\F^{\mr{ct}\,\sharp}_{Q^\point}$ 
and the quotient sheaf
\[ 
\un {\mc T}^{Q^\point}_{\F}:= \un{\mc B}^{Q^\point}_{\F} / \un {\mc X}^{Q^\point}_{\F}
 \]
is called  the \emph{sheaf of infinitesimal transversal symmetries of $\F^{\mr{ct}\,\sharp}_{Q^\point}$}.
Notice that,  
if as usual we denote by $\mf M_{Q,u_0}$ the maximal ideal of $\mc O_{Q,u_0}$, we have:
\[\un{\mc B}_{\F}^{Q^\point}\otimes_{\mc{O}_{Q,u_0}}(\mc{O}_{Q,u_0}/\mf M_{Q^\point})=\{0\}\neq\un{\mc B}_{\F}\,.\] 

We will give local expressions for the stalks $\un {\mc T}_{\F}(m)$ and  $\un {\mc T}^{Q^\point}_{\F}(m)$ at a point $m$ in an  invariant component $D$ of $\mc E_\F$. 
Let us fix in $M_\F$ a  local chart $z=(z_1, z_2):\Omega\iso \mb D_{r}^{2}$
satisfying
\begin{equation*}%\label{GoodChart}
r>1\,,\quad 
z(m)=(0,0)\,,\quad
D=\{z_2=0\}\,,\quad
\mc E_\F=\{z_1^\epsilon z_2=0\}\,,\quad \epsilon\in\{0,1\} \,.
\end{equation*}
We  suppose that $\overline{\Omega}\cap\mr{Sing}(\F^\sharp)$ is either empty or reduced to $\{m\}$. 
We also fix a chart $u : \Omega'\iso\mb D_{\eta}^q$, $\eta>0$,   on $Q^\point$ with $u(u_0)=0$.

Let us denote by  $V_m$, resp. by  $V_m^{Q^\point}$,   the space of germs of vector fields $Z$ in  the submanifold $\{z_1=1\}$ of 
$\Omega$, resp. of $\Omega\times \Omega'$, at the point of coordinates $(1,0)$,  resp.  $(1,0,\ldots,0)$,  
that satisfy: (a) $Z=0$ when $z_2=0$, and (b) $h_{m\ast}(Z)=Z$ where $h_m$ is the classical holonomy map of $\F^\sharp$, resp. of $\F_{Q^\point}^{\mr{ct}\,\sharp}$, along the loop  $z(t)=(e^{2\pi t i}, 0) $,  resp. $z(t)=(e^{2\pi t i}, 0)$, $u(t)=0$, $t\in[0,1]$, realized on the transverse manifold $\{z_1=1\}$.

If $Y$ is a vector field on an open set $U\subset M_\F$ we will consider the \emph{constant vertical extension}  $Y^{\mr{ct}}_{Q^\point}$ on $U\times Q$, i.e. the unique vertical vector field on $U\times Q$  related to $Y$ by the projection $U\times Q\to U$.

\begin{lema}\label{modeles-locaux}
Assume that $\F$ is a generalized curve.
 With the previous notations we have:
 \begin{enumerate}
  \item\label{nResnLin}  if $\F^\sharp$  at $m$  is singular and it is either (a) non-resonant, non-linearizable but formally linearizable or (b) resonant  non-formally linearizable nor normalizable, then:
 \[
  \un {\mc T}_{\F}(m) = \{0\}\,,
\quad 
\un {\mc T}^{Q^\point}_{\F}(m)=\{0\}\,,
\quad
V_m=  \{0\}\,,
\quad
V_m^{Q^\point}= \{0\}
\,;\]
 \item\label{casnIP} if $\F^\sharp$ at $m$ is not as in case (\ref{nResnLin}) and any germ of holomorphic first integral of $\F^\sharp$ at $m$ is constant, then we may choose the coordinates $z_1, z_2$ so that 
  \[
 \un {\mc T}_{\F}(m) = \C\left[Z \right]\,,
\quad 
\un {\mc T}^{Q^\point}_{\F}(m)= \mf M_{Q,u_0}\left[Z^{\mr{ct}}_{Q^\point} \right]\,,
 \]
\[
V_m=  \C\cdot Z|_{\{z_1=1\}}\,,
\quad
V_m^{Q^\point}=  \mf M_{Q,u_0}\cdot Z^{\mr{ct}}_{Q^\point}|_{\{z_1=1\}}\,,
\]
where $Z^{\mr{ct}}_{Q^\point}|_{\{z_1=1\}}$ denotes the restriction of $Z^{\mr{ct}}_{Q^\point}$ to ${\{z_1=1\}}$ and $Z$ is the following vector field on $\Omega$:
\begin{enumerate}
\item\label{caslin}  {$Z= z_2\DD{z_2}$ } when   $\F^\sharp$  is  linearizable at $m$, %
\item\label{normres} $Z= \frac{(z_1^az_2^b)^k}{1+\zeta(z_1^az_2^b)^k}z_2\DD{z_2} $ when    $\F^\sharp$ is  singular resonant normalizable at $m$, and    $z_1, z_2$ is chosen so that   $\F^\sharp$ is   given by $\omega=0$ where
\[\omega:=
bz_1(1+\zeta (z_1^az_2^b)^k)dz_2 + az_2(1+(\zeta-1) (z_1^az_2^b)^k)dz_1,
\]
with $ a, b, k\in \N^\ast$, $(a,b)=1$, $\zeta\in \C$;
\end{enumerate}
 \item\label{ChTransIP}  if  
 $\F^\sharp$ at $m$ has a non-constant first integral $F$,  then by choosing   $F$ minimal and $z_1, z_2$ such that  $F(z_1,z_2)=z_1^a z_2^b$, $a, b\in \N$, $b\neq 0$, $(a,b)=1$, 
 we have:
 \[
 \un {\mc T}_{\F}(m) = \C\{F\}\left[ z_2\DD{z_2} \right]\,,
\quad 
\un {\mc T}^{Q^\point}_{\F}(m)= \mf M_{Q,u_0}\C\{F,u\}\left[ z_2\DD{z_2} \right]\,,
 \]
 \[\hbox{ and: }\quad
 V_m= \C\{z_2^b\} \cdot  z_2\DD{z_2}\,,
 \quad
 V_m^{Q^\point}=  \mf M_{Q,u_0}\C\{ z_2^b,u\}\cdot  \left. z_2\DD{z_2}\right|_{\{z_1=1\}}\,.
 \]
 \end{enumerate}
\end{lema}
\begin{proof}  
Classically 
$\un{\mc T}_{\F}(m)$ and $V_m$ are  zero   except when $\F^\sharp$ is either regular, or linearizable or resonant normalizable. In these last cases 
$\un{\mc T}_{\F}(m)$  is a  free module of rank one   over the ring $\mc O_{\F^\sharp,m}\subset\mc O_{M_{\F,m}}$ of germs of  holomorphic first integrals (perhaps constant) of $\F^{\sharp}$. We deduce the expression of $\un{\mc T}_{\F}(m)$ after checking that the vector fields $Z$ in (\ref{casnIP})  and $z_2\DD{z_2}$ in (\ref{ChTransIP}) are basic and 
$\mc O_{\F^\sharp,m}=\C$, resp. $\mc O_{\F^\sharp,m}=\C\{F\}$,  in case (\ref{casnIP}), resp. (\ref{ChTransIP}),  cf. \cite[\S5.1.2]{MS}. The expressions of $\un{\mc T}^{Q^\point}_{\F}(m)$ are versions with parameters of these results.

In the cases (\ref{caslin}) and (\ref{ChTransIP}) the holonomy map $h_m$ is linear and  $V_m$ and $V_m^{Q^\point}$ is obtained by a direct computation. In order to obtain $V_m^{Q^\point}$ in case (\ref{normres}) one first notices that the flow $\Phi_t(z_2,u)=(\phi(z_2,t), u)$ of $Z^{\mr{ct}}_{Q^\point}|_{\{z_1=1\}}$  satisfies  $\phi(z_2,t)\in \C[t]{\{z_2\}}$; therefore any biholomorphism germ that commutes with a single element of this flow also  commutes with all the other elements. Since $h_m^{\circ q}=\Phi_{2i\pi q}$, the flow of any element  $X\in V_m^{Q^\point}$ commutes with that of $Z^{\mr{ct}}_{Q^\point}|_{\{z_1=1\}}$.  It follows that 
$X\in \mf{M}_{Q,u_0} Z^{\mr{ct}}_{Q^\point}|_{\{z_1=1\}}$.
\end{proof}

In order to describe $\un{\mc T}_{\F}^{Q^\point}(U)$ for any open set $U\subset D$,  we fix a geometric system as  in Definition~\ref{geometric-system}.
 
For any  $X\in\un {\mc B}_{\F}( o_D)$ there is a holomorphic  vector field $g_\ast(X)$ on $(\C,0)$ such that $X(0)=0$ and $g_\ast(X)\circ g=Dg(X)$. Moreover, $X\mapsto g_\ast(X)$ is  $\C$-linear. 
Let us adopt the following notations:
\begin{itemize}
\item $\mc V(H)$ is the vector space of holomorphic germs of vector fields on $(\C,0)$ vanishing at $0$ and invariant under the action of the subgroup $H\subset\mr{Diff}(\C,0)$;
\item $\mathcal V^0_{Q^\point}$ is the vector space of holomorphic germs of vector fields on $(\C\times Q,(0,u_0))$ which are vertical with respect to $\C\times Q\to \C$ and  vanish along $(\{0\}\times Q)\cup(\C\times\{u_0\})$;
\item If $G\subset\mr{Diff}_Q(\C\times Q,(0,u_0))$ is a subgroup, then $\mathcal V_{Q^\point}^0(G)$ denotes the subspace of $\mathcal V^0_{Q^\point}$ consisting of vector fields invariant by $G$.
\end{itemize}

Similarly if $X\in\un {\mc B}^{Q^\point}_{\F}( o_D)$ there is a (unique)  germ of vector field, again denoted by $g_{\ast}(X)$, such that $g_\ast (X)\circ (g\times\mr{id}_{Q^\point})=D(g\times\mr{id}_{Q^\point})(X)$. According to the model~(\ref{ChTransIP}) with $a=0$ and $b=1$ in Lemma \ref{modeles-locaux}   we have the following exact sequence:
\begin{equation}\label{flatoD}
   0\to \un {\mc X}^{Q^\point}_{\F}( o_D)\to \un {\mc B}^{Q^\point}_{\F}( o_D)\stackrel{g_\ast}{\to} 
 \mathcal V^0_{Q^\point}\to 0\,.
\end{equation}
This proves that  the sheaf $ \un {\mc T}_{\F}$ is locally constant on $ D \setminus \mr{Sing}(\F^\sharp)$.

\begin{obs}\label{analytic-continuation}
Let $X$ be a section of  $\un\T_{\F}$ over a connected open subset $V$ of $\mc E_\F$. If the germ of $X$ at some point $p$ of $V$ is zero, then $X=0$.
Indeed if $p$ is a regular point, by local triviality, the section is zero along the whole regular part of $D$. The vanishing at the remaining singularities follows by analytic continuation.
If $p$ is a singular point, then the germ of $X$ at a  regular point close to $p$ is zero and we conclude as before. The same property holds for $\T_{\F}^{Q^\point}$.
\end{obs}

\begin{obs}\label{monoholo}
The monodromy of  $\un {\mc T}_{\F}^{Q^\point}$ restricted to $D\setminus\mr{Sing}(\F^\sharp)$ corresponds to the holonomy of the foliation $\F_{Q^\point}^{\mr{ct}\,\sharp}$ in the following sense: if $Z'$ is the extension of $Z\in   \un {\mc T}^{Q^\point}_{\F}( o_D)$ (as germ of a locally constant sheaf) along a loop $\gamma$ in $D^\ast$ with origin $ o_D$, then $g_\ast(Z')=h_{\gamma\,\ast}(g_\ast(Z))$, where $h_\gamma=\mc H_{D}^{\F^{\mr{ct}\,\sharp}_{Q^\point}}(\dot\gamma)$, see~(\ref{holonomiefam}). Indeed we have: $g'=g \circ h_\gamma^{-1}$ and on the other hand,  since the expression $g_\ast(Z)$ remains constant when we  perform along $\gamma$ the analytic extension of $g$ and the extension of $Z$ as section of a locally constant sheaf, we also  have $g'_\ast (Z')=g_\ast(Z)$, where 
$g'$ is the analytic extension of $g$ along~$\gamma$. 
\end{obs}
 \begin{prop} Assume that $ o_D\in U\subset D$. The following sequence is exact:
 \begin{equation}\label{FlatU}
  0\to \un {\mc X}^{Q^\point}_{\F}(U)\to 
  \un {\mc B}^{Q^\point}_{\F}(U)\stackrel{g_{U\ast}}{\to} 
 \mathcal V^0_{Q^\point}(H_U)\to 0\,,
 \end{equation}
where $g_{U\ast}$ is the composition of the morphism $g_\ast$ in (\ref{flatoD}) with the natural map $ \un {\mc B}^{Q^\point}_{\F}(U)\to \un {\mc B}^{Q^\point}_{\F}( o_D)$ and  $H_U:= \mc H^{\F^{\mr{ct}\,\sharp}_{Q^\point}}_D(\pi_1(U\setminus\mr{Sing}(\F^\sharp), o_D))$. 
 \end{prop}
If $\ge=\langle D,D'\rangle$ and $\{p\}=D\cap D'$,
by applying this proposition to $U=U_p$, and  to $U=D$ we obtain isomorphisms
\begin{equation}\label{g*T}
G^\T_{D,\ge}:\T_{\F}^{Q^\point}(\ge)\iso \mc V^0_{Q^\point}(h_{D,\ge})\quad\text{and}\quad G^\T_D:\T_{\F}^{Q^\point}(D)\iso\mc V^0_{Q^\point}(H_D)\,.
\end{equation}
Under these isomorphisms the restriction map $\T_{\F}^{Q^\point}(D)\to\T_{\F}^{Q^\point}(\ge)$ corresponds to the inclusion $\mc V^0_{Q^\point}(H_D)\hookrightarrow\mc V^0_{Q^\point}(h_{D,\ge})$.

\begin{proof} The fact that  $g_{U\ast}$ takes values in  
 $\mathcal V^0_{Q^\point}(H_U)$ results from  Remark~\ref{monoholo} which also gives the exactness of the sequence when $U$ does not meet $\mr{Sing}(\F^\sharp)$.  It remains  to see that the restriction map
 \begin{equation*}\label{restrmap}
  \un {\mc T}^{Q^\point}_{\F}(U)\to\un {\mc T}^{Q^\point}_{\F}(U\setminus\mr{Sing}(\F^\sharp))\,,
  \quad
  Z\mapsto Z|_{U\setminus\mr{Sing}(\F^\sharp)}
 \end{equation*}
 is an isomorphism.   We may suppose that $U$ is a   disk such that $U\cap \mr{Sing}(\F^\sharp)=\{m\}$. Thus,  the map $g_{U\ast}$ in (\ref{FlatU}) induces an isomorphism 
 \[   
  \un {\mc T}^{Q^\point}_{\F}(U\setminus\mr{Sing}(\F^\sharp))\iso  \mathcal V^0_{Q^\point}(H_U)\,.
  \]
  We may also suppose that $U$ is  the domain $\Omega$  of a chart $(z_1,z_2)$ as in   Lemma~\ref{modeles-locaux}. The restriction of $g\times \mr{id}_{Q^\point}$ to $\{z_1=1\}\subset M_\F\times Q^\point$ induces a linear isomorphism  from $V_m^Q$ to $\mc V^0_Q(H_U)$. We conclude by noting that,  according to  Lemma \ref{modeles-locaux},  any  element of  $V_m^Q$ extends to a  vector field in  $\un {\mc B}^{Q^\point}_{\F}(U)$.
 \end{proof}
 In the same way we prove the exactness of the following sequence:
 \begin{equation}\label{FlatUT}
  0\to \un {\mc X}_{\F}(U)\to 
  \un {\mc B}_{\F}(U)\stackrel{g_{U\ast}}{\to} 
 \mathcal V(H_U)\to 0\,.
 \end{equation}
 
 \subsection{Group-graph of infinitesimal transversal symmetries}\label{S53}
 A \Cex-conjugacy does not induce a map between the sheaves of basic holomorphic vector fields, but it  will do for the sheaves of transverse infinitesimal symmetries. For this reason we do not consider the quotient of the group-graphs associated to $\un{\mc B}_{\F}^{Q^\point}$ and $\un{\mc X}_{\F}^{Q^\point}$ but  a group-graph $\T_{\F}^{Q^\point}$  associated to the sheaf~$\un\T_{\F}^{Q^\point}$. As in the case of the group-graph of automorphisms (see Definition~\ref{defAutQ}) we set:

 \begin{defin}\label{defInfSymTrans}
The \emph{vector space-graph over $\A_\F$ of infinitesimal transversal symmetries} of $\F$, resp. {of $\F^{\mr{ct}}_{Q^\point}$},  denoted by $\mc T_{\F}$, resp. $\mc T_{\F}^{Q^\point}$, is defined, for $\star\in\Ve_{\A_\F}\cup\Ed_{\A_\F}$, by:
\begin{enumerate}
\item\label{trivdicn} $\mc T_{\F}(\star)=\{0\}$ and $\mc T_{\F}^{Q^\point}(\star)=\{0\}$ if $\star\in \Ve_{\A_\F}$ is a dicritical component of $\mc E_\F$ or $\star=\langle D,D'\rangle\in \Ed_{\A_\F}$ and the foliation $\F^\sharp$ has a nodal singularity at the point $D\cap D'$;
\item\label{dic} $\mc T_{\F}(D)=\un{\mc T}_{\F}(D)$
and $\mc T_{\F}^{Q^\point}(D)=
\un{\mc T}_{\F}^{Q^\point}(D)$
if $D\in\Ve_{\A_\F}$ is invariant;
\item\label{nod} $\mc T_{\F}(\langle D,D'\rangle)=\un{\mc T}_{\F}(D\cap D')$
and $\mc T_{\F}^{Q^\point}(\langle D,D'\rangle)=
\un{\mc T}_{\F}^{Q^\point}(D\cap D')$
if $\langle D,D'\rangle\in\Ed_{\A_\F}$ and $D\cap D'$ is not a nodal singularitiy of $\F^\sharp$;
\item  the restriction map $\mc T_{\F}^{Q^\point}(D)\to \mc T_{\F}^{Q^\point}(\msf e)$ is the trivial map $\mc T_{\F}^{Q^\point}(D)\to \{0\}$ in case (\ref{trivdicn}) and it is the restriction map  of  sheaves in cases (\ref{dic}) and (\ref{nod}).
\end{enumerate}
\end{defin}
\noindent The support of $\T_\F^{Q^\point}$ is contained in the cut-graph of $\F$ which is the support of $\aut_\F^{Q^\point}$, see Remark~\ref{suppAut}.

The pull-back by a holomorphic map germ $\mu:P^\point\to Q^\point$ of a vertical vector field $X$ is also a vertical vector field and its flow is the pull-back of the flow of $X$.
%\footnote{$\mr{exp}[t](\mu^\ast Z)=\mu^\ast(\mr{exp}[t](Z))$}
 Thus, the pull-back operation defines sheaf morphisms  from the sheaves $ \un {\mc B}^{Q^\point}_{\F}$,
 $ \un {\mc X}^{Q^\point}_{\F}$ and $ \un {\mc T}^{Q^\point}_{\F}$ respectively  to the sheaves $ \un {\mc B}^{P^\point}_{\F}$, $ \un {\mc X}^{P^\point}_{\F}$ and $ \un {\mc T}^{P^\point}_{\F}$, inducing  a morphism of vector space-graphs 
\[
\mu^\ast :  \mc T_{\F}^{Q^\point}\to\mc T_{\F}^{P^\point}\,.
\]

On the other hand, let $\phi$ be an \Cex-conjugacy between $\G$ and a   foliation $\F$, $\phi(\G)=\F$. 
Since the germs of homeomorphisms $\phi^\sharp :(M_\G,\mc E_\G)\iso(M_\F,\mc E_\F)$ and 
\[
\phi_Q^\sharp:=\phi^\sharp\times\mr{id}_Q :(M_\G\times Q,\mc E_\G \times\{u_0\})\iso(M_\F,\mc E_\F\times\{u_0\})
 \]
are holomorphic at the singular points and transversely holomorphic elsewhere, we can define the inverse image morphisms of  sheaves
 over $\mc E_\G$
 \[
  \un\phi^*:\phi_{\ltE}^{-1}\un\T_{\F}\to\un\T_{\G}\,\quad\hbox{and}\quad 
  \un{\phi}_Q^\ast:\phi_{\ltE}^{-1}\un\T_{\F}^{Q^\point}\to\un\T_{\G}^{Q^\point}\,,
  \]
  where $\phi_\ltE:\mc E_\G\to\mc E_\F$ 
  is the restriction of $\phi^\sharp$ to the exceptional divisors, as in Section~\ref{sec:aut}.
Indeed, let us fix  $m\in\mc E_\G$ and $[Z]\in\un\T_{\F}(\phi_{\ltE}(m))$, which is the class of $ Z\in\un{\mc B}_{\F}(\phi_{\ltE}(m))$. If $m\in \mr{Sing}(\G^\sharp)\cup\mr{Sing}(\mc E_\G)$ 
 then $\phi^\sharp$ is holomorphic at $m$ and we define
 $\un\phi^*([Z])$ as the class of the usual inverse image $(\phi^\sharp)^*(Z)\in\un{\mc B}_\G(m)$.
Otherwise, there is  a homeomorphism germ $\xi$ at $\phi^\sharp(m)$ fixing the leaves of $\F^\sharp$ such that $\xi\circ\phi^\sharp$ is holomorphic and we define 
$\un\phi^*([Z])$ as the class of $(\xi\circ\phi^\sharp)^{\ast}(Z)$, which does not depend on the choice of~$\xi$. We can similarly define the sheaf morphism~$\un\phi_Q^\ast$.

We will denote in the same way by 
\begin{equation}\label{phisastT}
\phi^{\ast}: \mc T_{\F}\iso \mc T_{\G}
\quad
\hbox{ and }
\quad
\phi^{\ast} : \mc T_{\F}^{Q^\point}\iso \mc T_{\G}^{Q^\point}
\end{equation}
the vector space-graph morphisms over $\A_{\phi}:\A_\G\to\A_\F$ defined in  (\ref{Aphi}), which are associated  to the sheaf morphisms $\un\phi^*$ and $\un\phi^*_Q$, see Section \ref{gg-sh}.

We can check that the second morphism $\phi^*$ satisfies the relations   $\mu^*\circ\phi^*=\phi^*\circ\mu^*$,  allowing us to define the following contravariant functors (denoted by the same letter)
\[\T:\fol\to\VSG\,,\qquad \F\mapsto\T_{\F}\,,\qquad\phi\mapsto\phi^*\,,\]
\begin{equation*}\label{functorT}
\mc T : \mbf{Man^\point}\times\fol \to\VSG\,,
\qquad
(Q^\point,\F)\mapsto \mc T_{\F}^{Q^\point}\,,
\quad
(\mu,\phi)\mapsto (\mu,\phi)^*:=\phi^\ast \circ\mu^\ast\,,
\end{equation*}
where $\VSG$ denotes the category of  of $\C$-vector space-graphs and linear maps.\\

As we did for the group-graph of transversal symmetries we consider the restriction of infinitesimal transversal symmetries vector space-graphs to the red subgraph $\AR_\F\subset\A_\F$:

\begin{defin}\label{redT}
We call \emph{restricted group-graph of infinitesimal transversal symmetries of $\F$, resp. $\F_{Q^\point}^{\mr{ct}}$}, the  group-graph 
$\mr{R\T}_\F=r_\F^*\T_\F$, resp.
${\mr{R\T}}^{Q^\point}_{\F}=r_\F^*\T_\F^{Q^\point}$,
over $\AR_\F$ defined as the pull-back by the inclusion $r_\F:\AR_\F\hookrightarrow\A_\F$:
\[
\mr{R\T}_\F(\star)=\T_\F(\star),\quad \text{resp. }
{\mr{R\T}}^{Q^\point}_{\F}(\star)={\T}^{Q^\point}_{\F}(\star)\,,\qquad \star\in \Ve_{\AR_{\F}}\cup\Ed_{\AR_{\F}}\,.
\]
We denote by $\mr{R}\T:\fol\to\VSG$, resp. $\mr{R}\mc T: \mbf{Man^\point}\times\fol \to \VSG$, the functors $\F\mapsto \mr{R}\T_{\F}$, resp.  $(Q^\point, \F)\mapsto {\mr{R}\mc T}^{Q^\point}_{\F}$. 
\end{defin}

\begin{obs}\label{H1TRT} 
 As for transversal symmetries, the 
 collections of canonical morphisms $\imath_{r_\F}:\T_\F\to\mr{R}\T_\F$ and $\imath_{r_\F}:\T_\F^{Q^\point}\to\mr{R}\T_\F^{Q^\point}$ of vector space-graphs over the graph morphisms $r_\F:\AR_\F\hookrightarrow\A_\F$ define
 natural transformations, again denoted by
\begin{equation*}
R : \mc{T}\to \mr{R}\mc{T}\,, 
\quad
\hbox{and also}
\quad 
\mc R:=H^1(R) : H^1\circ \mc{T}\to H^1\circ \mr{R}\mc{T}\,.
\end{equation*}
If $\F$ is of finite type, thanks to the exact sequence  (\ref{FlatUT}), in each cut-component $\A_\F^\alpha$ of $\A_\F$ the red part $\AR_\F^\alpha$ is  repulsive  for the group-graph $\mc{T}_{\F}$ restricted to $\A_\F^\alpha$, see Section~\ref{Subsecfunctpruning}.
By applying again (\ref{h1decSoDir}) and Theorem~\ref{pruning} we directly obtain that the natural maps 
\begin{equation}\label{isoRTRTT}
\mc R_{\F} : 
H^1(\A_\F,{\mc{T}}_{\F})^\point\iso H^1(\AR_\F,\mr{R}{\mc{T}}_{\F})^\point\,.
\end{equation}
are bijective, thus   $\mc R$ is an isomorphism of contravariant functors.
 In the same  way,  using the exact sequence~(\ref{FlatU})
 we obtain a  natural isomorphism
\begin{equation*}
\mc R_{\F}^{Q^\point} : 
H^1(\A_\F,{\mc{T}}^{Q^\point}_{\F})^\point\iso H^1(\AR_\F,\mr{R}{\mc{T}}^{Q^\point}_{\F})^\point\,.
\end{equation*}
\end{obs}

\begin{lema}\label{multext}
Assume that $\F$ is a generalized curve.
Let  us again denote  by $Z^{\mr{ct}}_{Q^\point}$  the constant vertical extension of a vector field $Z$ on an open set of $M_\F$, defined just before Lemma~\ref{modeles-locaux}.
The extension of scalars sheaf morphism\footnote{We highlight that $\un{\mr{Ext}}_{\F}^{Q^\point}$ is not an isomorphism of sheaves.} 
\[\un{\mr{Ext}}_{\F}^{Q^\point}
: \un{\mc T}_{\F}
\otimes_\C \mf{M}_{Q^\point}\to \un{\mc T}_{\F}^{Q^\point}\,,
\quad
[Z]\otimes a\mapsto [aZ^{\mr{ct}}_{Q^\point}]\,,
\]
 define an isomorphism of vector space-graphs 
\[\mr{Ext}_{\F}^{Q^\point} : 
\mr{R}\mc T_{\F}\otimes_\C \mf{M}_{Q^\point}\iso\mr{R} \mc T_{\F}^{Q^\point}
\]
which induces a natural isomorphism $\mr{Ext}$ between the contravariant functors $(Q^\point,\F)\mapsto \mr{R}\T_{\F}\otimes_\C\mf M_{Q^\point}$ and $(Q^\point,\F)\mapsto \mr{R}\T_{\F}^{Q^\point}$,  from $\mbf{Man^\point}\times\mbf{Fol}$ to   $\VSG$.
\end{lema}
\noindent In this way we obtain a natural isomorphism
\begin{equation}\label{H1M}
H^1(\mr{Ext}^{-1})\,:\,H^1(\AR_\F,\mr{R}\mc T_{\F}^{Q^\point})^\point \iso  H^1(\AR_\F,\mr{R}\mc T_{\F}\otimes_\C \mf M_{Q^\point})^\point
\end{equation}
between functors from $\mbf{Man^\point}\times\mbf{Fol}$ to   $\mbf{Vec}$, as subcategory of pointed sets.

\begin{proof}
Consider an invariant component $D$ of $\mc E_\F$, an edge $\ge=\langle D,D'\rangle$ and the point $\{p\}:=D\cap D'$. Assume that $D$ and $\ge$ are red for $\F$.
We can then use the isomorphisms~(\ref{g*T}), the exact sequence~(\ref{FlatUT}) 
with $U=U_p$ as in Definition~\ref{geometric-system},
and cases (\ref{nResnLin}) and (\ref{casnIP}) in Lemma \ref{modeles-locaux}. With the notations used in this lemma and these sequences, we have the following commutative diagrams whose vertical arrows are isomorphisms:
\[
\xymatrix{
\mc T_{\F}(D)\otimes_\C \mf M_{Q,u_0}\ar[d]_{\wr}^{\dot g_{D\ast}\otimes_\C{\mr{id}_{\mf M_{Q,u_0}}}}\ar[rr]^<(.35){\un{\mr{Ext}}_{\F}^{Q^\point}(D)} && \mc T_{\F}^{Q^\point}(D)\ar[d]_{G_{D}^\T}^{\wr}
\\
\mc V(H_D)\otimes_\C \mf M_{Q,u_0}\ar[rr]^<(.35){\mr{Ext}(D)}&&\mc V_{Q^\point}^0(H_D)
}
\quad
\xymatrix{
\mc T_{\F}( \ge)\otimes_\C \mf M_{Q,u_0}\ar[d]_{\wr}^{\dot g_{U\ast}\otimes_\C{\mr{id}_{\mf M_{Q,u_0}}}}\ar[r]^<(.2){\un{\mr{Ext}}_{\F}^{Q^\point}(p)} & \mc T_{\F}^{Q^\point}( \ge)\ar[d]_{G^\T_{D,\ge}}^{\wr}
\\
\mc V(h_{D,\ge})\otimes_\C \mf M_{Q,u_0}\ar[r]^<(.2){\mr{Ext}(\ge)}&\mc V_{Q^\point}^0(h_{D,\ge})
}
\]
where  $\mr{Ext}(D)$ and $\mr{Ext}(\ge)$ are   the maps $Z\otimes_\C a \mapsto aZ^{\mr{ct}}_{Q^\point}$. 
To prove that the top horizontal arrows are isomorphisms it suffices 
to prove this property for the bottom arrows. 
%First we prove that the bottom arrows are isomorphisms. 
Since the holonomy of the constant deformation ``does not depend on the parameter''  this fact directly results from the definitions of $\mc V_{Q^\point}^0(H_D)
$ and $\mc V^0_{Q^\point}(h_{D,\ge})$ and $\dim_\C \mc V(h_{D,\ge}),\dim_\C\mc V(H_D)\le 1$. 
%Hence the top horizontal arrows are also isomorphisms.

Finally, this collection of isomorphisms induces the  isomorphism of functors $\mr{Ext}$ since 
$\mu^*\phi^*([a Z^{\mr{ct}}_{Q^\point}])=(\mu^*a)(\phi^*([Z^{\mr{ct}}_{Q^\point}]))$ for any morphism $(\mu,\phi)$
in the category $\mbf{Man^\point}\times\mbf{Fol}$.
\end{proof}

\begin{prop}\label{RemquchTF} Assume that $\F$ is a generalized curve.
The vector space-graphs $\mr{R}\T_{\F}$ and $\mr{R}\T_{\F}^{Q^\point}$ over $\AR_\F$ are regular (see 
Definition~\ref{regular}). 
Moreover, in each red subgraph $\AR^\alpha_\F\subset\A^\alpha_\F$
 the complementary of its support  is a subgraph. 
\end{prop}
\begin{proof}
By Lemma~\ref{modeles-locaux} for each $\star\in\Ve_{\AR_\F}\cup\Ed_{\AR_\F}$, either both $\mr{R}\T_{\F}(\star)$ and $\mr{R}\T_{\F}^{Q^\point}(\star)$ are  zero, or $\mr{R}\T_{\F}(\star)$ is isomorphic to $\C$ and $\mr{R}\T_{\F}^{Q^\point}(\star)$ is isomorphic to  the maximal ideal $\mf M_{Q,u_0}$ of $\mc O_{Q,u_0}$.
Assume that $D\in \Ve_{\AR_\F}$ is invariant and $\msf e=\langle D,D'\rangle\in \Ed_{\AR_\F}$ does not correspond  to a nodal singular point at $D\cap D'$. By Remark~\ref{analytic-continuation}  either the restriction map
$\mr{R}\T_{\F}(D)\to \mr{R}\T_{\F}(\msf e)$ is an isomorphism or
$\mr{R}\T_{\F}(D)=0$ and  $\mr{R}\T_{\F}(\msf e )\simeq\C$,
the situation  $\mr{R}\T_{\F}(D)\neq 0$ and $\mr{R}\T_{\F}(\msf e)=0$ being impossible.
According to Lemma~\ref{multext}, we deduce that $\mr{R}\T_{\F}^{Q^\point}\simeq\mr{R}\T_{\F}\otimes_\C\mf M_{Q,u_0}$ is also regular.
\end{proof}

\subsection{Exponential group-graph morphism} 
The flows of basic vector fields of $\F_{Q^\point}^{\mr{ct}\,\sharp}$  leave invariant the foliation $\F_{Q^\point}^{\mr{ct}\,\sharp}$. As in  \cite[Lemma 9.1]{MMS} we  see that the exponential maps $\un{\mc B}_{\F}^{Q^\point}(m)\to  \un{\mr{Aut}}_{\F}^{Q^\point}(m)$, $Z\mapsto \exp(Z)[1]$, $m\in \mc E_\F$,
send $\un{\mc X}_{\F}^{Q^\point}(m)$ in $\un{\mr{Fix}}_{\F}^{Q^\point}(m)$ and 
factorize into maps $\exp_m:  \un{\mc T}_{\F}^{Q^\point}(m)\to  \un{\mr{Sym}}_{\F}^{Q^\point}(m)$, thus defining a morphism of sheaves of sets 
\begin{equation*}\label{ExpSheaves}
\un{\mr{Exp}}^{Q^\point}_{\F}:  \un{\mc T}_{\F}^{Q^\point}\to  \un{\mr{Sym}}_{\F}^{Q^\point}\,.
\end{equation*}
Using the isomorphism~(\ref{symD}) it induces
maps 
\[\mr{Exp}^{Q^\point}_{\F}(\star) : {\mc{T}}^{Q^\point}_{\F}(\star)\to\un{\mr{Sym}}_{\F}^{Q^\point}(\star)\simeq {\mr{Sym}}^{Q^\point}_{\F}(\star)\,,\quad\star\in \Ve_{\A_\F}\cup\Ed_{\A_\F}\,.\] 
In general  these maps are not group morphisms but this will be the case when the $\mc O_{Q, u_0}$-module
 ${\mc T}_{\F}^{Q^\point}(\star)$ is free of rank one or null, cf. \cite[\S9]{MMS}. Therefore to define an exponential group-graph morphism we must restrict the group-graph of infinitesimal symmetries of $\F$ or $\F_{Q^\point}^{\mr{ct}\,\sharp}$ to the
group-graph $\mr{R}\T_\F^{Q^\point}$ over the
 sub-graph $\msf R_{\F}$ of~$\A_\F$.

Using the definitions of the isomorphisms $G_\star^\T$ in (\ref{g*T}) and the definitions of the isomorphisms $G_\star$ in
Proposition~\ref{G} with the same geometric system, cf.  Definition~\ref{geometric-system},  we have the following commutative diagrams 
\begin{equation}\label{diagramsexp}
\xymatrix{
{\mc{T}}^{Q^\point}_{\F}(\msf e)\ar[r]_{\sim}^{G^\T_{D,\ge}}\ar[d]^{\mr{Exp}^{Q^\point}_{\F}(\msf e)}&\mathcal V^0_{Q^\point}(h_{D,\msf e})\ar[d]^{\exp}\\
{\mr{Sym}}^{Q^\point}_{\F}(\ge)\ar[r]_{\sim}^{G_{D,\ge}} & C^0_{Q^\point}(h_{D,\msf e})}
\qquad\quad
\xymatrix{
{\mc{T}}^{Q^\point}_{\F}(D)\ar[r]_{\sim}^{G^\T_D}\ar[d]^{\mr{Exp}^{Q^\point}_{\F}(D)}&\mathcal V_{Q^\point}^0(H_D)\ar[d]^{\exp}\\ 
{\mr{Sym}}^{Q^\point}_{\F}(D)\ar[r]_{\sim}^{\hspace{2mm}G_D} & C_{Q^\point}^0(H_D)
}
\end{equation}
where $\msf e=\langle D,D'\rangle$ and 
$h_{D,\msf e}$ is the holonomy map defined in (\ref{hHD}).
Indeed, when the direct image $g_*(Z)$ of a basic vector field $Z$ is defined, its flow is also the direct image of the flow of $Z$ by $g$.

\begin{teo}\label{exp} Given a  foliation which is a generalized curve and a germ of manifold~$Q^\point$, the
 morphisms  $\mr{Exp}^{Q^\point}_{\F}(\star)$ induce a  group-graph isomorphism over~$\AR_\F$
\[
\mr{Exp}^{Q^\point}_{\F}: \mr{R}\mc T^{Q^\point}_{\F} \iso \mr{RSym}^{Q^\point}_{\F}\,.
\]
The collection $\{\mr{Exp}_{\F}^{Q^\point}\}$
 defines an isomorphism of  contravariant functors
\[
\mr{Exp}: \mr{R}\mc T \iso \mr{RSym}\,,
\]
from $\mbf{Man^\point}\times\fol$ to the category of abelian group-graphs,
the functor $\mr{R}\T$ taking values in the subcategory of $\C$-vector space group-graphs.
\end{teo}

In order to prove this theorem we will need an auxiliary result.

\begin{lema}\label{cent}
If $h\in\mr{Diff}(\C,0)$ is non-periodic then the exponential map induces a group isomorphism 
\[
\exp:\mc V_{Q^\point}^0(h)\stackrel{\sim}{\to}C_{Q^\point}^0(h).\]
\end{lema}
\begin{proof}
If $h$ is formally linearizable then there is a formal coordinate
  $w$ such that $w\circ h=\lambda w$ with $\lambda\in\C^*$. If $\phi\in C^0_{Q^\point}(h)$ then $w\circ\phi_t=\nu(t)w$ with $\nu\in\mc O_{Q,u_0}^*$ and $\nu(0)=1$. Indeed, $\tilde\phi(w,t):=w\circ\phi_t=\sum_{i\geq 1}\phi_i(t)w^i$ belongs to $\C\{w,t\}$ and $\tilde\phi(\lambda w,t)=\lambda\tilde\phi(w,t)$ implies that  $\phi_i(t)\equiv 0$ for $i>1$ and $\nu(t)=\phi_1(t)\neq 0$ is holomorphic.
There is $\xi\in\mf M_{Q,u_0}$ such that  $\nu(t)=\exp(\xi(t))$. If $w$ is convergent then $\phi_t=\exp(\xi(t)w\partial_w)$. If $w(z)$ is divergent then $|\lambda|=1$ and $C^0_{Q^\point}(h)$ is the set of $\phi(z,t)=(\phi_t(z),t)$ such that $\phi_t(z)=w^{-1}\circ\nu(t)w(z)$ is convergent. 
If $w$ is divergent then $\mc V_{Q^\point}^0(h)=0$ and $\nu(t)$ takes values in a discrete subset of the unit circle $\mb S^1\subset\C$. We conclude that $\nu\equiv 1$ by holomorphy.

If $h$ is resonant there is a formal coordinate $w$ such that $w\circ h=\ell^r\circ \exp s X$ with $X=\frac{w^{p+1}}{1+\lambda w^p}\partial_w$ for some integer $p\geq 1$. If $\phi\in C^0_{Q^\point}(h)$ then
$w\circ\phi_t=\ell^{r_t}\circ\exp \tau(t)X$ with $r_t\in\Z$. 
Since $\tilde\phi(w,t)=w\circ\phi_t=\sum_{i\geq 1}\phi_i(t)w^i\in\C\{w,t\}$ and $\exp \tau X(w)=w+\tau w^{p+1}+\cdots$ we conclude that the holomorphic function $\phi_1(t)=e^{\frac{2i\pi r_t}{p}}$ is identically equal to $1$ and the function $t\mapsto\tau(t)$ is holomorphic and vanishes at $t=0$. If $w$ is convergent then $\phi_t=\exp(\tau(t)X)$ with $\tau\in\mf M_{Q,u_0}$. If $w$ is divergent then $\mc V^0_{Q^\point}(h)=0$ and $C^0_{Q^\point}(h)$ is the set of $\phi(z,t)=(\phi_t(z),t)$ such that $w^{-1}\circ\exp(\tau(t)X)\circ w$ is convergent. This implies that $\tau(t)\in \mb Q$ by the \'Ecalle-Liverpool's Theorem \cite{Liverpool} and consequently $\tau\equiv 0$. 
\end{proof}

\begin{dem2}{of Theorem~\ref{exp}}
It suffices to see that for $D\in\Ve_{\AR_{\F}}$ 
and $\ge\in\Ed_{\AR_{\F}}$ the right vertical arrows in the diagrams (\ref{diagramsexp}) are isomorphisms.
For the diagram in the left this follows from Lemma~\ref{cent} by taking $h=h_{D,\ge}$. 

It only remains to examine the diagram in the right of (\ref{diagramsexp}) with $D$ red.
Since $H_D$ is infinite there is a non-periodic element $h_0\in H_D$. Indeed, when $H_D$ is non-abelian  it contains a non-trivial commutator, which is tangent to the identity, hence non-periodic. When $H_D$ is abelian, if all its elements were periodic then $H_D$ would be finite.
We must prove that the exponential map $\exp:\mc V^0_{Q^\point}(H_D)\to C^0_{Q^\point}(H_D)$ is an isomorphism. By Lemma~\ref{cent} the bottom horizontal map in the following diagram is an isomorphism:
\[\xymatrix{\mc V^0_{Q^\point}(H_D)\ar[r]^{\mr{exp}}\ar@{^{(}->}[d]&C^0_{Q^\point}(H_D)\ar@{^{(}->}[d]\\ \mc V^0_{Q^\point}(h_0)\ar[r]^{\sim}_{\mr{exp}}&{\,}C^0_{Q^\point}(h_0)\,.}\]
This shows that the top horizontal map is injective. To prove the surjectivity we distinguish two cases:
\begin{enumerate}[(a)]
\item $\T_{\F}(D)=0$. In this case $\mc V_{Q^\point}^0(H_D)=0$.
By contradiction, we must see that if $C_{Q^\point}^0(H_D)\neq\{\mr{id}_{\C\times Q}\}$ then $\mc V^0_{Q^\point}(H_D)\neq \{ 0\}$. 
If $(f(z,u),u)\in C_{Q^\point}^0(H_D)\setminus\{\mr{id}_{\C\times Q}\}$ there is a holomorphic germ $\lambda:(\C,0)\to (Q,u_0)$ and $n\in\N^*$ such that
$z\not\equiv g(z,t):= f(z,\lambda(t))=z+t^na(z)\mod t^{n+1}$  with $a(z)\not\equiv 0$. For every $h\in H_D$ we have: $g(h(z),t)=h(g(z,t))$. Working modulo $t^{n+1}$ we deduce that $h(z)+t^na(h(z))=h(z)+h'(z)t^na(z)$, i.e. $a(h(z))=h'(z)a(z)$. This means that $0\not\equiv a(z)\partial_z\in\cap_{h\in H_D} \mc V(h)=\mc V(H_D)\neq 0$ and consequently $\mc V^0_{Q^\point}(H_D)\simeq\mc V(H_D)\otimes_\C\mf M_{Q,u_0}\neq 0$.
\item $\T_{\F}(D)\neq0$. In this case,  $0\neq\T_{\F}(D)\simeq\mc V(H_D)\subset\mc V(h_0)$ and since $h_0$ is non-periodic classically, we have: $\dim_\C\mc V(h_0)\le 1$. Consequently $\mc V(H_D)=\mc V(h_0)$ has dimension $1$ and
$\mc V_{Q^\point}^0(H_D)=\mc V(H_D)\otimes_\C\mf M_{Q,u_0}=\mc V(h_0)\otimes_\C\mf M_{Q,u_0}=\mc V_{Q^\point}^0(h_0)$. Using Lemma~\ref{cent} we have:
\[
C_{Q^\point}^0(h_0)=\exp(\mc V_{Q^\point} ^0(h_0))=\exp(\mc V_{Q^\point} ^0(H_D))\subset C_{Q^\point}^0(H_D)\subset C_{Q^\point}^0(h_0)\,.\]
Hence $\exp(\mc V^0_{Q^\point}(H_D))=C_{Q^\point}^0(H_D)$.
\end{enumerate}
We let the reader check that if 
$(\mu,\phi):(P^\point,\G)\to (Q^\point,\F)$ is a morphism in the category $\mbf{Man^\point}\times\fol$, then
the following diagram of group-graph morphisms is commutative:
\[\xymatrix{\mr{R}\T_\F^{Q^\point}\ar[r]^{(\mu,\phi)^*}\ar[d]_{\mr{Exp}_\F^{Q^\point}}&\mr{R}\T_\G^{P^\point}\ar[d]^{\mr{Exp}_\G^{P^\point}}\\
\mr{R}\sym_\F^{Q^\point}\ar[r]^{(\mu,\phi)^*} &\mr{R}\sym_\G^{P^\point}\,.}\]
\end{dem2}

\subsection{Characterization of finite type foliations}\label{subsecCharFT}
In this section we prove that, under a technical hypothesis, a  foliation $\F$ is of finite type if and only if the cohomology vector space $H^1(\A_\F,\T_{\F})$ is of finite dimension, which justifies the name that we have adopted.
\begin{teo}\label{characterization}
Let $\F$ be a  foliation which is a generalized curve.
If there is no cut-component $\A^\alpha_{\F}$ of $\A_\F$  entirely green, then $\F$ is of finite type if and only if 
 $\dim _{\C}H^1(\A_\F,\T_{\F})<\infty$.
\end{teo}

Before proving the theorem we need to state some auxiliary results.

\begin{obs}\label{sousgraphedim}
If $\K$, $\K'$   are subgraphs of $\A_{\F}$, 
then we have: \[\mr{dim}_\C H^1(\K', \mc T_{\F} )\leq \mr{dim}_\C H^1(\K, \mc T_{\F}) \hbox{ as soon as  } \K'\subset \K\,.\]
\end{obs}

\begin{lema} \label{codiminfiny}
If an edge $\ge\in\Ed_{\A_{\F}}$  is  green and $D\in \ge$ then the following properties are equivalent:
\begin{enumerate}
\item the holonomy group $H_D$ is generated by $h_{D,\ge}$;
\item the restriction morphism $\rho_{D}^{\ge}:\T_{\F}(D)\to\T_{\F}(\ge)$ is surjective;
\item the image of $\rho_{D}^{\ge}$ has finite codimension in $\mc T_{\F}(\ge)$;
\end{enumerate}
where $h_{D,\ge}$ are defined using a given geometric system, cf. Definition~\ref{geometric-system}.
\end{lema}
An immediate consequence of this lemma is the following:
\begin{cor}\label{codiminfiny2} If there is no cut-component  of $\A_\F$  entirely green, 
then $\F$ is of finite type if and only if in each cut-component $\A_{\F}^\alpha$ of $\A_\F$, $\alpha\in\mc A$, the red part $\AR_\F^\alpha$ is connected and repulsive for the group-graph $\T_{\F}$. 
\end{cor}
To lighten the text, in this proof we will denote  by $\mc T$ the vector space-graph $\mc T_{\F}$ and by  $\T_\star$ the vector space $\T_{\F}(\star)$.
\begin{dem2}{of Lemma~\ref{codiminfiny}} 
If $D$ is not green then $\dim_\C\T_D\in\{0,1\}$, $\dim_\C\T_{\ge}=\infty$, $h_{D,\ge}$ is periodic and $H_D$ is infinite. Thus, none of the three assertions hold. If $D$ is green then there is a transverse factor $z:(M_\F,o_D)\to(\C,0)$ at a regular point $ o_D\in D$ such that $H_D=\langle z\mapsto e^{2i\pi/n_D}z\rangle$ and $h_{D,\ge}(z)=e^{2i\pi/n_{D,\ge}}z$ where $n_D,n_{D,\ge}\in\Z$. 
The proof of \cite[Proposition 6.4]{MMS} shows that $\T_\ge/\rho^\ge_{D}(\T_D)\simeq
\C\{z^{n_{D,\ge}}\}/\C\{z^{n_D}\}$ is either zero (when $n_D=n_{D,\ge}$) or it has infinite codimension (when $n_D\neq n_{D,\ge}$).
\end{dem2}

Let us highlight that by Remark~\ref{analytic-continuation} 
the restriction maps $\rho^{\ge}_D: \mc T_D\to\mc T_\ge$, with $\ge=\langle D,D'\rangle\in \Ed_{\A_{\F}}$,  of the group graph $\mc T_{\F}$ are always  injective. We now provide "orientations" to  the edges $\ge$ of $\A_{\F}$ in the following way: 
\begin{enumerate}[(i)]
\item\label{cas1i} $\stackrel{D}{\circ} \; \stackrel{\ge}{\rightarrow} \;\stackrel{D'}{\circ}$ means that $\rho^\ge_D$  is not  bijective  and 
$\rho^{\ge}_{D'}$ is bijective,
\item $\stackrel{D}{\circ}  \;\stackrel{\ge}{\leftarrow} \; \stackrel{D'}{\circ}$ means that $\rho^\ge_D$ bijective  and
$\rho^\ge_{D'}$ is not bijective,
\item
$\stackrel{D}{\circ} \;\leftrightarrow \; \stackrel{D'}{\circ}$ means that both $\rho^\ge_D$  and
$\rho^\ge_{D'}$ are bijective,
\item\label{cas4i}
$\stackrel{D}{\circ} \;\bdf \; \stackrel{D'}{\circ}$ means that both $\rho^\ge_D$  and
$\rho^\ge_{D'}$  are not bijective.
\end{enumerate}

\begin{lema}\label{39mai}
In a cut-component $\A^\alpha_\F$ of $\A_\F$, let 
$\K$ be  a geodesic of  one of following types:
\begin{enumerate}
\item\label{type1} \; $\stackrel{D_0}{\bullet}  \;\stackrel{\ge_0}{\ftrait} \;  \stackrel{D_1}{\star} \;
\;\stackrel{\ge_1}{\ftrait}\cdots \stackrel{\ge_{n-1}}{\ftrait}
\;\; \stackrel{D_n}{\star}\; \stackrel{\,\ge_n}{\longleftarrow} \;
\stackrel{D_{n+1}}{\star}$, with $n \geq 1\,$;
\item \; $\stackrel{D_0}{\bullet} \;\;\stackrel{\ge_0}{\bdf}\;\;\stackrel{D_1}{\star}$, the edge $\ge_0$ being necessarily green;
\item \; $\stackrel{D_0}{\bullet} \; \stackrel{\ge_0}{\ftrait}\; \stackrel{D_1}{\star} \;
\;\stackrel{\ge_1}{\ftrait}\cdots \stackrel{\ge_{n-1}}{\ftrait}
\;\stackrel{D_n}{\star}\;
\stackrel{\ge_n}{\ftrait}\; \;\stackrel{D_{n+1}}{\bullet}$, with $n \geq 1\, ;$
\item\;  $\stackrel{D_0}{\bullet} \;\; \stackrel{\ge_0}{\ftrait} \;\; \stackrel{D_1}{\bullet}$, the edge $\ge_0$ being green;
\end{enumerate}
where the green vertices are denoted by $\star$, the red vertices by $\bullet$ and $\;\stackrel{D}{\circ}  \;\stackrel{\ge}{\ftrait} \;\stackrel{D'}{\circ}\;$ denotes any "orientation" (\ref{cas1i})-(\ref{cas4i}).
Then the dimension of $H^1(\K, \mc T)$ is infinite.
\end{lema}
\begin{proof}
First consider case (1). Thanks to Remark \ref{sousgraphedim}, even if we restrict to a smaller geodesic, we can suppose that all arrows $\ge_0,\ldots,\ge_{n-1}$ are either simple arrows directed to $D_n$, i.e. $\star_{D_{j-1}}\; \stackrel{\ge_{j-1}}{\longrightarrow}\;\star_{D_{j}}\,$ or double arrows 
$\star_{D_{j-1}}\stackrel{\ge_{j-1}}{\longleftrightarrow}\,{\star}_{D_{j}}\,$; therefore all the restriction maps $\rho_{D_j}^{\ge_{j-1}} :\mc T_{D_j}\to\mc T_{\ge_{j-1}}$, $j=0,\ldots, n$, are  isomorphisms. Every map $\rho_{D_j}^{\ge_{j}}$ being injective we can identify all the spaces $\mc T_{\ge_{j}}$, $j=0,\ldots n$, and $\mc T_{D_j}$, $j=0,\ldots, n+1$ with subspaces of $\mc T_{\ge_{n}}$. 
With these identifications we have:
\begin{equation}\label{identif2}
\mc T_{D_0}\subseteq\mc T_{\ge_0}=\mc T_{D_1}
\subseteq \cdots\subseteq
\mc T_{D_{n}}=\mc T_{\ge_n}\supsetneq\mc T_{D_{n+1}}\,.
\end{equation}
Since $D_{n+1}\in\ge_n$ are green, from Lemma \ref{codiminfiny} it follows that
\begin{equation}\label{codim-infini}
\mr{dim}_\C( \mc T_{\ge_n}/\mc T_{D_{n+1}})=\infty\,.
\end{equation}
With the identifications (\ref{identif2}) the coboundary morphism for the subgraph $\K$ can be written as 
\[\partial_\K^0: C^0(\K,\mc T)= \prod_{j=0}^{n+1}\mc T_{D_j}\longrightarrow Z^1(\K,\mc T)\iso\prod_{j=0}^n \mc T_{\ge_{j}}\,,\]
\[\partial_\K^0((X_j)_{j=0,\ldots,n+1})=(X_{j}-X_{j-1})_{j=1,\ldots,n+1}\,.
\]
The surjective linear map  \[\beta: \prod_{j=0}^n \mc T_{\ge_{j}}\to \mc T_{\ge_n}\,,\quad (X_{j})_{j=0,\ldots,n}\mapsto \sum_{j=0}^n X_j\,\]
induces the following diagram whose  rows and columns are all  exact:
\[
\xymatrix{
\prod_{j=0}^{n+1}\mc T_{D_j} \ar[r]^{\phantom{aaa}\partial^0_{\K}\phantom{aaa}}\ar[d]_{\alpha}^{\phantom{\alpha}} & \prod_{j=0}^n \mc T_{\ge_{j}}\ar[r]^{\phantom{aaa}}\ar[d]^{\beta}_{\phantom{\beta}} & H^{1}(\K, \mc T)\ar[r]\ar[d]^{\wt\beta}_{\phantom{\wt\beta}} &0\\
 \mc T_{D_{0}}\times\mc T_{D_{n+1}}\ar[r]^{\sigma}\ar[d]&\mc T_{\ge_{n}}\ar[r]\ar[d]&
 \mc T_{\ge_n}/( \mc T_{D_{0}}+\mc T_{D_{n+1}})
 \ar[r]\ar[d]&0\\
 0&0&0
}
\]
with $\alpha((X_j)_{j=0,\ldots,n+1}):=(X_0,X_{n+1})$ and $\sigma(X_0,X_{n+1}):=X_{n+1}-X_0$.
Since the dimension of $\mc T_{D_0}$ is finite and the codimension of $\mc T_{D_{n+1}}$ in $\mc T_{\ge_{n}}$ is infinite according to (\ref{codim-infini}),
we deduce that  the dimension of $\mc T_{\ge_n}/( \mc T_{D_{0}}+\mc T_{D_{n+1}})$ is infinite and consequently $\dim_\C\,H^{1}(\K, \mc T)=+\infty$. 

Case (2) can be treated as case (1).
In case (3), if $\K$ does not contain a subgraph of type (1) nor (2), even by renumbering, then the configuration must be
\[ \stackrel{D_0}{\bullet} \; \stackrel{\ge_0}{\longrightarrow}\; \stackrel{D_1}{\star} \;
\;\stackrel{\ge_1}{\longleftrightarrow}\cdots \stackrel{\ge_{n-1}}{\longleftrightarrow}
\;\stackrel{D_n}{\star}\;
\stackrel{\ge_n}{\longleftarrow}\; \;\stackrel{D_{n+1}}{\bullet}\]
and we can make again the identifications (\ref{identif2}). The spaces $ \mc T_{D_{0}}$ and $\mc T_{D_{n+1}}$ having both finite dimension,  we obtain the conclusion. Case (4) is trivial because $\T_{D_0}$ and $\T_{D_1}$ have finite dimension and $\dim_\C\T_{\ge_0}=\infty$.
\end{proof}

\begin{proof}[Proof of Theorem~\ref{characterization}] 
We will use the characterization of finite type foliations given in Corollary~\ref{codiminfiny2}.
Notice that the red part 
$\AR^{\alpha}_{\F}$ of a cut-component $\A^\alpha_{\F}$  is not repulsive with respect to $\T_{\F}$ if and only if it contains a geodesic of type (1) or (2) because the configuration  $\bullet\longleftarrow \star$ cannot occur.
On the other hand,
$\AR^\alpha_{\F}$ is not connected if and only if it contains a geodesic of type (3) or (4). It follows from Lemma~\ref{39mai}  that if $\F$ is not of finite type then a cut-component $\A^\alpha_{\F}$ contains a geodesic $\K$ with $\dim H^1(\K,\T_{\F})=\infty$ and consequently $\dim_\C H^1(\A_\F,\T_{\F})\geq \dim_\C H^1(\A^\alpha_\F,\T_{\F})=\infty$, cf. Remark~\ref{sousgraphedim}.

Conversely, if $\F$ has finite type, from Remark~\ref{H1TRT}, Proposition~\ref{RemquchTF} and Theorem~\ref{teobasegeom} we deduce that $H^1(\A_\F,\T_{\F})$ has finite dimension. 
\end{proof}

\bigskip

\section{$\mc C^{\mr{ex}}$-universal deformations}\label{secCexUnivDef}

\subsection{$\mc C^{\mr{ex}}$-universality}  We will show the existence of a \Cex-universal deformation for finite type foliations through the representability of the corresponding  deformation functor. 

\begin{defin}\label{defdefuniv} Let 
 $\F_{Q^\point}$ be an equisingular deformation over a germ of manifold 
$Q^\point:=(Q,u_0)$,  
of a   foliation $\F$. We say that $\F_{Q^\point}$ is a  \emph{$\mc C^{\mr{ex}}$-universal deformation of $\F$} if for any germ of manifold $P^\point=(P,t_0)$ and any equisingular deformation $\G_{P^\point}$ of $\F$ over $P^\point$, there exists a unique germ of holomorphic map 
$\lambda : P^{\point}\to Q^{\point}$ such that the deformations $\G_{P^\point}$ and
 $\lambda^\ast \F_{Q^\point}$  of $\F$ are $\mc C^{\mr{ex}}$-conjugated.
\end{defin} 
\begin{obs}\label{rempreliminaires}
Notice that if $\mu :Q{'}^\point\to Q^\point$ is a germ of biholomorphism, the $\mc C^\mr{ex}$-universality of $\F_{Q^\point}$ and of $\mu^\ast\F_{Q^\point}$ are clearly equivalent. 
On the other hand, it directly results from the definition that the  $\mc C^{\mr{ex}}$-universality of $\F_{Q^\point}$ only depends on its class $\mf f_{Q^\point}:=[\F_{Q^\point}]\in  \mr{Def}_{\F}^{Q^\point}$. We will  then say that \emph{$\mf f_{Q^\point}$ is $\mc C^{\mr{ex}}$-universal}. 
\end{obs}
Let us consider  the maps 
\begin{equation*}\label{lambdaPF}
\Lambda_{{\mf f}_{Q^\point}}^{P^\point} : \mc O(P^\point, Q^\point)\to \mr{Def}_{\F}^{P^\point}\,,
\quad
\lambda\mapsto [\lambda^\ast \F_{Q^\point}]\,,
\end{equation*}
where $ \mc O(P^\point, Q^\point)$ always denotes the set of holomorphic map germs  $P^\point\to Q$ sending $t_0$ to $u_0$. 
By definition we have:
\begin{equation*}\label{cnsuniversality}
{\mf f}_{Q^\point} \hbox{ is $\mc C^{\mr{ex}}$-universal }\Longleftrightarrow 
\hbox{ for any $P^\point$ the map } \Lambda_{{\mf f}_{Q^\point}}^{P^\point} \hbox{ is bijective. }
\end{equation*}
One easily checks that  $(\Lambda_{{\mf f}_{Q^\point}}^{P^\point})_{P^\point}$ defines a natural transformation
\[
\Lambda_{{\mf f}_{Q^\point}} : \mr{F}_{Q^\point}: \iso \mr{Def}_{\F}
\] 
where $\mr F_{Q^\point}$, $\mr{Def}_{\F} : \mbf{Man^\point}\to \mbf{Set}^\point$ are the following contravariant  functors:
\[
\mr{F}_{Q^\point}(P^\point):=\mc O(P^\point, Q^\point),
\quad
\mr{F}_{Q^\point}(\lambda)=\cdot\, \circ\lambda\,,
\quad
\mr{Def_{\F}}(P^\point):= \mr{Def}_{\F}^{P^\point},
\quad \mr{Def}_{\F}(\lambda) := \lambda^\ast\,,
\]
where the first set is pointed by the constant map $\kappa_{u_0}:P^\point\to Q^\point$ and the second one is pointed by the class of the constant deformation $\F_{Q^\point}^\mr{ct}$, see Section~\ref{subsecDefSpaFunct}.
Thus  ${\mf f}_{Q^\point}$ is $\mc C^{\mr{ex}}$-universal if and only if $\Lambda_{\mf f_{Q^\point}}$ is an isomorphism of functors. Classically  $Q^\point$ being fixed,  any isomorphism of functors
\[
\Lambda : \mr F_{Q^\point}\iso \mr{Def_{\F}}
\,,\qquad
\Lambda=(\Lambda^{P^\point}: \mc O(P^\point, Q^\point)\iso \mr{Def}_{\F}^{P^\point})_{P^\point}
\]
  is of this type: 
  \[
  \Lambda=\Lambda_{{\mf f}_{Q^\point}}\quad \hbox{ with }\quad {\mf f}_{Q^\point}:=\Lambda^{Q^\point}(\mr{id}_{Q^\point})\,.
  \]
 It is  Yoneda's Lemma which may be summarized in the  diagrams below whose commutativity results from the functoriality of $\Lambda$:
\[
\xymatrix{
\mc O(Q^\point,Q^\point)\ar[r]^{\Lambda^{Q^\point}}\ar[d]_{\cdot\,\circ\lambda}& 
\mr{Def}_{\F}({Q^\point}) \ar[d]^{\lambda^\ast}\\
\mc O(P^\point,Q^\point)\ar[r]^{\Lambda^{P^\point}} & \mr{Def_{\F}}(P^\point)
}
\qquad\qquad
\xymatrix{
 \mr{id}_{Q^\point}\ar@{|->}[r]^{}\ar@{|->}[d]&{\mf f}_{Q^\point}\ar@{|->}[d]\\ 
\lambda\ar@{|->}[r]^{} & \Lambda^{P^\point}(\lambda)=\lambda^\ast{\mf f}_{Q^\point}\hspace{4,5em}
}
\]
Finally, to find a germ of manifold $Q^\point$ and a $\mc C^{\mr {ex}}$-universal deformation $\F_{Q^\point}$ is equivalent to \emph{represent the functor} $\mr{Def}_{\F}$, i.e. to find a germ of manifold $Q^\point$ and an isomorphism of functors 
$\mr{Def_{\F}}\iso \mr{F}_{Q^\point}\,$:
\begin{equation}\label{univid}
\left(\mf f_{Q^\point}\in \mr{Def}_{\F}^{Q^\point} \hbox{ is $\mc C^{\mr{ex}}$-universal }\right)\Longleftrightarrow 
\left(\exists \;\xi^{Q^\point}:\mr{Def_{\F}}\iso\, \mr{F}_{Q^\point} \,,\;\; 
{ \xi}^{Q^\point}(\mf f_{Q^\point})=\mr{id}_{Q^\point}\right)\,.
\end{equation}
As we will also  need later the naturality of $\xi^{Q^\point}$  relative to the foliation $\F\in \fol$,  we will prove a slightly stronger result.
\\

If $\phi:\G\to\F$ is a $\mc C^{\mr {ex}}$-conjugacy between two foliations $\G$ and $\F$, we will denote by
\begin{equation}\label{H1phi*}
[\phi^*]:=H^1(\phi^*):H^1(\A_\F,\T_\F)\iso H^1(\A_\G,\T_\G),
\end{equation} 
the morphism induced by the vector space-graph isomorphism $\phi^\ast : \mc T_{\F}\iso \mc T_{\G}$ defined in~(\ref{phisastT}).  
We define  the  contravariant \emph{factorizing functor} 
$\mr{Fac}: \mbf{Man^\point}\times\fol
\to
\mbf{Set}^\point
$  as 
\begin{equation*}\label{FonctFac}
\mr{Fac}(Q^\point,\F):=\mc{O}(Q^\point, H^1(\A_{\F}, \mc T_{\F})^\point)\,,
\end{equation*}
this set being 
pointed by the zero map,
and if $(\mu,\phi):(P^\point , \G)  \to(Q^\point,\F)$, then $\mr{Fac}(\mu,\phi):=\mr{Fac}_\phi ^{\mu}$ is the following linear map:
 \begin{equation}\label{FonctFacMorph}
 \mr{Fac}_\phi ^{\mu} : \mc{O}(Q^\point, H^1(\A_{\F}, \mc T_{\F})^\point)
\to
 \mc{O}(P^\point, H^1(\A_{\G}, \mc T_{\G})^\point)\,,
\quad
\lambda\mapsto [\phi^\ast]\circ \lambda\circ\mu\,,
 \end{equation}
where $H^1(\A_{\F}, \mc T_{\F})^\point$ is  the vector space $H^1(\A_{\F}, \mc T_{\F})$ pointed by the origin.

\begin{teo}\label{teounivtop} 
For any
 finite type foliation $\F$  which is a generalized curve and
for any germ of manifold $Q^\point$ 
there is  a bijection 
\begin{equation*}
\xi_{\F}^{Q^\point} : \mr{Def}_{\F}^{Q^\point}\iso \mc{O}(Q^\point, H^1(\A_{\F}, \mc T_{\F})^\point)
\end{equation*}
such that the collection $\{\xi_{\F}^{Q^\point}\}_{(Q^\point,\,\F)}$  defines an isomorphism of contravariant functors
\begin{equation}\label{isoDefFac}
\xi : \mr{Def}\iso\mr{Fac}\,, 
\end{equation}
when both functors are restricted to the subcategory  
$\mbf{Man^\point}\times\foltf$ of the category $\mbf{Man^\point}\times\fol$, see Definition~\ref{def-type-fini}.
\end{teo}

\begin{proof}
We successively apply Theorem~\ref{cocycle}, Proposition \ref{isoH1AutSym}, 
Theorem \ref{aut-Raut/fix}, Theorem~\ref{exp}, Lemma \ref{multext}, natural isomorphisms (\ref{H1M}),  (\ref{tensorprodfunctIso}) and (\ref{isoRTRTT}). 
We obtain for any $(Q^\point, \F)\in \mbf{Man^\point}\times\mbf{Fol}_\mbf{ft}$, the following isomorphisms: 
\begin{align}
\mr{Def}_{\F}^{Q^\point}
\stackrel{\mr{Th.} \,\ref{cocycle}}{\iso}&
H^1(\A_{\F},\mr{Aut}_{\F}^{Q^\point})^{\point}
\stackrel{\mr{Prop.} \,\ref{isoH1AutSym}}{\iso}
H^1(\A_{\F},\sym_{\F}^{Q^\point})^{\point}
\stackrel{\mr{Th.}\, \ref{aut-Raut/fix}}{\iso}H^1(\AR_{\F},\mr{R}\sym_{\F}^{Q^\point})^{\point}\nonumber\\ 
\stackrel{\mr{Th.}\, \ref{exp}}{\iso}&
H^1(\AR_{\F},\mr{R}\mc T_{\F}^{Q^\point})^{\point}
\stackrel{(\ref{H1M})
}{\iso}
H^1(\AR_{\F},\mr{R}\mc T_{\F}\otimes_\C \mf{M}_{Q^\point})^{\point}
\stackrel{(\ref{tensorprodfunctIso})}{\iso} H^1(\AR_{\F},\mr{R}\mc T_{\F})^\point\otimes_\C \mf{M}_{Q^\point}&
\nonumber\\
 \stackrel{(\ref{isoRTRTT})}{\iso}&
H^1(\A_{\F},\mc T_{\F})^\point\otimes_\C \mf{M}_{Q^\point}
{\iso}
\mc O(Q^\point, H^1(\A_{\F},\mc T_{\F})^\point)\,,
\label{defxi}
\end{align}
the last natural isomorphism being as usual $(\mf c\otimes a)\mapsto(t\mapsto a(t)\mf c)$. 
Each of them defines in fact a  natural transformation between contravariant functors from $\mbf{Man^\point}\times\mbf{Fol}_\mbf{ft}$ to   $\mbf{Set}^\point$.
The functor isomorphism $\xi$ is defined as the composition of all the isomorphisms in~(\ref{defxi}).
\end{proof}

\begin{teo}\label{defuniv}
For any  foliation of finite type (which is a generalized curve) there exists a $\mc C^{\mr{ex}}$-universal deformation $\F_{Q^\point}$, with base
\[Q^\point= H^1(\A_{\F}, \mc T_{\F})^\point\,,\]
such that for any equisingular deformation $\F_{P^\point}$ of $\F$, we have that $\lambda:=\xi_{\F}^{P^\point}([\F_{P^\point}])$ satisfies $[\lambda^*\F_{Q^\point}]=[\F_{P^\point}]$.
Moreover, $H^1(\A_{\F},\T_{\F})$ is a $\C$-vector space of dimension
the rank of $H_1(\AR_{\F}/(\AR_\F\setminus\mr{supp}(\mr{R}\T_\F)))$.
\end{teo}
\noindent Here $\AR_{\F}/(\AR_\F\setminus\mr{supp}(\mr{R}\T_\F))$ denotes the graph obtained by contracting to a single vertex the complementary of the support of $\mr{R}\T_\F$, which is a subgraph of $\AR_\F$ according to Proposition~\ref{RemquchTF}.

\begin{proof} By (\ref{univid})  with $Q^\point=H^1(\A_\F,\mc T_{\F})^\point$ we can choose for $\F_{Q^\point}$ any element in $(\xi^{Q^\point}_{\F})^{-1}(\mr{id}_{Q^\point})$. 
To obtain the description of $H^1(\A_{\F},\mc T_{\F})$  we use the isomorphism $H^1(\A_\F,\T_\F)\simeq H^1(\AR_\F,\mr{R}\T_\F)$ given by  (\ref{isoRTRTT}) and Proposition~\ref{RemquchTF}. We then apply Theorem~\ref{teobasegeom}
 to each connected component of $\AR_\F$, taking $d=1$ and noting that $a-p=\mr{rk}\,H_1(\AR_{\F}/(\AR_\F\setminus\mr{supp}(\mr{R}\T_\F)))$.
\end{proof}

\subsection{Kodaira-Spencer map}\label{KSSection} This map assigns to each equisingular deformation its associated  ``infinitesimal deformation''. We will define for any  germ of manifold $Q^\point=(Q,u_0)$ and any foliation $\F\in\fol$, a group-graph morphism
\begin{equation*}\label{KSGrGr}
\Theta^{Q^\point}_{\F}
\;:\;
\aut_{\F}^{Q^\point}\to \mc T_{\F}\otimes_\C(\mf M_{Q,u_0}/\mf M_{Q,u_0}^2)\,
\end{equation*}
so that this collection is a natural transformation $\Theta$ between the functor $\mr{Aut}$ considered in (\ref{functorAut})  and the functor $(Q^\point,\F)\mapsto  \mc T_{\F}\otimes_\C(\mf M_{Q,u_0}/\mf M_{Q,u_0}^2)$. The definition of $\Theta^{Q^\point}_{\F,\msf e}$ for
$\msf e:=\langle D,D'\rangle\in \Ed_{\A_\F}$ is based on the following fact: let $(u_1,\ldots,u_q)$  be a centered coordinate system on $Q^\point$ and let us denote by $\mr{pr}_{M_\F}$ the canonical projection $M_\F\times Q\to M_\F$;  if a germ of biholomorphism $\Phi$ at the  point  $(s,u_0)\in M_\F\times Q$, with $\{s\}:=D\cap D'$, leaves invariant the constant deformation $\F^{\mr{ct}\,\sharp}_{Q^\point}$,  then $\frac{\partial \,\mr{pr}_{M_\F}\circ\,\Phi}{\partial u_k}\big|_{u=u_0}$, $k=1,\ldots,q$, are germs of  vector fields in $M_\F$ at $s$, basic for the foliation $\F^\sharp$. We denote by $\left[\frac{\partial  \,\mr{pr}_{M_\F}\circ\,\Phi}{\partial u_k}\big|_{u=u_0}\right]$ its class in $\mc T_{\F}(\msf e)$ 
and, when $s$ is not a nodal singularity of $\F^\sharp$, we set:
\[
\Theta_{\F, \;\msf e}^{Q^\point}\,:\,
\aut_{\F}^{Q^\point}(\msf e)\to \mc T_{\F}(\msf e)\otimes_\C(\mf M_{Q,u_0}/\mf M_{Q,u_0}^2)\,,
\]
\begin{equation}\label{DDuksAute}
\Theta_{\F,\; \msf e}^{Q^\point}(\Phi):= \sum_{k=1}^q \left[\left.\frac{\partial \,\mr{pr}_{M_\F}\circ \Phi}{\partial u_k}\right|_{u=u_0}\right]\otimes\;\dot u_k\,.
\end{equation}
The definition of $\Theta_{\F,\,D}^{Q^\point}$ for $D\in \Ed_{\A_\F}$ invariant is less direct because the homeomorphisms $\Phi\in \aut_{\F}^{Q^\point}(D)$ are not holomorphic a priori. We will fix the germ of a submersion $g:(M_\F,o_D)\to(\C,0)$  at a regular point $o_D\in D$ constant along the leaves of $\F^\sharp$ and
we will use the composition of group morphisms
\[ 
{\aut}_{\F}^{Q^\point}(D)\to\sym_{\F}^{Q^\point}(D)\stackrel{G_D}{\to} C^0_{Q^\point}(H_D),\quad \Phi\mapsto g_*\Phi,\] 
 cf. Proposition~\ref{G}, and the isomorphism 
\[\dot g_{D*} : \mc T_{\F}(D)\iso \mc V(H_D)\] given
by the exact sequence (\ref{FlatUT}) with $U=D$.  
 One easily checks that  if $h(z)$ is a germ of biholomorphism of $(\C,0)$ and  $(\phi(z,u),u)$ is a germ of biholomorphism of $(\C\times Q, (0,u_0))$ over $Q$  satisfying $\phi(z,u_0)=z$ and $\phi(h(z),u)=h(\phi(z,u))$, then $\frac{\partial \phi}{\partial u_k}\big|_{u=u_0}$, $k=1,\ldots,q$, are vector field germs on $(\C,0)$ invariant by  $h$. We set: 
\[
\Theta_{\F, \;D}^{Q^\point}\,:\,
\aut_{\F}^{Q^\point}(D)\to \mc T_{\F}(D)\otimes_\C(\mf M_{Q,u_0}/\mf M_{Q,u_0}^2)\,,
\]
\begin{equation}\label{DDuksAutD}
\Theta_{\F, \;D}^{Q^\point}(\Phi):= \sum_{k=1}^q\;
\dot g_{D*}^{-1}\left(\left.\frac{\partial\, \mr{pr}_\C\circ g_*\Phi}
 {\partial u_k}\right|_{u=u_0}\right)\otimes\;\dot u_k\,,
\end{equation}
where $\mr{pr}_\C$ again denotes the canonical projection $\C\times Q\to \C$.
One can check 
that definitions (\ref{DDuksAute}) and (\ref{DDuksAutD}) do not depend on the choice of the germ of first integral submersion $g$ at some regular point $o_D\in D$ nor on that of the coordinate system on~ $Q^\point$. To see that these group morphisms define a group-graph morphism we need to show that for $\Phi\in \aut_{\F}^{Q^\point}(D)$,  the germ at $\{s\}=D\cap D'$ of   
$ \dot g_{D\ast}^{-1}\left(\frac{\partial\, \mr{pr}_\C\circ g_*\Phi}
 {\partial u_k}\big|_{u=u_0}\right)$ is equal to the class in $\un{\mc T}_{\F}(s)$ of the germ at $s$ of  $\frac{\partial  \,\mr{pr}_{M_\F}\circ\,\Phi}{\partial u_k}\big|_{u=u_0}$, $k=1,\ldots,q$. 
Thanks to Remark~\ref{analytic-continuation} it suffices to check this equality at a regular point $s'\in D$ close to $s$. We may suppose that  $o_D=s'$. Using the map $g_\ast$ in the exact sequence (\ref{flatoD}), the commutativity of the operations  of partial derivatives at $s'$ and direct image by the first integral~$g$:
 \[
g_\ast\left( \left.\frac{\partial \,\mr{pr}_{M_\F}\circ \Phi}{\partial u_k}\right|_{u=u_0}\right)
=
\left.\frac{\partial\, \mr{pr}_\C\circ g_*\Phi}
 {\partial u_k}\right|_{u=u_0}\,,
 \]
 gives the required equality.\\
 
It is easy to check that the collection $\{\Theta_\F^{Q^\point}\}$ defines a natural transformation of functors $\Theta:\aut\to\T\otimes_{\C}\mf M/\mf M^2$.
 Now we apply the cohomological functor to $\Theta$ and we use the natural  identification  between  $\mf M_{Q,u_0}/\mf M_{Q,u_0}^2$ and the cotangent vector space $T^{\ast}_{u_0}Q$ of $Q$ at $u_0$. We obtain natural  maps
 \begin{align}
 &\mr{Def}_{\F}^{Q^\point}\iso
 H^1(\A_{\F}, \aut_{\F}^{Q^\point})\stackrel{H^1(\Theta)}{\longrightarrow} H^1(\A_{\F},\mc T_{\F}\otimes_\C(\mf M_{Q,u_0}/\mf M_{Q,u_0}^2))\iso\nonumber
\\
  &H^1(\A_{\F},\mc T_{\F})\otimes_\C(\mf M_{Q,u_0}/\mf M_{Q,u_0}^2)\
 \iso  H^1(\A_{\F},\mc T_{\F})\otimes_\C T^{\ast}_{u_0}Q
 = L(T_{u_0}Q, H^1(\A_{\F},\mc T_{\F}))\,,
 \label{defKS}
 \end{align}
 where $L(E,E')$ denotes the space of 
 $\C$-linear maps from the $\C$-vector space $E$ to the $\C$-vector space $E'$. 
 We call \emph{Kodaira-Spencer map for $(Q^\point,\F)$} the composition (\ref{defKS}) of these maps: 
\[
\mr{KS}^{Q^\point}_{\F}\;:\; \mr{Def}_{\F}^{Q^\point}\to  L(T_{u_0}Q, H^1(\A_{\F},\mc T_{\F}))\,.
\]
We will also write
\[\mr{KS}^{Q^\point}_{\F}([\F_{Q^\point}]) =:
\left.\frac{\partial[\F_{Q^\point}]}{\partial u}\right|_{u=u_0}\,.
 \]
Consider now the contravariant functor $\mr{DFac}:\mbf{Man^\point}\times\fol\to \mbf{Set}^\point$ defined by 
 \begin{equation*}\label{defDFac}
\mr{DFac}(Q^\point, \F):= L(T_{u_0}Q, H^1(\A_{\F},\mc T_{\F}))\,,
\quad
\mr{DFac}(\mu,\phi) :=\mr{DFac}_{\phi}^{\mu}\,,
\end{equation*}
with $\mr{DFac}_{\phi}^{\mu}:    L(T_{u_0}Q, H^1(\A_{\F},\mc T_{\F}))\to L(T_{t_0}P, H^1(\A_{\G},\mc T_{\G}))
$ defined by 
\[
\mr{DFac}_{\phi}^{\mu}(\ell):= [\phi^{\ast}]\circ\ell\circ D_{t_0} \mu\
\]
if $(\mu,\phi):(P^\point,\G)\to(Q^\point,\F)$ is a morphism in the category $\mbf{Man^\point}\times\fol$.

Since $D_{t_0}([\phi^{\ast}]\circ \lambda\circ \mu) = [\phi^\ast]\circ D_{u_0}\lambda\circ D_{t_0}\mu$ the derivation maps
\[
D^{Q^\point}_{\F} : \mc O(Q^\point,H^1(\A_{\F},\mc T_{\F}))\to  L(T_{u_0}Q, H^1(\A_{\F},\mc T_{\F}))\,,
\quad
\lambda\mapsto D_{u_0}\lambda
\]
constitute a natural transformation 
\begin{equation}\label{Dfunctor}
D:\mr{Fac}\to\mr{DFac}
\end{equation}
according to (\ref{FonctFacMorph}).
One can check the following:
\begin{prop} For any morphism $(\mu,\phi):(P^\point , \G)  \to(Q^\point,\F)$ in $\mbf{Man^\point}\times\fol$ and any deformation $[\F_{Q^\point}]\in\mr{Def}_{\F}^{Q^\point}$, 
we have the following commutative diagram:
\[\xymatrix{T_{t_0}\ar[d]_{D_{t_0}\mu}P\ar[rrrr]^{\frac{\partial((\mu,\phi)^*[\F_{Q^\point}])}{\partial t}\big|_{t=t_0}}&&&&H^1(\A_{\G},\T_{\G})\\
T_{u_0}Q\ar[rrrr]^{\frac{\partial([\F_{Q^\point}])}{\partial u}\big|_{u=u_0}}&&&&H^1(\A_{\F},\T_{\F})\ar[u]_{[\phi^*]}}
\]
in other words,  the collection $\{\mr{KS}_{\F}^{Q^\point}\}_{(Q^\point,\F)}$ defines a natural transformation 
\[\mr{KS}:\mr{Def}\to \mr{DFac}
\]
between contravariant functors from $\mbf{Man^\point}\times\fol$ to $\mbf{Set}^{\point}$.
\end{prop}

\subsection{Criteria for universality} Let us suppose now that the foliation $\F$ has finite type. Using the representation of the deformation functor,  the Kodaira-Spencer transformation becomes the usual derivation:
\begin{prop}\label{KSD}
 Restricted to the subcategory $\mbf{Man^\point}\times\mbf{Fol}_\mbf{ft}$ the natural transformation $\mr{KS}$ is equal to the composition of the  natural transformation derivative (\ref{Dfunctor}) with the natural isomorphism $\xi:\mr{Def}\iso\mr{Fac}$ defined in~(\ref{isoDefFac})
\[
\mr{KS}=D\circ \xi
\]
\end{prop}
\begin{proof}
Let us fix $(Q^\point,\F)\in \mbf{Man^\point}\times\mbf{Fol}_\mbf{ft}$.
Since $\F$ is assumed to be of finite type, $\xi$ is an isomorphism of functors and it suffices to see that, after the identifications 
\[\mc O(Q^\point,H^1(\A_\F,\T_\F))\simeq H^1(\A_\F,\T_\F)\otimes_\C\mf M_{Q,u_0}\] and 
\[L(T_{u_0}Q,H^1(\A_F,\T_\F))\simeq H^1(\A_\F,\T_\F)\otimes_\C\mf M_{Q,u_0}/\mf M^2_{Q,u_0},\] 
 the following map
\[\mr{KS}^{Q^\point}_{\F}\circ(\xi_{\F}^{Q^\point})^{-1}:H^1(\A_{\F},\T_{\F})\otimes_\C\mf M_{Q,u_0}\to H^1(\A_{\F},\T_{\F})\otimes_\C\mf M_{Q,u_0}/\mf M_{Q,u_0}^2\]
coincides with the tensor product of the identity map of $H^1(\A_{\F},\T_{\F})$ and the quotient map $\mf M_{Q,u_0}\to\mf M_{Q,u_0}/\mf M_{Q,u_0}^2$, $a\mapsto \dot a$.
By following the functor morphisms in (\ref{defxi}) and (\ref{defKS}) and formula (\ref{DDuksAutD}) we obtain that
\begin{align*}
(\mr{KS}^{Q^\point}_{\F}\circ(\xi^{Q^\point}_{\F})^{-1})([X_{D,\msf{e}}]\otimes a(u))&=\sum_k\left[\frac{\partial}{\partial u_k}\Big|_{u=u_0}\exp(a(u)X_{D,\msf{e}})[1]\right]\otimes \dot{u}_k\\
&=\sum_k\left[\frac{\partial}{\partial u_k}\Big|_{u=u_0}\exp(X_{D,\msf{e}})[a(u)]\right]\otimes \dot u_k\\
&=\sum_k\left[\frac{\partial a}{\partial u_k}(u_0)X_{D,\msf{e}}\right]\otimes \dot u_k\\
&=[X_{D,\msf e}]\otimes\sum_k \frac{\partial a}{\partial u_k}(u_0)\dot u_k=[X_{D,\msf e}]\otimes\dot a\,.
\end{align*}

\end{proof}
This interpretation of $\mr{KS}$ provides an infinitesimal criterium of universality.
\begin{teo}\label{unicityuniv} 
Let $\F$ be a  finite type foliation which is a generalized curve.
For any equisingular  deformation $\F_{P^\point}$  of $\F$ over a germ of manifold $P^\point$, the following properties are equivalent:
\begin{enumerate}
\item\label{Funiv}  $\F_{P^\point}$ is $\mc C^{\mr{ex}}$-universal,
\item\label{exixtsfaciso} there is a biholomorphism germ $\mu:R^\point\iso P^\point$ such that $\mu^\ast\F_{P^\point}$ is $\mc C^{\mr{ex}}$-universal,
\item\label{anyfaciso} for any biholomorphism germ $\mu:R^\point\iso P^\point$ the deformation  $\mu^\ast\F_{P^\point}$ is $\mc C^{\mr{ex}}$-universal,
\item\label{lambdabij} the map  $\xi^{P^\point}_{\F}([\F_{P^\point}]): P^\point\to H^1(\A_{\F},\mc T_{\F})^\point$ is a biholomorphism germ,
\item\label{DDFiso} the Kodaira-Spencer map $\left.\frac{\partial [\F_{P^\point}]}{\partial t }\right|_{t=t_0} $ is an isomorphism. 
\end{enumerate}
\end{teo}
\begin{proof} The equivalence of the first three assertions  follows  directly from the definition of  $\mc C^{\mr{ex}}$-universality. 

The proof of  
$(\ref{Funiv})\Longrightarrow(\ref{lambdabij})$ is classical\footnote{In fact, in the category whose objects are the classes of equisingular deformations of $\F$ and the morphisms are pull-backs, a class of an equisingular deformation is $\mc C^{\mr{ex}}$-universal if and only if it is a final object. It is well-known that the final objects are canonically isomorphic, i.e. by a unique isomorphism.}: after setting  $Q^\point:=H^1(\A_{\F},\mc T_{\F})$ one considers the class $\mf f_{Q^\point}\in \mr{Def}^{Q^\point}_{\F}$ such that 
$\xi^{Q^\point}_{\F}(\mf f_{Q^\point})=\mr{id}_{Q^\point}$, which is $\mc C^{\mr{ex}}$-universal, according to (\ref{univid}). 
Therefore, the map $\lambda:=\xi_\F^{P^\point}([\F_{P^\point}]):P^\point\to Q^\point$ satisfies $\mf f_{P^\point}:=[\F_{P^\point}]=\lambda^\ast\mf f_{Q^\point}
$.  On the other hand, since  $\mf f_{P^\point}$  is assumed to be   $\mc C^{\mr{ex}}$-universal, there is  $\mu :Q^\point\to P^\point$ such that $\mf f_{Q^\point}=\mu^\ast\mf f_{P^\point}$. The  uniqueness of factorizations and  the relations $\mu^\ast\lambda^\ast \mf f_{Q^\point}=\mf f_{Q^\point}$, $\lambda^\ast\mu^\ast \mf f_{P^\point}=  \mf f_{P^\point}$, give  $\lambda\circ\mu=\mr{id}_{Q^\point}$ and $\mu\circ\lambda=\mr{id}_{P^\point}\,$. 

The implication 
$(\ref{lambdabij})\Rightarrow(\ref{Funiv})
$ is a consequence of
Theorem~\ref{defuniv} and Remark~\ref{rempreliminaires}.

\noindent According to Proposition~\ref{KSD}, $\left.\frac{\partial [\F_{P^\point}]}{\partial t }\right|_{t=t_0}$ is the derivative of the map $\xi^{P^\point}_{\F}([\F_{P^\point}])$, thus the equivalence  $(\ref{lambdabij}) \Longleftrightarrow (\ref{DDFiso})$ is trivial.
\end{proof}

\begin{cor}\label{functFuniv} Let $\phi$ be an \Cex-conjugacy between two foliations  $\F,\G\in\fol$ of finite type, $\phi(\G)=\F$. Then  $\mf f_{Q^\point}\in \mr{Def}_{\F}^{Q^\point}$ is $\mc C^{\mr{ex}}$-universal  if and only if   $\mf g_{Q^\point} := \phi^*(\mf f_{Q^\point})\in \mr{Def}_{\G}^{Q^\point}$ is universal.
\end{cor}

\begin{proof}
Let us suppose $\mf f_{Q^\point}$ $\mc C^{\mr ex}$-universal. 
According to Theorem~\ref{unicityuniv}, 
$\mf g_{Q^\point}$   is $\mc C^{\mr{ex}}$-universal as soon as  $\lambda^\ast\mf g_{Q^\point}=\phi^*(\lambda^\ast\mf f_{Q^\point})$ is $\mc C^{\mr{ex}}$-universal for some biholomorphism germ $\lambda :P^\point\to Q^\point$. Therefore  
 we may suppose  that  $Q^\point:=H^1(\A_{\F},\mc T_{\F})^\point$ and $\mf f_{Q^\point} =(\xi^{Q^\point}_{\F})^{-1}(\mr{id}_{Q^\point})$. Then we set:
\[
P^\point:=H^1(\A_{\G},\mc {T}_{\G})^\point\,,
\quad
\lambda:=[\phi^\ast]^{-1} : P^\point\to Q^\point\,.
\]
Since $\xi$ is a natural transformation we have the following commutative diagram:
\[
\xymatrix{
\mr{Def}_{\F}^{Q^\point}
\ar[d]^{(\lambda,\phi)^*}\ar[rrr]^{{\xi_{\F}^{Q^\point}}}&&&
\mc O\left({Q^\point},\;{Q^\point}\right)
\ar[d]^{\mr{Fac}_{\phi}^{\lambda}}
\\
\mr{Def}_{\G}^{P^\point}
\ar[rrr]^{\xi_{\G}^{P^\point}}&&&
\mc O\left({P^\point},\;{P^\point}\right)
}
\]
We check that $\mr{Fac}_{\phi}^\lambda(\mr{id}_{Q^\point})= \mr{id}_{P^\point}$, hence $\xi^{P^\point}_{\G}(\lambda^\ast \mf g_{Q^\point})=\xi^{P^\point}_{\G}( (\lambda,\phi)^*(\mf f_{Q^\point}))=\mr{id}_{P^\point}$. Thanks  to criterion (\ref{lambdabij}) in Theorem \ref{unicityuniv}, $\lambda^\ast \mf g_{Q^\point}$ is $\mc C^{\mr{ex}}$-universal. 
\end{proof}

\bibliographystyle{plain}

\end{document}